\documentclass{amsart}

\newtheorem{theorem}{Theorem}[section]
\newtheorem{lemma}[theorem]{Lemma}
\newtheorem{corollary}[theorem]{Corollary}
\newtheorem{proposition}[theorem]{Proposition}

\theoremstyle{definition}
\newtheorem{definition}[theorem]{Definition}

\theoremstyle{remark}

\numberwithin{equation}{section}

\begin{document}
\title{Triangular decomposition of right coideal subalgebras}
\author{V.K. Kharchenko}
\address{FES-Cuautitl\'an, Universidad Nacional Aut\'onoma de M\'exico, 
Centro de Investigaciones Te\'oricas, 
Primero de Mayo s/n, Campo 1, Cuautitl\'an Izcalli, 
Edstado de M\'exico, 54768, M\'EXICO}
\email{vlad@servidor.unam.mx}
\dedicatory{To  Susan Montgomery --- famous mathematician and beautiful person}
\thanks{The author was supported by  PACIVE CONS-304, FES-C UNAM, M\'exico.}


\date{}


\begin{abstract} 
\small
Let $\mathfrak g$ be a Kac-Moody algebra.
We show that every homogeneous right coideal subalgebra $U$ of the multiparameter version of the quantized universal 
enveloping algebra  $U_q(\mathfrak{g}),$ $q^m\neq 1$ containing  all group-like elements has a triangular decomposition
$U=U^-\otimes _{{\bf k}[F]} {\bf k}[H] \otimes _{{\bf k}[G]} U^+$, where $U^-$ and $ U^+$
are right coideal subalgebras of negative and positive quantum Borel subalgebras.
However if $ U_1$ and $ U_2$
are arbitrary right coideal subalgebras of respectively positive and negative  quantum Borel subalgebras, 
then the  triangular composition
$ U_2\otimes _{{\bf k}[F]} {\bf k}[H]\otimes _{{\bf k}[G]} U_1$
 is a right coideal but not necessary a subalgebra.
Using a recent combinatorial classification of right coideal subalgebras of the quantum Borel algebra
$U_q^+(\mathfrak{so}_{2n+1}),$
we find a  necessary condition  for the triangular composition
to be a right coideal subalgebra of $U_q(\mathfrak{so}_{2n+1}).$

If $q$ has a finite multiplicative order $t>4,$ similar results remain valid for 
homogeneous right coideal subalgebras of the multiparameter version 
of the small Lusztig quantum groups $u_q({\frak g}),$ $u_q(\frak{ so}_{2n+1}).$
\end{abstract}

\maketitle

\section{Introduction}
It is well-known that the quantized universal enveloping algebras $U_q({\mathfrak g})$ of the
Kac-Moody algebras have so called triangular decomposition.
In this paper we  are studying when a
right coideal subalgebra of $U_q({\mathfrak g})$ 
also has the triangular decomposition. 
In fact the triangular decomposition holds not only
for $U_q({\mathfrak g}),$
but also for a large class of character Hopf algebras $\mathfrak A$ having positive and negative skew-primitive
generators connected by  relations of the type $x_ix_j^--p_{ji}x_j^-x_i=\delta _i^j(1-g_if_i),$
see \cite[Proposition 3.4]{KL}. In Theorem \ref{raz2} we show that a
right coideal subalgebra $U$ of $\mathfrak A$ containing all group-like elements 
has a required triangular decomposition provided that 
$U$ is homogeneous with respect to the degree function $D$ under the identification $D(x^-_i)=-D(x_i).$
Interestingly, if ${\mathfrak A}=U_q({\mathfrak g}),$ $q^t\neq 1$ with $\mathfrak g$ defined by a Cartan matrix of finite type
then each subalgebra containing  all group-like elements is homogeneous with respect to the above degree function,
\cite[Corollary 3.3]{KL}. Hence in Corollary \ref{fin2},  applying a recent   Heckenberger---Schneider theorem, \cite[Theorem 7.3]{HS},
 we see that for a semisimple complex Lie algebra $\mathfrak g$ the quantized universal enveloping algebra 
$U_q(\frak{g}),$ $q^t\neq 1$ has not more then  $|W|^2$ different right coideal subalgebras containing the coradical.
Here $W$ is the Weyl group of $\mathfrak g.$

We should stress that when $U^{\pm }$ run through the sets of right coideal 
subalgebras of the quantum Borel subalgebras,
the triangular composition $ U^-\otimes _{{\bf k}[F]} {\bf k}[H]\otimes _{{\bf k}[G]} U^+$ is a right coideal
but not always a subalgebra. For example, in \cite{KL} there are given the numbers $C_n$ of pairs 
that define right coideal subalgebras of $U_q({\mathfrak g})$ when ${\mathfrak g}={\mathfrak sl}_{n+1}$
is the simple Lie algebra of type $A_n.$ Using these numbers we can find the probabilities $p_n$ for a
pair $U^-, U^+$ to define a right coideal subalgebra of $U_q({\mathfrak g}),$ ${\mathfrak g}={\mathfrak sl}_{n+1}$:
$$
p_2=72.3\% ; \ p_3=43.8\% ;\  p_4=23.4\% ;\  p_5=11.4\% ; \ p_6=5.1\% ;\  p_7=2.2\% .
$$
If $\mathfrak g$ is the simple Lie algebra of type $G_2$ then the probability 
equals $60/144=41.7\% ,$ see B. Pogorelsky \cite{Pog, Pog1}. 

The next goal of the paper is to prove a necessary condition  for two  right coideal subalgebras 
of the quantum Borel subalgebras to define by means of the triangular composition 
a right coideal subalgebra of $U_q({\mathfrak g})$  
(respectively of $u_q(\mathfrak{g})$)
 when ${\mathfrak g }={\mathfrak so}_{2n+1}$ is the simple Lie algebra of type $B_n.$
In the fourth and fifth sections we follow the classification given in \cite{Kh08} 
to  recall  the basic properties of right coideal subalgebras of 
quantum Borel algebras $U_q^{\pm }({\mathfrak so}_{2n+1}).$
In particular we lead out the following ``integrability" condition: if all  partial derivatives of a homogeneous polynomial $f$
in positive generators of an admissible degree belong to a right coideal subalgebra $U\supseteq G$
of  $U_q^{+}({\mathfrak so}_{2n+1})$ then $f$ itself belongs to $U,$ see Corollary \ref{lat2}.

In Section 6 we introduce  the elements $\Phi^{S}(k,m)$ 
defined by the sets $S\subseteq [1,2n]$ and the ordered pairs of indices $1\leq k\leq m\leq 2n,$
see (\ref{dhs}). We display the element  $\Phi ^{S}(k,m)$
schematically as a sequence of black and white points labeled by the numbers
$k-1,$ $k,$ $k+1, \ldots $ $m-1,$ $m,$ where the first point is always white, and
the last one is always black, while an intermediate point labeled by $i$ is black if and only if 
$i\in S:$  
\begin{equation}
 \stackrel{k-1}{\circ } \ \ \stackrel{k}{\circ } \ \ \stackrel{k+1}{\circ } 
\ \ \stackrel{k+2}{\bullet }\ \ \ \stackrel{k+3}{\circ }\ \cdots
\ \ \stackrel{m-2}{\bullet } \ \ \stackrel{m-1}{\circ }\ \ \stackrel{m}{\bullet }\ .
\label{grbi}
\end{equation}
These elements  are very important since every right coideal subalgebra 
$U\supseteq G$ of the  quantum Borel algebra is generated as an algebra 
by $G$ and the elements of that form, see \cite[Corollary 5.7]{Kh08}. 
Moreover $U$ is uniquely defined by its {\it root sequence}
$\theta =(\theta _1,\theta_2,\ldots ,\theta _n).$ The root sequence satisfies 
$0\leq \theta_i\leq 2n-2i+1,$ and each sequence satisfying these conditions is a root sequence
for some $U$. There exists a constructive algorithm that allows one to find the generators
$\Phi ^{S}(k,m)$ if the sequence $\theta $ is given, see \cite[Definition 10.1 and Eq. (10.6)]{Kh08}.
In particular one may construct all schemes (\ref{grbi}) for the generators.

The minimal generators $\Phi ^{S}(k,m)$ (the generators that do not belong to the subalgebra 
generated by the other generators of that form) satisfy important duality relation 
$\Phi ^{S}(k,m)=\alpha \,  \Phi ^{R}(\psi (m),\psi (k)),$ $\alpha \neq 0,$ where by definition $\psi (i)=2n-i+1,$ while
$R$ is the complement of $\{ \psi (s)-1\, |\, s\in S\} $ with respect to the interval  $[\psi (m),\psi (k)),$ see Proposition \ref{xn0}.
In particular to every minimal generator $\Phi ^{S}(k,m)$ correspond two essentially different 
schemes (\ref{grbi}). Respectively, if $\Phi ^{S}(k,m)$ and $\Phi ^{T}_-(i,j)$
are minimal generators for given right coideal subalgebras $U_1\subseteq U_q^+({\mathfrak so}_{2n+1})$ 
and $U_2\subseteq U_q^-({\mathfrak so}_{2n+1})$ then we have 
four different diagrams of the form
\begin{equation}
\begin{matrix}
S:\stackrel{k-1}{\circ } \ & \cdots \ & \stackrel{i-1}{\bullet } 
\ & \stackrel{i}{\bullet }\ \ & \stackrel{i+1}{\circ }\ & \cdots &
\ & \stackrel{m}{\bullet } \ & \ & \stackrel{j}{\cdot } \cr
T:\ \ \ \ \ & \  \ & \circ  
\ & \circ \ \ & \bullet \ & \cdots &
\ & \bullet  \ & \cdots  \ & \bullet
\end{matrix} \ \ .
\label{gt}
\end{equation}
In Theorem \ref{bale} we prove the main result of the paper:  
If the triangular composition $ U_2\otimes _{{\bf k}[F]} {\bf k}[H]\otimes _{{\bf k}[G]} U_1$
is a subalgebra then for each pair of minimal generators one of the following two options is fulfilled:

\noindent
a) no one of the possible four diagrams (\ref{gt}) has fragments of the form 
$$
\begin{matrix}
\stackrel{t}{\circ } \ & \cdots & \stackrel{s}{\bullet } \cr
\circ  
\ & \cdots  & \bullet 
\end{matrix}\ \ ;
$$

\noindent
b) one of the possible four diagrams (\ref{gt}) has the form 
$$
\begin{matrix}
\stackrel{k-1}{\circ } \ & \cdots & \circ & \cdots & \bullet & \cdots & \stackrel{m}{\bullet } \cr
\circ  
\ & \cdots  &  \bullet & \cdots & \circ & \cdots &  \bullet 
\end{matrix}\ \ ,
$$
\noindent
where no one of the intermediate columns has points of the same color.

Certainly $U_q({\mathfrak sl}_n)$ is a Hopf subalgebra of $U_q({\mathfrak so}_{2n+1}).$
If we apply the found condition to right coideal subalgebras of $U_q({\mathfrak sl}_{n}),$
we get precisely the necessary and sufficient condition given in \cite[Theorem 11.1]{KL}.
Hence we have a reason to believe that the found necessary condition is also sufficient for the
triangular composition to define a right coideal subalgebra of $U_q({\mathfrak so}_{2n+1}).$

\smallskip
\smallskip

Finally we would like to stress that right coideal subalgebras that do not admit the triangular 
decomposition (inhomogeneous or not including the coradical) are also of interest due to their relations with  
quantum symmetric pairs, quantum Harish-Chandra modules, and quantum symmetric spaces.
Many of the (left) coideal subalgebras studied by M. Noumi and G. Letzter, see the survey  \cite{Let},
do not admit a triangular decomposition.

\section{Bracket technique}

Let $X=$ $\{ x_1, x_2,\ldots, x_n\} $ be {\it quantum variables}; that is, associated with each letter
$x_i$ are an element $g_i$ of a fixed Abelian group $G$ and a 
character $\chi ^i:G\rightarrow {\bf k}^*.$ 
For every word $w$ in $X$ let $g_w$ or gr$(w)$ denote
an element of $G$ that appears from $w$ by replacing each $x_i$ with $g_i.$
In the same way  $\chi ^w$ denotes a character that appears from $w$
by replacing each $x_i$ with $\chi ^i.$

Let $G\langle X\rangle $ denote the skew group algebra generated by $G$
and {\bf k}$\langle X\rangle $ with the commutation rules $x_ig=\chi ^i(g)gx_i,$
or equivalently $wg=\chi ^w(g)gw,$ where $w$ is an arbitrary word in $X.$
Certainly $G\langle X\rangle $ is spanned by the products $gw,$ where $g\in G$
and $w$ runs through the set of  words in $X.$

The algebra $G\langle X\rangle $ has natural gradings by the group $G$ and by 
the group $G^*$ of characters. More precisely the basis element $gw$ belongs to the 
$g {\rm gr}(w)$-homogeneous component with repect to the grading by $G$ and 
it belongs to the $\chi ^w$-homogeneous component with respect to the grading by $G^*.$ 

Let $u$ be a homogeneous element with respect to the grading by $G^*,$ and 
$v$ be a homogeneous element with respect to the grading by $G.$
We define a  skew commutator  by the formula 
\begin{equation}
[u,v]=uv-\chi ^u(g_v) vu,
\label{sqo}
\end{equation}
where $u$ belongs to the $\chi ^u$-homogeneous component, while
$v$ belongs to the $g_v$-homogeneous component.  
Sometimes for short we use the notation  $\chi ^u(g_v)=p_{uv}=p(u,v).$
Of course $p(u,v)$ is a bimultiplicative map: 
\begin{equation}
p(u,v)p(u,t)=p(u, vt), \ \ p(u,v)p(t,v)=p(ut,v).
\label{sqot}
\end{equation}
In particular the form $p(-,-)$ is completely defined by the {\it quantification matrix} $||p_{ij}||,$
where $p_{ij}=\chi ^{i}(g_{j}).$

The brackets satisfy the following Jacobi identities for homogeneous 
(with respect to the both gradings) elements:
\begin{equation}
[[u, v],w]=[u,[v,w]]+p_{wv}^{-1}[[u,w],v]+(p_{vw}-p_{wv}^{-1})[u,w]\cdot v.
\label{jak1}
\end{equation}
\begin{equation}
[[u, v],w]=[u,[v,w]]-p_{vu}^{-1}[v,[u,w]]+(p_{vu}^{-1}-p_{uv})v\cdot [u,w].
\label{ja}
\end{equation}
These identities can be easily verified by direct computations using (\ref{sqo}), (\ref{sqot}).
In particular the following conditional  identities are valid (both in $G \langle X\rangle $ and in all
of its homomorphic images)
\begin{equation}
[[u, v],w]=[u,[v,w]],\hbox{ provided that } [u,w]=0.
\label{jak3}
\end{equation}
\begin{equation}
[u,[v,w]]=p_{uv}[v,[u,w]],\hbox{ provided that } [u,v]=0 \hbox{ and }p_{uv}p_{vu}=1.
\label{jak4}
\end{equation}
By an evident induction on the length  these conditional identities admit  the following  generalization,
see \cite[Lemma 2.2]{KL}.
\begin{lemma}
Let $y_1,$ $y_2,$ $\ldots ,$ $y_m$ be linear combinations of words homogeneous in each $x_k\in X.$ If
$[y_i,y_j]=0,$ $1\leq i<j-1<m,$ then  the bracketed polynomial $[y_1y_2\ldots y_m]$ is independent 
of the precise alignment of brackets:
\begin{equation}
[y_1y_2\ldots y_m]=[[y_1y_2\ldots y_s],[y_{s+1}y_{s+2}\ldots y_m]], \ 1\leq s<m.
\label{ind}
\end{equation}
\label{indle}
\end{lemma}
The brackets are related to the product by the following ad-identities
\begin{equation}
[u\cdot v,w]=p_{vw}[u,w]\cdot v+u\cdot [v,w], 
\label{br1f}
\end{equation}
\begin{equation}
[u,v\cdot w]=[u,v]\cdot w+p_{uv}v\cdot [u,w].
\label{br1}
\end{equation}
In particular, if $[u,w]=0,$ we have
\begin{equation}
[u\cdot v,w]=u\cdot [v,w].
\label{br2}
\end{equation}
The antisymmetry identity takes the form
\begin{equation}
[u,v]=-p_{uv}[v,u] \ \ \hbox{ provided that } \ \ p_{uv}p_{vu}=1.
\label{bri}
\end{equation}
Further we have
\begin{equation} 
[u, gv]=u\cdot gv-\chi ^u(gg_v)gv\cdot u=\chi ^u(g)\, g[u,v], \ \ \ g\in G;
\label{cuq1}
\end{equation}

\begin{equation} 
[gu, v]=gu\cdot v-\chi ^u(g_v)v\cdot gu=g(uv-p_{uv}\chi ^v(g)\, vu),
\label{cuq}
\end{equation}
or in a bracket form
\begin{equation} 
[gu,v]=g[u,v]+p_{uv}(1-\chi^v(g))\, g\, v\cdot u,  \ \ \ g\in G.
\label{cuq2}
\end{equation}
\begin{equation} 
[gu,v]=\chi^v(g)\, g[u,v]+(1-\chi^v(g))\, g\, u\cdot v,  \ \ \ g\in G.
\label{cuq21}
\end{equation}

\smallskip
\noindent 
{\bf Quantization of variables.} Let $p_{ij},$ $1\leq i,j\leq n$ be a set of 
parameters, $0\neq p_{ij}\in {\bf k}.$ Let $g_j$ be the linear transformation 
$g_j:x_i\rightarrow p_{ij}x_i$ of the linear space 
spanned by a set of variables $X=\{ x_1, x_2, \ldots , x_n\} .$
Let $\chi ^i$ denote a character $\chi^i :g_j\rightarrow p_{ij}$ of the group $G$ generated by 
$g_i,$ $1\leq i\leq n.$ We may consider each $x_i$ as a quantum variable with parameters $g_i,$ $\chi ^i.$

\smallskip
\noindent 
{\bf Algebra ${\mathfrak F}_n.$} Let $X^-=$ $\{ x^-_1, x^-_2,\ldots , x^-_n\} $
be a new set of variables.  
We consider $X^-$ as a set of quantum variables quantized by the parameters 
$p_{ji}^{-1},$ $1\leq i,j\leq n.$ More precisely  we have an Abelian group $F$ generated 
by elements $f_1, f_2, \ldots ,f_n$ acting on the linear space spanned by 
$X^-$ so that $(x_i^-)^{f_j}=p_{ji}^{-1}x_i^-,$ where $p_{ij}$ are the same parameters 
that define the quantization of the variables $X.$ In this case gr$(x_i^-)=f_i,$ 
$\chi ^{x_i^-}(f_j)=p_{ji}^{-1}.$

We may extend the characters $\chi ^i $ on $G\times F$ in the following way
\begin{equation}
\chi ^i(f_j)\stackrel{df}{=}p_{ji}=\chi ^j(g_i).
\label{shar1}
\end{equation}
Indeed, if $\prod_k f_k^{m_k}=1$ in $F,$ then application to $x_i^-$
implies $\prod_k p_{ki}^{-m_k}=1,$ hence $\chi ^i(\prod _k f_k^{m_k})=\prod p_{ki}^{m_k}=1.$

In the same way we may extend the characters $\chi ^{x_i^-}$ on $G\times F$
so that 
\begin{equation}
\chi ^{x_i^-}=(\chi ^i)^{-1} \ \ \hbox{as characters of } G\times F.
\label{shar2}
\end{equation}

In what follows  $H$ denotes a quotient group $(G\times F)/N,$ where 
$N$ is an arbitrary subgroup with $\chi ^{i}(N)=1,$ $1\leq i\leq n.$ For example, if the quantification parameters satisfy additional symmetry conditions $p_{ij}=p_{ji},$  $1\leq i,j\leq n,$
(as this is a case for the original Drinfeld-Jimbo and Lusztig quantifications) then 
$\chi ^i(g_k^{-1}f_k)=p_{ik}^{-1}p_{ki}=1,$ and we may take $N$ to be the subgroup generated by 
$g_k^{-1}f_k,$  $1\leq k\leq n. $ In this particular case the groups 
$H,$ $G,$ $F$ may be identified. 

In the general case without loss of generality we may suppose that $G,F\subseteq H.$
Certainly $\chi ^i, 1\leq i\leq n$ are characters of $H$ and $H$ still 
acts on the  space spanned by $X\cup X^-$ by means of these characters and their inverses.

We define the algebra ${\mathfrak F}_n$ as a quotient of $H\langle X\cup X^-\rangle $
by the following relations
\begin{equation}
[x_i, x_j^-]=\delta_i^j(1-g_if_i), \ \ \ \ \ 1\leq i,j\leq n,
\label{rela3}
\end{equation}
where the brackets are defined on $H\langle X\cup X^-\rangle $ 
by the above quantization of the  variables $X\cup X^-$; 
that is, $[x_i,x_j^-]=x_ix_j^--p_{ji}x_j^-x_i,$ for $\chi ^{i}(f_j)=p_{ji}.$ 

We go ahead with a number of useful notes for calculation of the skew commutators in 
${\mathfrak F}_n.$
 If $u$ is a word in $X,$ then $u^-$ denotes
a word in $X^-$ that appears from $u$ under the substitution $x_i\leftarrow x_i^-.$
We have $p(v,w^-)=\chi ^v(f_w)=p(w,v),$ while $p(w^-,v)=(\chi ^w)^{-1}(g_v)=p(w,v)^{-1}.$
Thus  $p(v,w^-)p(w^-,v)=1.$ Therefore the Jacobi and antisymmetry identities
(see, (\ref{jak1}), (\ref{bri})) take up their original ``colored" form:
\begin{equation} 
[[u,v],w^-]=[u,[v,w^-]]+p_{wv}[[u,w^-],v];
\label{uno}
\end{equation}
\begin{equation} 
[u^-,w]=-p_{uw}^{-1}[w,u^-].
\label{dos}
\end{equation}
In the same way
\begin{equation} 
[[u^-,v^-],w]=[u^-,[v^-,w]]+p_{vw}^{-1}[[u^-,w],v^-].
\label{tres}
\end{equation}
Using  (\ref{ja}) we have also
\begin{equation} 
[u,[v^-,w^-]]=[[u,v^-],w^-]+p_{vu}[v^-,[u,w^-]].
\label{cua}
\end{equation}
If we put $w^-\leftarrow [w^-,t^-]$ in (\ref{uno}) we have
$$
[[u,v],[w^-,t^-]]=[u,[v, [w^-,t^-]]]+p_{wt,v}[[u,[w^-,t^-]],v].
$$
Using (\ref{cua}) we get
$$
[[u,v],[w^-,t^-]]= \hbox{\big [}u,[[v,w^-],t^-]\hbox{\big ]}+p_{w,v}\hbox{\big [}u,[w^-,[v,t^-]]\hbox{\big ]}
$$
\begin{equation} 
+p_{wt,v}[[u,[w^-,t^-]],v].
\label{fo}
\end{equation}
%(
Using once more (\ref{cua}) we get
$$
[[u,v],[w^-,t^-]]= \hbox{\big [}u,[[v,w^-],t^-]\hbox{\big ]}+p_{w,v}\hbox{\big [}u,[w^-,[v,t^-]]\hbox{\big ]}
$$
\begin{equation} 
+p_{wt,v}\hbox{\big [}[[u,w^-],t^-],v\hbox{\big ]}+p_{wt,v}p_{w,u}\hbox{\big [}[w^-,[u,t^-]],v\hbox{\big ]}.
\label{fo1}
\end{equation}

We must stress that relations (\ref{rela3}) are homogeneous with respect to the grading by the
character group $H^*,$ but they are not homogeneous with respect to the grading by $H.$
Therefore once we apply relations (\ref{rela3}), or other ``inhomogeneous in $H$" relations,
we have to develop the bracket to its explicit form as soon as
the inhomogeneous substitution applies to the right factor of the bracket.
For example we have
\begin{equation}
[u,[x_i,x_i^-]]=u(1-g_if_i)-\chi ^u(g_if_i)(1-g_if_i)u=(1-\chi^u(g_if_i))u,
\label{cuq3}
\end{equation}
but not $[u,[x_i,x_i^-]]=[u,1-g_if_i]=[u,1]-[u,g_if_i]=0.$ 
In fact here the bracket $[u,1-g_if_i]$ is undefined since the right factor
$1-g_if_i$ is inhomogeneous in $H$ (unless $g_if_i=1$).
At the same time
\begin{equation}
[[x_i,x_i^-],u]=(1-g_if_i)u-u(1-g_if_i)=(\chi^u(g_if_i)-1)\, g_if_i\cdot u,
\label{cuq4}
\end{equation}
and $[[x_i,x_i^-],u]=[1-g_if_i, u]$ $=[1,u]-[g_if_i,u]$ is valid since the inhomogeneous 
substitution has been applied to the left factor in the brackets. 
\begin{lemma} Let $X_1,$ $X_2$ be subsets of $X.$ Suppose that $u$ is a word
in $X_1$ and $v$ is a word in $X_2.$ If $X_1\cap X_2=\emptyset ,$ then
in the algebra ${\mathfrak F}_n$ we have $[u,v^-]=0.$
\label{suu}
\end{lemma}
\begin{proof}
Defining relations (\ref{rela3}) imply $[x_i,x_j^-]=0,$ $x_i\in X_1,$ $x_j\in X_2.$
Ad-identities (\ref{br1f}) and (\ref{br1}) with evident induction prove the statement.
\end{proof}
\begin{lemma}
In the algebra ${\mathfrak F}_n$ for any pair $(i,j)$ with $1\leq i,j\leq n,$ $i\neq j$ we have
$$
\hbox{\rm \Large [}[x_i,x_j],[x_j^-,x_i^-]\hbox{\rm \Large ]}=(1-p_{ij}p_{ji}) (1-g_ig_jf_if_j).
$$
\label{suu1}
\end{lemma}
\begin{proof} Without loss of generality we may assume $i=1,$ $j=2.$
Since $[x_1,x_2^-]=[x_2,x_1^-]=0,$ identity (\ref{fo1}) implies
\begin{equation}
[[x_1,x_2],[x_2^-,x_1^-]]=[x_1,[[x_2,x_2^-],x_1^-]]]+p(x_2x_1,x_2)p(x_2, x_1)[[x_2^-,[x_1,x_1^-]],x_2].
\label{sh1}
\end{equation}
Using (\ref{cuq4}) and then (\ref{cuq1}) we get
$$
[x_1,[[x_2,x_2^-],x_1^-]]]=((\chi^1)^{-1}(g_2f_2)-1)\chi^1(g_2f_2)g_2f_2[x_1,x_1^-]=
(1-p_{12}p_{21})g_2f_2 (1-g_1f_1).
$$
Taking into account  (\ref{cuq3}), we have
$$
[x_2^-,[x_1,x_1^-]]=(1-(\chi^2)^{-1}(g_1f_1))x^-_2=(1-p_{21}^{-1}p_{12}^{-1})x_2^-.
$$ 
Antisymmetry relation (\ref{dos}) implies $[x_2^-,x_2]=-p_{22}^{-1}[x_2,x_2^-].$ Hence
$$
[[x_2^-,[x_1,x_1^-]],x_2]=(1-p_{21}^{-1}p_{12}^{-1})(-p_{22}^{-1})(1-g_2f_2).
$$
In (\ref{sh1}) we have  $p(x_2x_1,x_2)p(x_2, x_1)=p_{22}p_{12}p_{21}, $
hence
$$
[[x_1,x_2],[x_2^-,x_1^-]]=(1-p_{12}p_{21})(g_2f_2-g_1g_2f_1f_2+1-g_2f_2),
$$
which is required.
\end{proof}
\begin{lemma} In the algebra ${\mathfrak F}_n$ for any pair $(i,j)$ with $1\leq i,j\leq n,$ $i\neq j$ we have
$$
\hbox{\bf \Large [}[[x_i,x_j],x_j],[x_j^-,[x_j^-,x_i^-]]\hbox{\bf \Large ]}=
\varepsilon \, (1-g_ig_j^2f_if_j^2),
$$
where $\varepsilon =(1+p_{jj})(1-p_{ij}p_{ji})(1-p_{ij}p_{ji}p_{jj}).$
\label{suu2}
\end{lemma}
\begin{proof} Again, without loss of generality we may assume $i=1,$ $j=2.$
Let us put $u\leftarrow [x_1,x_2],$ $v\leftarrow x_2,$ $w^-\leftarrow x_2^-, $ 
$t^-\leftarrow [x_2^-,x_1^-]$ in (\ref{fo1}). 
We have $[v,w^-]=[x_2,x_2^-]=1-g_2f_2.$ By means of (\ref{cuq4}) we get
$[[v,w^-],t^-]=(\chi ^{t^-}(g_2f_2)-1)g_2f_2\cdot t^-.$ Here  
$\chi ^{t^-}(g_2f_2)=p_{22}^{-2}p_{12}^{-1}p_{21}^{-1}.$ Using first (\ref{cuq1}) and then 
Lemma \ref{suu1} we get 
\begin{equation}
[u,[[v,w^-],t^-]]=\varepsilon _1\, g_2f_2(1-g_1g_2f_1f_2),
\label{ssh1}
\end{equation}
where 
$$
\varepsilon _1=(p_{22}^{-2}p_{12}^{-1}p_{21}^{-1}-1)\chi ^u(g_2f_2)(1-p_{12}p_{21})=
(1-p_{12}p_{21}p_{22}^2)(1-p_{12}p_{21}).
$$

Further, $[v,t^-]=[x_2,[x_2^-,x_1^-]].$ By (\ref{jak3}) we have 
$[x_2,[x_2^-,x_1^-]]=[[x_2,x_2^-],x_1^-].$ Hence (\ref{cuq4}) implies
$[v,t^-]=((\chi ^1)^{-1}(g_2f_2)-1)g_2f_2\cdot x_1^-.$ By (\ref{cuq1}) we get
$[w^-,[v,t^-]]=p_{22}^{-2}(p_{12}^{-1}p_{21}^{-1}-1)g_2f_2\cdot [x_2^-,x_1^-].$
Using first (\ref{cuq1}) and then Lemma \ref{suu1} we get
\begin{equation}
p_{w,v}[u,[w^-,[v,t^-]]]=\varepsilon _2\, g_2f_2(1-g_1g_2f_1f_2),
\label{ssh2}
\end{equation}
where 
$$
\varepsilon _2=
p_{22}\cdot p_{22}^{-2}(p_{12}^{-1}p_{21}^{-1}-1)\cdot \chi ^u(g_2f_2)\cdot (1-p_{12}p_{21})=
p_{22}(1-p_{12}p_{21})^2.
$$

In the same way $[u,w^-]=[[x_1,x_2],x_2^-]=(1-\chi ^1(g_2f_2))\cdot x_1$ due to (\ref{jak3})
and (\ref{cuq3}). Further $[[u,w^-],t^-]=(1-p_{12}p_{21})[x_1, [x_2^-,x_1^-]].$
Using (\ref{cua}) we have $[x_1, [x_2^-,x_1^-]]=p_{21}[x_2^-,[x_1,x_1^-]].$
Hence (\ref{cuq3}) allows us to find 
$[[u,w^-],t^-]=(1-p_{12}p_{21})p_{21}(1-p_{21}^{-1}p_{12}^{-1})\cdot x_2^-.$
This implies 
$$
[[[u,w^-],t^-],v]=(1-p_{12}p_{21})(p_{21}-p_{12}^{-1})[x_2^-,x_2].
$$
Since $[x_2^-,x_2]=-p_{22}^{-1}[x_2,x_2^-],$ and $[x_2,x_2^-]=1-g_2f_2,$ we get
\begin{equation}
p_{wt,v}[[[u,w^-],t^-],v]=\varepsilon _3\, (1-g_2f_2),
\label{ssh3}
\end{equation}
where 
$$
\varepsilon _3=p_{22}^2p_{12}\cdot (1-p_{12}p_{21})(p_{21}-p_{12}^{-1})\cdot (-p_{22}^{-1})
=\varepsilon _2.
$$

Finally, by Lemma \ref{suu1} we have $[u,t^-]=(1-p_{12}p_{21})(1-g_1g_2f_1f_2).$
If we apply (\ref{cuq3}) with $x_i\leftarrow u,$ $x_i^-\leftarrow t^-,$
then 
$
[w^-,[u,t^-]] =(1-p_{12}p_{21})(1-(\chi ^2)^{-1}(g_1g_2f_1f_2))\cdot x_2^-. 
$
Hence 
$$[[w^-,[u,t^-]],v^-]=(1-p_{12}p_{21})(1-p_{22}^{-2}p_{12}^{-1}p_{21}^{-1})[x_2^-,x_2].$$
Here $[x_2^-,x_2]=-p_{22}^{-1}(1-g_2f_2).$ Since $p_{wt,v}p_{w,u}=p_{22}^2p_{12}p_{21}p_{22},$
we may write
\begin{equation}
p_{wt,v}p_{w,u}[[w^-,[u,t^-]],v^-]=\varepsilon _4\, (1-g_2f_2),
\label{ssh4}
\end{equation}
where 
$$
\varepsilon _4=p_{22}^2p_{12}p_{21}p_{22}\cdot (1-p_{12}p_{21})(1-p_{22}^{-2}p_{12}^{-1}p_{21}^{-1})
\cdot (-p_{22}^{-1})
=\varepsilon _1.
$$
Now we see that the sum of (\ref{ssh1}) and (\ref{ssh4}) equals
$\varepsilon _1\, (1-g_1g_2^2f_1f_2^2),$ while the sum of (\ref{ssh2}) and (\ref{ssh3}) equals
$\varepsilon _2\, (1-g_1g_2^2f_1f_2^2).$ 
It remains to check that $\varepsilon _1+\varepsilon _2=\varepsilon .$
\end{proof}

\smallskip
The algebra ${\mathfrak F}_n$ has a structure of Hopf algebra with the following coproduct:
\begin{equation}
\Delta (x_i)=x_i\otimes 1+g_i\otimes x_i,\ \ \ \Delta (x_i^-)=x_i^-\otimes 1+f_i\otimes x_i^-.
\label{AIcm}
\end{equation}
\begin{equation}
\Delta (g_i)=g_i\otimes g_i,\ \ \ \Delta (f_i)=f_i\otimes f_i.
\label{AIcm1}
\end{equation}
In this case $G\langle X\rangle $ and $F\langle X^-\rangle $ are Hopf subalgebras of 
${\mathfrak F}_n.$

The free algebra ${\bf k}\langle X\rangle $ has a coordinate differential calculus
\begin{equation}
\partial_i(x_j)=\delta _i^j,\ \ \partial _i (uv)=\partial _i(u)\cdot v+\chi ^u(g_i)u\cdot \partial _i(v).
\label{defdif}
\end{equation}
The partial derivatives connect the calculus with the coproduct on $G\langle X\rangle$ via
\begin{equation}
\Delta (u)\equiv u\otimes 1+\sum _ig_i\partial_i(u)\otimes x_i\ \ \ 
(\hbox{mod }G\langle X\rangle \otimes \hbox{\bf k}\langle X\rangle ^{(2)}),
\label{calc}
\end{equation}
for all $u\in {\bf k}\langle X\rangle .$ Here ${\bf k}\langle X\rangle ^{(2)}$ is the ideal
of  ${\bf k}\langle X\rangle $ generated by $x_ix_j,$ $1\leq i,j\leq n.$
Symmetrically the equation
\begin{equation}
\Delta (u)\equiv g_u\otimes u+\sum _ig_ug_i^{-1}x_i\otimes\partial _i^*(u)\ \ \ 
(\hbox{mod }G\langle X\rangle ^{(2)}\otimes \hbox{\bf k}\langle X\rangle )
\label{calcdu}
\end{equation}
defines a dual differential calculus on ${\bf k}\langle X\rangle $
where the partial derivatives satisfy
\begin{equation}
\partial _j^*( x_i)=\delta _i^j,\ \ \partial _i^*(uv)=
\chi ^{i}(g_v) \partial _i^*(u)\cdot v+u\cdot \partial _i^*(v).
\label{difdu}
\end{equation}
Here $G\langle X\rangle ^{(2)}$ is the ideal of $G\langle X\rangle $
generated by $x_ix_j,$ $1\leq i,j\leq n.$

Similarly the algebra {\bf k}$\langle X^-\rangle$ has a pair of differential calculi:
\begin{equation}
\partial_{-i}(x_j^-)=\delta _i^j,\ \ \partial _{-i} (u^-v^-)
=\partial _{-i}(u^-)\cdot v^-+\chi ^{u^-}(f_i)u^-\cdot \partial _{-i}(v^-),
\label{dem}
\end{equation}
\begin{equation}
\partial _{-j}^*( x_i^-)=\delta _i^j,\ \ \partial _{-i}^*(u^-v^-)=
(\chi ^{i}(f_v))^{-1} \partial _{-i}^*(u^-)\cdot v^-+u^-\cdot \partial _{-i}^*(v^-).
\label{dem1}
\end{equation}
These calculi are related to the coproduct by the similar formulae
\begin{equation}
\Delta (u^-)\equiv u^-\otimes 1+\sum _if_i\partial_{-i}(u^-)\otimes x_i^-\ \ \ 
(\hbox{mod }F\langle X^-\rangle \otimes \hbox{\bf k}\langle X^-\rangle ^{(2)}),
\label{calc1}
\end{equation}
\begin{equation}
\Delta (u^-)\equiv f_u\otimes u^-+\sum _if_uf_i^{-1}x_i^-\otimes\partial _{-i}^*(u^-)\ \ \ 
(\hbox{mod }F\langle X^-\rangle ^{(2)}\otimes \hbox{\bf k}\langle X^-\rangle ).
\label{dum2}
\end{equation}
It will be important for us that operators $[x_i,-]$ and $[-,x^-_i]$ defined respectively 
on ${\bf k}\langle X^-\rangle$ and ${\bf k}\langle X\rangle$ have a nice differential form
(see \cite[Remark, page 2586]{KL}):
\begin{equation}
[x_i,u^-]=\partial_{-i}^*(u^-)p(x_i,u^-)p_{ii}^{-1}-g_if_i \partial_{-i}(u^-),\ \ \ u^-\in{\bf k}\langle X^-\rangle ,
\label{sqi3}
\end{equation}
\begin{equation}
[u, x_i^-]= \partial_ i^*(u)-p_{ii}^{-1}p(u,x_i)\partial_i(u)g_if_i, \ \ \ u\in{\bf k}\langle X\rangle .
\label{sqi4}
\end{equation}
These relations are clear if $u=x_j,$ or $u^-=x_j^-$ while ad-identities 
(\ref{br1f}) and (\ref{br1}) with Leibniz rules 
(\ref{defdif}, \ref{difdu}, \ref{dem}, \ref{dem1}) allow one to perform  evident induction.

\smallskip
\noindent
{\bf Quantification of Kac-Moody algebras}.
Let $C=||a_{ij}||$ be a symmetrizable by $D={\rm diag }(d_1, \ldots d_n)$ generalized Cartan matrix, $d_ia_{ij}=d_ja_{ji}.$
Let $\mathfrak g$ be a Kac-Moody algebra defined by $C,$ see \cite{Kac}.
Suppose that the quantification parameters $p_{ij}=p(x_i,x_j)=\chi ^i(g_j)$ are related by
\begin{equation}
p_{ii}=q^{d_i}, \ \ p_{ij}p_{ji}=q^{d_ia_{ij}},\ \ \ 1\leq i,j\leq n. 
\label{KM1}
\end{equation}
As above  $g_j$ denotes a linear transformation $g_j:x_i\rightarrow p_{ij}x_i$ of the linear space 
spanned by a set of variables $X=\{ x_1, x_2, \ldots , x_n\} .$
Let $\chi ^i$ denote a character $\chi^i :g_j\rightarrow p_{ij}$ of the group $G$ generated by 
$g_i,$ $1\leq i\leq n.$ We consider each $x_i$ as a quantum variable with parameters $g_i,$ $\chi ^i.$
Respectively ${\mathfrak F}_n$ is the above defined algebra related to quantum variables $X,$ and
 $X^-=\{ x_1^-, x_2^-, \ldots , x_n^- \} ,$ where by definition gr$({x_i^-})=f_i,$ $\chi ^{x_i^-}=(\chi ^i)^{-1},$
see (\ref{shar2}), (\ref{rela3}). 

In this case the multiparameter quantization $U_q ({\mathfrak g})$ of ${\mathfrak g}$ 
is a quotient of $H\langle X\cup X^-\rangle$ defined by Serre relations 
with the skew brackets in place of the Lie operation:
\begin{equation}
[\ldots [[x_i,\underbrace{x_j],x_j], \ldots ,x_j]}_{1-a_{ji} \hbox{ times}}=0, \ \ 1\leq i\neq j\leq n;
\label{rela1}
\end{equation}
\begin{equation}
[\ldots [[x_i^-,\underbrace{x_j^-],x_j^-], \ldots ,x_j^-]}_{1-a_{ji} \hbox{ times}}=0, \ \ 1\leq i\neq j\leq n;
\label{rela2}
\end{equation}
\begin{equation}
[x_i, x_j^-]=\delta_i^j(1-g_if_i), \ \ \ \ \ 1\leq i,j\leq n,
\label{rela31}
\end{equation}
where the brackets are defined on $H\langle X\cup X^-\rangle $ by (\ref{sqo}). 
Certainly relations (\ref{rela31}) coincide with (\ref{rela3}). Hence  $U_q ({\mathfrak g})$
is a homomorphic image of ${\mathfrak F}_n.$ 
The algebra $U_q ({\mathfrak g})$ has
a structure of Hopf algebra with the coproduct (\ref{AIcm}), (\ref{AIcm1}); that is, the above
homomorphism is a homomorphism of Hopf algebras.

If the multiplicative order $t$ of $q$ is finite then the multiparameter version 
of the small  Lusztig quantum group is defined as the homomorphic image 
of   $U_q(\frak{g})$ subject to additional relations $u=0, u\in {\bf \Lambda },$ $u^-=0, u^-\in {\bf \Lambda }^-,$
where ${\bf \Lambda }$ is the biggest Hopf ideal of $G\langle X\rangle $ that is contained in the ideal
$G\langle X\rangle ^{(2)}$ generated by $x_ix_j,$ $1\leq i,j\leq n. $ Respectively 
${\bf \Lambda }^-$ is the biggest Hopf ideal of $F\langle X^-\rangle $ that is contained in the ideal
$F\langle X^-\rangle ^{(2)}$ generated by $x_i^-x_j^-,$ $1\leq i,j\leq n.$ 

\smallskip
\noindent
{\bf Mirror generators}. Of course there is no essential difference between positive and negative
quantum Borel subalgebras. More precisely, let us put $y_i=p_{ii}^{-1}x_i^-,$ $y_i^-=-x_i.$ Consider
$y_i$ as a quantum variable with parameters $f_i,$ $(\chi ^i)^{-1},$ while $y_i^-$ as a quantum
variable with parameters $g_i,$ $\chi ^i.$  Relations (\ref{rela1} -- \ref{rela31}) are invariant under 
the substitution $x_i\leftarrow p_{ii}^{-1}x_i^-,$ 
$x_i^-\leftarrow -x_i.$ Hence $y_i,$ $y_i^-$ with $H$ generate
 a subalgebra which can be identified with the quantification $U_{q^{-1}}(\frak{g}).$
At the same time this subalgebra coincides with $U_q(\frak{g}).$ In this way one may 
replace positive and negative quantum Borel subalgebras. We shall call the generators
$y_i=p_{ii}^{-1}x_i^-,$ $y_i^-=-x_i$ as {\it mirror generators}.

\smallskip
\noindent
{\bf Antipode}. Recall that the antipode $\sigma $ by definition satisfies
$\sum a^{(1)}\cdot \sigma (a^{(2)}) $ $=\sum \sigma (a^{(1)})\cdot a^{(2)}=\varepsilon (a).$
Hence (\ref{AIcm}) implies $\sigma (x_i)=-g_i^{-1}x_i,$ $\sigma (x_i^-)=-f_i^{-1}x_i^-.$
In particular if $u$ is a word in $X\cup X^-,$ then $g_u\sigma (u)$ is proportional 
to a word in  $X\cup X^-,$ for $\sigma $ is an antiautomorphism: $\sigma (ab)=\sigma (b)\sigma (a).$
Moreover, if $u,v$ are  linear combinations of words homogeneous in each  $y\in X\cup X^- ,$
then we have 
\begin{equation}
g_ug_v\sigma ([u,v])=p_{vu}^{-1}[g_v\sigma (v),g_u\sigma (u)].
\label{ant1}
\end{equation} 
Indeed, the left hand side equals $g_ug_v(\sigma (v)\sigma (u)-p_{uv}\sigma (u)\sigma (v)),$
while the right hand side is $p_{vu}^{-1}g_v\sigma (v)g_u\sigma (u)-g_u\sigma (u)g_v\sigma (v).$
We have $g_v\sigma (v)\cdot g_u=p_{vu}g_u\cdot g_v\sigma (v),$ and 
$g_u\sigma (u)\cdot g_v=p_{uv}g_v\cdot g_u\sigma (u).$
This implies (\ref{ant1}). 

\smallskip
\noindent
{\bf $\Gamma $-Grading and $\Gamma ^+\oplus \Gamma ^-$-filtration}. We are reminded that  {\it constitution} of a word 
$u$ in $H\cup X\cup X^-$ is a family of nonnegative integers $\{ m_y, y\in X\cup X^-\} $ such that $u$
has $m_y$ occurrences of $y.$ 
Let $\Gamma ^+$ denote the free additive  (commutative) monoid generated by $X,$
while by $\Gamma ^-$  the free additive  monoid generated by $X^-.$ Respectively 
$\Gamma ^+\oplus \Gamma ^-$ is the free additive  monoid generated by $X\cup X^-,$
while $\Gamma $ by definition is the free commutative  group generated by 
$X\cup X^-$ with identification $x^-_i=-x_i,$ $1\leq i\leq n.$ 
We fix the following order on $X\cup X^-:$
\begin{equation} 
x_1>x_2>\ldots >x_n>x_1^->x_2^->\ldots >x_n^-.
\label{orr}
\end{equation}
The monoid $\Gamma ^+\oplus \Gamma ^-$ is a completely ordered monoid 
with respect to the order
\begin{equation}
m_1y_{i_1}+m_2y_{i_2}+\ldots +m_ky_{i_k}>
m_1^{\prime }y_{i_1}+m_2^{\prime }y_{i_2}+\ldots +m_k^{\prime }y_{i_k}
\label{ord}
\end{equation}
if the first from the left nonzero number in
$(m_1-m_1^{\prime}, m_2-m_2^{\prime}, \ldots , m_k-m_k^{\prime})$
is positive, where $y_{i_1}>y_{i_2}>\ldots >y_{i_k}$ in $X\cup X^-.$ 

 We associate 
a formal degree $D(u)=\sum _{y\in X\cup X^-}m_yy\in \Gamma ^+\oplus \Gamma ^-$
to  a word $u$ in $X\cup X^-,$ where $\{ m_y, y\in X\cup X^-\} $ is the constitution 
of $u.$ Respectively, if $f=\sum \alpha _iu_i\in H\langle X\cup X^-\rangle ,$ $0\neq \alpha _i\in \, {\bf k}[H]$
is a linear combination of different words,
then $D(f)=\max _i\{ D(u_i)\} .$ This degree function defines a grading by 
$\Gamma ^+\oplus \Gamma ^-$ on $H\langle X\cup X^-\rangle .$ However relations 
(\ref{rela31}), (\ref{rela3}) are not homogeneous with respect to this grading.  Hence neither ${\mathfrak F}_n$  nor $U_q({\mathfrak g}),$
 $u_q({\mathfrak g}),$ are graded by $\Gamma ^+\oplus \Gamma ^-,$ but certainly they have a filtration defined by the induced 
degree function.

Relations (\ref{rela31}), (\ref{rela3}) became homogeneous if we consider 
the degree $D(u)$ as an element of the group $\Gamma $ with identifications
$x_i^-=-x_i.$ Hence ${\mathfrak F}_n,$  $U_q({\mathfrak g}),$ and $u_q({\mathfrak g})$ have
grading by $\Gamma $ (are $\Gamma $-homogeneous).

\section{Triangular decomposition}

It is well-known that there is so called triangular decomposition
\begin{equation}
U_q(\frak{g})= U_q^-(\frak{g})\otimes _{{\bf k}[F]} {\bf k}[H]
\otimes _{{\bf k}[G]}U_q^+(\frak{g}),
\label{tr}
\end{equation}
where $U_q^+ (\frak{g})$ is the positive quantum Borel subalgebra,
the subalgebra generated by $G$ and values of $x_i,$ $1\leq i\leq n, $
while $U_q^- (\frak{g})$ is the negative quantum Borel subalgebra,
 the subalgebra generated by $F$ and values of $x_i^-,$  $1\leq i\leq n.$

The small Lusztig quantum group
has the triangular decomposition also
\begin{equation}
u_q(\frak{g})= u_q^-(\frak{g})\otimes _{{\bf k}[F]} {\bf k}[H]
\otimes _{{\bf k}[G]}u_q^+(\frak{g}).
\label{trs}
\end{equation}

In fact the triangular decomposition holds not only
for the quantizations defined by the quantum Serre relations but also for arbitrary Hopf 
homomorphic images of ${\mathfrak F}_n.$ More precisely we have the following statement.
\begin{theorem} {\rm (\cite[Proposition 3.4]{KL}).} The algebra 
${\frak A}=\langle {\frak{F}_n\, ||\, u_l=0,\, w_t^-=0}\rangle $ has the triangular decomposition
\begin{equation}
{\frak A}=\langle {\frak{F}_n^-\, ||\, w_t^-=0}\rangle \otimes _{{\bf k}[F]} {\bf k}[H] 
\otimes _{{\bf k}[G]} \langle {\frak{F}_n^+\, ||\, u_l=0}\rangle 
\label{trug}
\end{equation}
provided that $\langle {\frak{F}_n^-\, ||\, w_t^-=0}\rangle $ and 
$\langle {\frak{F}_n^+\, ||\, u_l=0}\rangle$ are Hopf algebras, and $u_l,$ $l\in L,$ $w_t^-,$ $t\in T$ are homogeneous
polynomials respectively  in $x_i,$ $1\leq i\leq n$ and $x_i^-,$ $1\leq i\leq n$  of total degree $>1.$
\label{34}
\end{theorem}
Our goal in this section is to find conditions when a
right coideal subalgebra of $\mathfrak A$ has a triangular decomposition. 
\begin{theorem} Let $\mathfrak A$ be the Hopf algebra defined in the above theorem. 
Every $\Gamma $-homogeneous 
right coideal subalgebra $U\supset H$ of $\mathfrak A$ has a decomposition
\begin{equation}
 U=U^-\otimes _{{\bf k}[F]} {\bf k}[H] \otimes _{{\bf k}[G]}U^+,
\label{tru}
\end{equation}
 where $U^-\supset F$ and  $U^+\supset G$ are homogeneous  right coideal subalgebras respectively of
$\langle {\frak{F}_n^-\, ||\, w_t^-=0}\rangle$ and  $\langle {\frak{F}_n^+\, ||\, u_l=0}\rangle .$
\label{raz2}
\end{theorem}
\begin{proof} 
By \cite[Theorem 1.1]{KhT} the algebra $U$ has a PBW-basis 
over the coradical {\bf k}$[H].$  We shall prove that the PBW-basis 
can be constructed in such a way that each PBW-generator for $U$ 
belongs to either positive or negative component of (\ref{trug}). By definition of 
PBW-basis (see, for example \cite[Section 2]{KL}) this  implies the required decomposition 
of $U$.

Recall that the PBW-basis of $U$ is constructed in the following way, see \cite[Section 4]{KhT}.
 First, we fix a PBW-basis of ${\mathfrak A}$ defined by the {\it hard super-letters} \cite{Kh3}.
Due to the triangular decomposition (\ref{trug}) the  PBW-generators for $\mathfrak A$
belong to either ${\mathfrak A}^+=\langle {\frak{F}_n^+\, ||\, u_l=0}\rangle$
or ${\mathfrak A}^-=\langle {\frak{F}_n^-\, ||\, w_t^-=0}\rangle .$ Then, for each 
PBW-generator (hard super-letter) $[u]$ we fix an arbitrary element $c_u\in U$ with minimal possible $s,$ if any,
such that 
\begin{equation}
c_u=[u]^s+\sum \alpha _iW_iR_i+\sum_j\beta_jV_j \in U, \ \ \alpha _i\in {\bf k}, \ \ \beta _j\in {\bf k}[H],
\label{vad10}
\end{equation}
where $W_i$ are basis words in less than $[u]$ super-letters,  $R_i$ are basis words in
greater than or equal to $[u]$ super-letters, $D(W_iR_i)=sD(u),$ $D(V_j)<sD(u).$
Next, Proposition 4.4 \cite{KhT} implies that the set of all chosen $c_u$ form a set of PBW-generators for $U.$
Since $U$ is $\Gamma $-homogeneous, we may choose $c_u$ to be $\Gamma $-homogeneous as well.

We stress that  the leading terms here are defined by the degree function with values in the
additive monoid $\Gamma ^+\oplus \Gamma ^-$ freely generated by $X\cup X^-,$
but not in the group $\Gamma ,$ see the last subsection of  Section 2.
Equality  $D(W_iR_i)=sD(u)$ implies that all $W_iR_i$ in (\ref{vad10}) have the same constitution in 
$X\cup X^-$ as the leading term $[u]^s$ does. Thus all $W_iR_i$'s and 
the leading term $[u]^s$ belong to the same component of the triangular decomposition. Hence
it remains to show that if $c_u$ is $\Gamma $-homogeneous then there are no  terms $V_j.$
In this case all terms $V_j$ have the same $\Gamma $-degree and smaller 
$\Gamma ^+\oplus \Gamma ^-$-degree. We shall prove that this is impossible.

If $[u]\in {\mathfrak A}^-$
then $sD(u)=m_1x_1^-+m_2x_{2}^-+\ldots +m_nx_n^-,$ while
 the $\Gamma ^+\oplus \Gamma ^-$-degree
of $V_j$ should be less than $m_1x_1^-+m_2x_{2}^-+\ldots +m_nx_n^-.$ Hence due to
definitions (\ref{orr}) and (\ref{ord}) we have $V_j\in {\mathfrak A}^-.$
In particular the  $\Gamma $-degree of $V_j$ coincides with the $\Gamma ^+\oplus \Gamma ^-$-degree,
a contradiction.

Suppose that $[u]\in {\mathfrak A}^+.$ In this case $sD(u)=m_1x_1+m_2x_{2}+\ldots +m_nx_n.$
 Let $d=\sum_{i\leq n} s_ix_i+\sum_{i\leq n} r_ix_i^-$
be the $\Gamma ^+\oplus \Gamma ^-$-degree of $V_j.$ Since
$\Gamma $-degree of $V_j$ coincides with $\Gamma $-degree of $[u]^s,$
we have $s_i-r_i=m_i,$ $1\leq i\leq n.$ This implies $s_i=m_i+r_i\geq m_i.$
At the same time definition (\ref{ord}) and the condition $d<sD(u)$ imply 
$s_k<m_k,$ where $k$ is the smallest index such that $s_k\neq m_k.$
Thus $s_k=m_k$ for all $1\leq k\leq n.$ This yields $r_k=0,$ $1\leq k\leq n.$
In particular $\Gamma $-degree of $V_j$ coincides with the $\Gamma ^+\oplus \Gamma ^-$-degree,
again a contradiction.
\end{proof}  
\begin{corollary} Let $\mathfrak g$ be a semisimple complex Lie algebra. 
If $q$ is not a root of 1, then $U_q(\frak{g})$ has at most  
$|W|^2$ different right coideal subalgebras containing the coradical,
where $W$ is the Weyl group of $\mathfrak g$.
\label{fin2}
\end{corollary}
\begin{proof} 
Due to  Heckenberger---Schneider theorem, \cite[Theorem 7.3]{HS},
 each of the quantum Borel subalgebras $U_q^{\pm }(\frak{g})$  
 has exactly $|W|$ different right coideal subalgebras containing the coradical. 
At the same time by \cite[Corollary 3.3]{KL} every subalgebra of $U_q(\frak{g})$
containing $H$ is $\Gamma $-homogeneous. Hence by Theorem \ref{raz2} we have a decomposition (\ref{tru}).
We see that there are just $|W|^2$ options to form the right hand side of  (\ref{tru}).
 \end{proof}

We should stress that when $U^{\pm }$ run through the sets of right coideal 
subalgebras of the quantum Borel subalgebras
the tensor product in the right hand side of  (\ref{tru}) is a right coideal
but not always a subalgebra. 

Our next goal is to state and prove a necessary condition  for two  right coideal subalgebras $U^{+},$ $U^{-}$
of the quantum Borel algebras to define in (\ref{tru}) a right coideal subalgebra of $U_q(\frak{so}_{2n+1})$  
(respectively of $u_q(\frak{so}_{2n+1})$).

\section{Structure of quantum Borel subalgebras of $U_q({\mathfrak so}_{2n+1})$}
In this section we follow \cite{Kh08} to  recall  the basic properties of quantum Borel subalgebras $U_q^{\pm }({\mathfrak so}_{2n+1}).$
In what follows we fix a parameter $q$ such that $q^2\not= \pm 1,$ $q^3\not= 1.$
Let $\sim $ denote the projective equality:
$a\sim b$ if and only if $a=\alpha b,$ where $0\neq \alpha \in {\bf k}.$ 

If $C$ is a Cartan matrix of type $B_n,$ relations (\ref{KM1}) take up the form
\begin{equation}
p_{nn}=q,\,  p_{ii}=q^2, \ \ p_{i\, i+1}p_{i+1\, i}=q^{-2}, \ 1\leq i<n; 
\label{b1rel}
\end{equation}
\begin{equation}
p_{ij}p_{ji}=1,\  j>i+1. 
\label{b1rell}
\end{equation}
Starting with parameters $p_{ij}$ satisfying these relations, we define 
the group $G$ and the character Hopf algebra  $G\langle X\rangle $
as in the above section. In this case
the quantum Borel algebra $U^+_q ({\frak so}_{2n+1})$ is defined as a quotient of 
$G\langle X\rangle $ by the following relations
\begin{equation}
[x_i,[x_i,x_{i+1}]]=0, \ 1\leq i<n; \ \ [x_i,x_j]=0, \ \ j>i+1;
\label{relb}
\end{equation}
\begin{equation}
[[x_i,x_{i+1}],x_{i+1}]=[[[x_{n-1},x_n],x_n],x_n]=0, \ 1\leq i<n-1.
\label{relbl}
\end{equation}
Here we slightly modify  Serre relations (\ref{rela1}) so that the left hand side of each 
relation is a bracketed Lyndon-Shirshov word. It is possible to do due to the following
general relation in ${\bf k}\langle X\rangle ,$ see \cite[Corollary 4.10]{Kh4}:
\begin{equation}
[\ldots [[x_i,
\underbrace{x_j],x_j],\ldots x_j]}_a\sim \underbrace{[x_j,[x_j,\ldots [x_j}_a,x_i]\ldots ]],
\label{rsc}
\end{equation}
provided that $p_{ij}p_{ji}=p_{jj}^{1-a}.$
\begin{definition} \rm
The elements $u,v$ are said to be 
{\it separated} if there exists an index $j,$ $1\leq j\leq n,$
such that either $u\in {\bf k}\langle x_i\ |\ i<j\rangle ,$
$v\in {\bf k}\langle x_i\ |\ i>j\rangle $ or vice versa 
$u\in {\bf k}\langle x_i\ |\ i>j\rangle ,$
$v\in {\bf k}\langle x_i\ |\ i<j\rangle .$
\label{sep}
\end{definition}
\begin{lemma}
In the algebra $U^+_q ({\frak so}_{2n+1})$ every two separated
homogeneous in each $x_i\in X$ elements $u,v$  $($skew$)$commute, $[u,v]=0,$
in particular $u\cdot v\sim v\cdot u.$
\label{sepp}
\end{lemma}  
\begin{proof}
The statement follows from the second group of defining relations (\ref{relb})
 due to (\ref{br1f}), (\ref{br1}).
\end{proof}
\begin{definition} \rm 
In what follows  $x_i,$ $n<i\leq 2n$ denotes the generator $x_{2n-i+1}.$
Respectively, $u(k,m),$ $1\leq k\leq m\leq 2n$ is the word
$x_kx_{k+1}\cdots x_{m-1}x_m.$  If $1\leq i\leq 2n,$ then  $\psi (i)$
is the number $2n-i+1,$ so that $x_i=x_{\psi (i)}.$ 
We shall frequently use the following properties of $\psi :$
if $i<j,$ then $\psi (i)>\psi (j);$ $\psi (\psi (i))=i;$ $\psi (i+1)=\psi (i)-1,$ $\psi (i-1)=\psi (i)+1.$
\label{fis}
\end{definition}
\begin{definition} \rm 
If $k\leq i<m\leq 2n,$ then we define
\begin{equation}
\sigma _k^m\stackrel{df}{=}p(u(k,m),u(k,m)),
\label{mu11}
\end{equation}
\begin{equation}
\mu _k^{m,i}\stackrel{df}{=}p(u(k,i),u(i+1,m))\cdot p(u(i+1,m),u(k,i)).
\label{mu1}
\end{equation}
\label{slo}
\end{definition}
Of course, one can easily find the $\sigma $'s  and the $\mu$'s by means of (\ref{b1rel}), (\ref{b1rell}).
More precisely, by \cite[Eq. (3.10)]{Kh08} we have 
\begin{equation}
\sigma_k^m =\left\{ \begin{matrix}
q,\hfill &\hbox{if } m=n, \hbox{ or } k=n+1;\hfill \cr
q^4,\hfill &\hbox{if }  m=\psi (k);\hfill \cr
q^{2},\hfill &\hbox{otherwise}.\hfill 
                             \end{matrix}
                    \right.
\label{mu21}
\end{equation}
If $m<\psi (k),$  then by \cite[Eq. (3.13)]{Kh08} we have
\begin{equation}
\mu_k^{m,i} =\left\{ \begin{matrix}
q^{-4},\hfill &\hbox{if } m>n, \, i=\psi (m)-1;\hfill \cr
1,\hfill &\hbox{if }  i=n;\hfill \cr
q^{-2},\hfill &\hbox{otherwise}.\hfill 
                             \end{matrix}
                    \right.
\label{mu2}
\end{equation}
If $m=\psi (k),$ then by \cite[Eq. (3.14)]{Kh08} we have
\begin{equation}
\mu_k^{m,i} =\left\{ \begin{matrix}
q^2,\hfill &\hbox{if }  i=n;\hfill \cr
1,\hfill &\hbox{otherwise}.\hfill 
                             \end{matrix}
                    \right.
\label{mu3}
\end{equation}
If $m>\psi (k),$ then then by \cite[Eq. (3.15)]{Kh08} we have
\begin{equation}
\mu_k^{m,i} =\left\{ \begin{matrix}
q^{-4},\hfill &\hbox{if } k\leq n,\, i=\psi (k);\hfill \cr
1,\hfill &\hbox{if }  i=n;\hfill \cr
q^{-2},\hfill &\hbox{otherwise}.\hfill 
                             \end{matrix}
                    \right.
\label{mu4}
\end{equation}

We define the bracketing of $u(k,m),$ $k\leq m$ as follows.
\begin{equation}
u[k,m]=\left\{ \begin{matrix} [[[\ldots [x_k,x_{k+1}], \ldots ],x_{m-1}], x_m],\hfill 
&\hbox{if } m<\psi (k);\hfill \cr
 [x_k,[x_{k+1},[\ldots ,[x_{m-1},x_m]\ldots ]]],\hfill &\hbox{if } m>\psi (k);\hfill \cr 
\beta [u[n+1,m],u[k,n]],\hfill &\hbox{if } m=\psi (k),\hfill 
\end{matrix}\right.
\label{ww}
\end{equation}
where $\beta =-p(u(n+1,m),u(k,n))^{-1}$ normalizes the coefficient at $u(k,m).$
Conditional identity (\ref{ind}) and the second group of defining relations (\ref{relb}) 
show that the value of  $u[k,m]$ in $U_q^+({\mathfrak so}_{2n+1})$
is independent of the precise alignment of brackets provided that $m\leq n$ or $k>n.$
Formula (\ref{ant1}) and evident induction show that 
\begin{equation}
g_kg_{k+1}\cdots g_m\sigma (u[k,m])\sim u[\psi (m),\psi (k)],
\label{ant2}
\end{equation}
where $\sigma $ is the antipode.
\begin{lemma} {\rm (\cite[Corollary 3.13]{Kh08}).} If $m\neq \psi (k),$ $k\leq n<m,$ 
then in $U_q^+({\mathfrak so}_{2n+1})$ we have
\begin{equation}
u[k,m]=[u[k,n],u[n+1,m]]=\beta[u[n+1,m],u[k,n]],
\label{ww1}
\end{equation}
where $\beta =-p(u(n+1,m),u(k,n))^{-1}.$
\label{rww}
\end{lemma}
\begin{proposition} {\rm (\cite[Proposition 3.14]{Kh08}).}
If $m\neq \psi (k),$ $k\leq i<m,$ then in $U_q^+({\mathfrak so}_{2n+1})$ for each $i,$ $k\leq i<m$  we have
$$
[u[k,i],u[i+1,m]]=u[k,m]
$$
with only two possible exceptions being  $i=\psi (m)-1,$  and $i=\psi (k).$
In particular this decomposition holds for arbitrary $i$ if $m\leq n$ or $k>n.$
\label{ins2}
\end{proposition}
\begin{proposition} Let $k\leq i<j<m.$
If  $m\neq \psi (i)-1,$  $j\neq \psi (k),$ and $m\neq \psi (k)$ then 
$[u[k,i],u[j+1,m]]=0.$ If $m\neq \psi (i)-1,$  $j\neq \psi (k),$ and  $i\neq \psi (j)-1,$ then $[u[j+1,m],u[k,i]]=0.$
\label{NU} 
\end{proposition}
\begin{proof}
The former statement follows from \cite[Proposition 3.15]{Kh08}. 
Let $m\neq \psi (i)-1,$  $j\neq \psi (k),$ and  $i\neq \psi (j)-1.$ If additionally $m\neq \psi (k)$
then still  \cite[Proposition 3.15]{Kh08} applies. Assume $m=\psi (k).$
We shall use the following two relations
\begin{equation}
[x_{\lambda },[x_{\lambda -1}x_{\lambda }x_{\lambda +1}]]
=[[x_{\lambda -1}x_{\lambda }x_{\lambda +1}],x_{\lambda }]=0, 
l* {too}
\end{equation} 
where $1<\lambda <2n,$ $\lambda \neq n,n+1.$
The latter one is precisely \cite[Eq. (3.7)]{Kh08} with $k\leftarrow \lambda $ if 
$\lambda <n,$ and with $k \leftarrow \psi (\lambda )$ if $\lambda >n+1.$
The former one follows from antisymmetry identity (\ref{bri}), for
$$
p(x_{\lambda }, x_{\lambda -1}x_{\lambda }x_{\lambda +1})p(x_{\lambda -1}x_{\lambda }x_{\lambda +1},x_{\lambda })
=q^{-2}q^4q^{-2}=1.
$$
That equalities imply the following two ones 
\begin{equation}
[x_{\lambda },u[k,a]]=0, \ \ k\leq \lambda <a\leq n;
\label{too1}
\end{equation}
 \begin{equation}
[u[k,a],x_{\lambda }]=0, \ \ n<k<\lambda \leq a.
\label{too2}
\end{equation} 
Indeed, if in (\ref{too1}) we have $\lambda =k$ then $[x_k,u[k,a]]=[[x_k,[x_k,x_{k+1}]], u[k+2,a]]=0,$
for in this case $u[k,a]$ is independent of the precise alignment of brackets, see Lemma \ref{indle}, 
and of course $[x_k, u[k+2,a]]=0$ due to Lemma \ref{sepp}.
If $\lambda >k$ then 
$$
[x_{\lambda },u[k,a]]
\sim [u[k,\lambda -2], [x_{\lambda },[x_{\lambda -1}x_{\lambda }x_{\lambda +1}]],u[\lambda +2,a] ]]=0,
$$
for $[x_{\lambda },u[k,\lambda -2]]=[x_{\lambda },u[\lambda +2,a]]=0.$ The proof of (\ref{too2})
is quite similar.

Let $i\leq n<j.$ In this case the equality $[u[1+j,m],u[k,i]]=0$ follows from (\ref{too1})
 with $a\leftarrow i$ if $1+j>\psi (i).$ If $1+j<\psi (i)$ this follows from
(\ref{too2}) with $k\leftarrow 1+j,$ $a\leftarrow m.$  
We have $1+j\neq \psi (i),$ for $i\neq \psi (j)-1.$

Let $i<j\leq n.$ By Lemma \ref{rww} we have $u[1+j,m]$ $=[u[1+j,n],$ $u[n+1,m]].$
At the same time $[u[n+1,m],u[k,i]]=0$ due to (\ref{too2}) with $k\leftarrow n+1,$ $a\leftarrow m,$
while $[u[1+j,n],u[k,i]]=0$ since $u[k,i]$ and $u[1+j,n]$ are separated, see Lemma \ref{sepp}.

Similarly, if $n<i<j$ then by  Lemma \ref{rww} we have $u[k,i]=[u[k,n],u[n+1,i]].$
At the same time $[u[1+j,m],u[k,n]]=0$ due to (\ref{too1}) with $a\leftarrow n,$
while $[u[1+j,m],u[n+1,i]]=0$ since $u[1+j,m]$ and $u[n+1,i]$ are separated. 
\end{proof}

The elements $u[k,m]$ are important due to the following statements. 
\begin{proposition} {\rm \cite[Proposition 4.1]{Kh08}.} If $q^3\neq 1,$ $q^4\neq 1,$ then
values of the elements  $u[k,m],$ $k\leq m<\psi (k)$ form a set of PBW-generators
for the algebra $U_q^+({\mathfrak so}_{2n+1})$ over {\bf k}$[G].$ All heights are infinite.
\label{strB}
\end{proposition}
\begin{proposition} {\rm (\cite[Proposition 4.5]{Kh08}).} 
If the multiplicative order $t$ of $q$ is finite, $t>4,$ then the values of $u[k,m],$ 
$k\leq m<\psi (k)$ form a set of PBW-generators for $u_q^+({\mathfrak so}_{2n+1})$ 
over {\bf k}$[G].$ The height $h$ of $u[k,m]$ equals $t$ if $m=n$ or $t$ is odd.
If $m\neq n$ and $t$ is even, then $h=t/2.$ In all cases $u[k,m]^h=0$
in $u_q^+({\mathfrak so}_{2n+1}).$
\label{strBu}
\end{proposition}
We stress that due to (\ref{mu21}) the height $h$ here equals the multiplicative 
order of $p_{uu},$ where $u=u[k,m].$ The coproduct on $u[k,m],$ $k\leq m\leq 2n$ is given by the following 
elegant formula, see \cite[Theorem 4.3]{Kh08}:
\begin{equation}
\Delta (u[k,m])=u[k,m]\otimes 1+g_{k\rightarrow m}\otimes u[k,m]
\label{co}
\end{equation}
$$
+\sum _{i=k}^{m-1}\tau _i(1-q^{-2})g_{k\rightarrow i}\, u[i+1,m]\otimes u[k,i],
$$
where by definition $g_{k\rightarrow i}=g_kg_{k+1}\cdots g_i=g(u[k,i]),$ and 
\begin{equation}
\tau_i=q^{\delta _i^n} =\left\{ \begin{matrix}
q,\hfill &\hbox{if } i=n;\hfill \cr
1,\hfill &\hbox{otherwise}.\hfill 
                             \end{matrix}
                    \right.
\label{tau}
\end{equation}

Formula (\ref{co}) with (\ref{calc}) and (\ref{calcdu}) allows one to find  the differentiation formulae
\begin{equation}
\partial _i(u[k,m])=\left\{ \begin{matrix}
(1-q^{-2})\tau _ku[k+1,m],\hfill & \hbox{ if } i\in \{ k,\psi (k)\} ,k<m;\hfill \cr
0, \hfill & \hbox{ if } i\notin \{ k,\psi (k)\} ;\hfill \cr
1, \hfill & \hbox{ if } i\in \{ k,\psi (k)\} ,k=m.\hfill
\end{matrix} \right.
\label{pdee}
\end{equation}
\begin{equation}
\partial _i^*(u[k,m])=\left\{ \begin{matrix}
(1-q^{-2})\tau _{m-1 }u[k,m-1],\hfill & \hbox{ if } i\in \{ m, \psi (m)\} , m>k;\hfill \cr
0, \hfill & \hbox{ if } i\notin \{ m,\psi (m)\};\hfill \cr
1, \hfill & \hbox{ if } i\in \{ m, \psi (m)\} , m=k.\hfill
\end{matrix} \right.
\label{pdu}
\end{equation}

These differentiation formulae with differential representation of the simplest 
adjoint operators (\ref{sqi3}), (\ref{sqi4}) allows one to find the (skew)
bracket of basis elements $u[k,m]^{\mp }$ with the main generators $x_i^{\pm }.$
\begin{lemma} If $k<m,$ then in $U_q(\mathfrak{so}_{2n+1})$ we have
\begin{equation}
[u[k,m],x_i^-]\sim \left\{ \begin{matrix} 0, \hfill & \hbox{ if } i\notin \{ k,m,\psi (k), \psi (m)\} ;\hfill \cr
g_kf_ku[k+1,m], \hfill & \hbox{ if } i\in \{ k, \psi (k)\} ,\  m\neq \psi (k);\hfill \cr
u[k,m-1], \hfill & \hbox{ if } i\in \{ m, \psi(m)\} , \  m\neq \psi (k).\hfill
\end{matrix}
\right.
\label{kom1}
\end{equation}
\label{ruk1}
\end{lemma}
\begin{proof}  
The statement follows from (\ref{sqi4}), (\ref{pdu}), and (\ref{pdee}).
\end{proof}
\begin{lemma} If $i<j,$ then in $U_q(\mathfrak{so}_{2n+1})$ we have
\begin{equation}
[x_k,u[i,j]^-]\sim \left\{ \begin{matrix} 0, \hfill & \hbox{ if } k\notin \{ i,j,\psi (i), \psi (j)\} ; \hfill \cr
g_if_iu[i+1,j], \hfill & \hbox{ if } k\in \{ i, \psi (i)\} ,\ j\neq \psi (i);\hfill \cr
u[i,j-1], \hfill & \hbox{ if } k\in \{ j, \psi(j)\} , \ j\neq \psi (i). \hfill 
\end{matrix}
\right.
\label{kom2}
\end{equation}
\label{ruk2}
\end{lemma}
\begin{proof} 
The statement follows from (\ref{sqi3}), (\ref{pdu}), and (\ref{pdee}). 
\end{proof}
\begin{corollary} If either $k,m, \psi (k), \psi (m)\notin [i,j]$ or
$i,j, \psi (i), \psi (j)\notin [k,m],$ then
$$
[u[k,m],u[i,j]^-]=0.
$$
\label{ruk3}
\end{corollary}
\begin{proof}  
If $k,m, \psi (k), \psi (m)\notin [i,j],$ then due to Lemma \ref{ruk2} we have  $[u[k,m],x_t^-]=0$
for every $t\in [i,j].$ Hence ad-identity (\ref{br1}) and evident induction 
imply the required equality, for $u[i,j]^-$ belongs to the subalgebra generated by
$x_t^-, t\in [i,j].$  If $i,j, \psi (i), \psi (j)\notin [k,m],$ then in perfect analogy 
we use ad-identity (\ref{br1f}) and Lemma \ref{ruk1}.  
\end{proof}

\section{Roots and related properties of quantum Borel subalgebras}

Recall  that a {\it root} of a homogeneous right coideal subalgebra $U$ is  degree  of a 
PBW-generator of $U$, see \cite[Definition 2.9]{KL}. Due to \cite[Corollary 5.7]{Kh08}
all roots of a homogeneous right coideal subalgebra $U \supset G$ 
of positive  quantum Borel subalgebra have the form 
$[k:m]\stackrel{df}{=}x_k+x_{k+1}+\cdots +x_{m-1}+x_m=D(u[k,m]),$
where $1\leq k\leq m \leq 2n.$ Here  $x_{2n-i+1}=x_i,$ see Definition \ref{fis}.
An $U$-root is {\it simple} if it is not a sum of two or more other $U$-roots. 

In what follows  
$\Sigma (U)$ denotes the submonoid of $\Gamma ^+$ generated by all $U$-roots. Certainly 
degree of any nonzero homogeneous element from $U$ belongs to $\Sigma (U).$ 
Moreover if $q$ is not a root of 1, then all PBW-generators have infinite heights.
 Hence in this case $\Sigma (U)$ is precisely the set of all degrees of nonzero 
homogeneous elements from $U$.
Simple $U$-roots are nothing more than indecomposable elements from $\Sigma (U).$ In particular 
\cite[Lemma 8.9]{Kh08} shows that $U$ is uniquely defined by $\Sigma (U):$
if $\Sigma (U)=\Sigma (U_1),$ then $U=U_1.$ The following statement shows  that 
the lattice of right coideal subalgebras that contain the coradical is isomorphic to some
lattice of submonoids of $\Gamma ^+.$

\begin{proposition} Let $U, U_1\supseteq G$ 
be $($homogeneous$)$ right coideal subalgebras 
of $U_q^+(\mathfrak{so}_{2n+1}),$ $q^t\neq 1$  $($respectively of $u_q^+(\mathfrak{so}_{2n+1}),$
if $q^t=1,$ $t>4).$ Then $U\subseteq U_1$ if and only if 
$\Sigma (U)\subseteq \Sigma (U_1).$
\label{lat}
\end{proposition}
\begin{proof}
If ${ U}\subseteq { U}_1,$ then every PBW-generator $a$ of { U} belongs to ${ U}_1.$
In particular $a$ is a (noncommutative) polynomial in $G$ and PBW-generators of ${ U}_1.$
Hence every $U$-root, being a degree of some $a,$ is a sum of 
$U_1$-roots (degrees of PBW-generators of $U_1);$
that is, $\Sigma ({ U})\subseteq \Sigma ({ U}_1).$

Let $\Sigma (U)\subseteq \Sigma (U_1).$ Consider the subalgebra $U_2$
generated by $U$ and $U_1.$ Certainly this is a right coideal subalgebra.
At the same time  
$$\Sigma (U_1)\subseteq \Sigma (U_2)\subseteq \Sigma (U)+\Sigma (U_1)=\Sigma (U_1),$$
which implies $\Sigma (U_1)=\Sigma (U_2),$ and $U_1=U_2\supseteq U.$
\end{proof}

The proved statement implies the following nice characterization of elements from $U$
 in terms of degrees of its partial derivatives. Recall that the subalgebra $A$ of 
$U_q^+(\mathfrak{so}_{2n+1})$ or $u_q^+(\mathfrak{so}_{2n+1})$ generated over {\bf k}
by $x_1,x_2,\ldots ,x_n$ has a noncommutative differential calculus (\ref{defdif}).
Due to (\ref{calc}) the subalgebra $U_A\stackrel{df}{=}U\cap A$ is differential:
$\partial _i(U_A)\subseteq U_A,$ $1\leq i\leq n.$ Conversely, 
if $U_A$ is any  differential subalgebra of $A$ homogeneous in each $x_i,$
then the subalgebra $U$ generated by $U_A$ and $G$ is a right coideal subalgebra of 
$U_q^+(\mathfrak{so}_{2n+1})$ or $u_q^+(\mathfrak{so}_{2n+1}),$ see \cite[Lemma 2.10]{KL}.
Let $\partial _u,$ $u=x_{i_1}x_{i_2}\cdots x_{i_m}$ denote the differential operator
$\partial _{i_1}\partial _{i_2}\cdots \partial _{i_m}.$ Certainly if $f\in U_A,$ $\partial _u(f)\neq 0,$
 then degree of $\partial _u(f)$ belongs to $\Sigma (U),$ for $\partial _u(f)\in U_A\subset U.$
Interestingly the converse statement is true as well.
\begin{proposition} Let $U\supseteq G$ 
be a $($homogeneous$)$ right coideal subalgebra 
of $U_q^+(\mathfrak{so}_{2n+1}),$ $q^t\neq 1$  $($respectively of $u_q^+(\mathfrak{so}_{2n+1}),$
if $q^t=1,$ $t>4).$  If $f\in A$ is a homogeneous 
element such that for each differential operator 
$\partial _u$ we have  $D(\partial _u(f))\in \Sigma (U)$ or $\partial _u(f)=0,$ then $f\in U.$
\label{lat1}
\end{proposition}
\begin{proof}
Consider the differential subalgebra $B$ generated by $U_A$ and $f.$ As an algebra 
$B$ is generated by $U_A$ and all $\partial _u(f).$ Hence degrees of all nonzero homogeneous 
elements from $B$ belong to $\Sigma (U)$ (in particular  $D(f)=D(\partial _{\emptyset }(f))\in \Sigma (U)$).
Proposition \ref{lat} applied to the pair $U,$ $BG$ implies $BG\subseteq U,$ and $f\in U.$
\end{proof}
We stress that the condition $D(\partial _u(f))\in \Sigma (U)$ is equivalent to 
$D(f)\in \Sigma (U)+D(u).$ Hence we may restate the proved statement:
$f\in U$ if and only if $\partial _u(f)=0$ for all words $u$ such that $D(f)\notin \Sigma (U)+D(u).$
To put it another way, we have a representation of homogeneous components 
$U_A^{(\gamma )},$ $\gamma \in \Gamma ^+$  of $U_A$ in the form of kernel
of a set of differential operators: 
\begin{equation} 
U_A^{(\gamma )}=\bigcap _{\gamma \notin \Sigma (U)+D(u)} {\rm Ker}\, \partial _u.
\label{late1}
\end{equation}
Moreover Proposition \ref{lat1} shows that right coideal subalgebras are differentially 
closed in the following sense. 
\begin{corollary} If under the conditions of the above proposition $D(f)\in \Sigma (U)$
and $\partial _i (f)\in U,$ $1\leq i\leq n,$ then $f\in U.$
\label{lat2}
\end{corollary}
\begin{proof}
Indeed, if $\partial _i (f)\in U,$ $1\leq i\leq n,$ then of course $\partial _u (f)\in U$
for all nonempty words $u.$ In particular either $D(\partial _u (f))\in \Sigma (U)$
or $\partial _u (f)=0.$ Proposition \ref{lat1} applies.
\end{proof}

Needless to say that all statements of this and the above sections remain valid for negative quantum Borel subalgebra too.
In particular  all roots of a homogeneous right coideal subalgebra $U^- \supset F$ 
of negative  quantum Borel subalgebra have the form 
$[i:j]^-\stackrel{df}{=}x_i^-+x_{i+1}^-+\cdots +x_{j-1}^-+x_j^-,$
where $1\leq i\leq j \leq 2n$.

\section{Minimal generators for right coideal subalgebras\\  of the quantum Borel algebra}
Let $S$ be a set of integer numbers from the interval $[1,2n].$
A (noncommutative) polynomial $\Phi^{S}(k,m),$ $1\leq k\leq m\leq 2n$ is defined
by induction on the number $r$ of elements
in the set $S \cap [k,m)=\{ s_1,s_2,\ldots ,s_r\} ,$ $k\leq s_1<s_2<\ldots <s_r<m$
as follows:
\begin{equation}
\Phi^{S}(k,m)=u[k,m]-(1-q^{-2})\sum_{i=1}^r \alpha _{km}^{s_i} \, 
\Phi^{S}(1+s_i,m)u[k,s_i],
\label{dhs}
\end{equation}
where $\alpha _{km}^{s}=\tau _{s}p(u(1+s,m),u(k,s))^{-1},$ 
while the $\tau $'s was defined in (\ref{tau}).

\smallskip

We display the element  $\Phi ^{S}(k,m)$
schematically as a sequence of black and white points labeled by the numbers
$k-1,$ $k,$ $k+1, \ldots $ $m-1,$ $m,$ where the first point is always white, and
the last one is always black, while an intermediate point labeled by $i$ is black if and only if 
$i\in S:$  
\begin{equation}
 \stackrel{k-1}{\circ } \ \ \stackrel{k}{\circ } \ \ \stackrel{k+1}{\circ } 
\ \ \stackrel{k+2}{\bullet }\ \ \ \stackrel{k+3}{\circ }\ \cdots
\ \ \stackrel{m-2}{\bullet } \ \ \stackrel{m-1}{\circ }\ \ \stackrel{m}{\bullet }
\label{grb}
\end{equation}
Sometimes, if $k\leq n<m,$ it is more convenient  to display the element
$\Phi ^{S}(k,m)$ in two lines putting the points labeled by indices 
$i,\psi (i)$ that define the same variable $x_i=x_{\psi (i)}$ in one column:
\begin{equation} 
\begin{matrix}
 \ \ \ \ \ \ \ & \stackrel{m}{\bullet } \ \cdots & \bullet  & \stackrel{\psi (i)}{\circ }   \ \cdots & \stackrel{n+1}{\bullet }  \cr 
 \stackrel{k-1}\circ \ \circ  \ \cdots \ & \stackrel{\psi (m)}{\circ } \ \cdots & \bullet  & \stackrel{i}{\bullet } \ \cdots & \stackrel{n}{\circ }   
\end{matrix}
\label{grb1}
\end{equation}
The elements $\Phi ^{S}(k,m)$ are very important since every right coideal subalgebra 
$U\supseteq G$ of the  quantum Borel subalgebra is generated as an algebra 
by $G$ and the elements of this form, see \cite[Corollary 5.7]{Kh08}. 
Moreover $U$ is uniquely defined by its {\it root sequence}
$\theta =(\theta _1,\theta_2,\ldots ,\theta _n).$ The root sequence satisfies 
$0\leq \theta_i\leq 2n-2i+1,$ and each sequence satisfying these conditions is a root sequence
for some $U$. There exists a constructive algorithm that allows one to find the generators
$\Phi ^{S}(k,m)$ if the sequence $\theta $ is given, see \cite[Definition 10.1 and Eq. (10.6)]{Kh08}.
More precisely the algorithm allows one to find all possible values of the numbers $k,m$
 and the sets $S.$ In particular one may construct all schemes (\ref{grb}) for the generators.
However the explicit form of $\Phi ^{S}(k,m)$ needs  complicated inductive procedure (\ref{dhs}).
These generators satisfy two additional important properties. First, their degrees,
$D(\Phi ^{S}(k,m))=x_k+x_{k+1}+\cdots +x_m,$ are simple $U$-roots; that is, 
$D(\Phi ^{S}(k,m))$ is not a sum of nonzero degrees of other elements from $U$, see
\cite[Claims 7,8]{Kh08}. Next, the set $S$ is always $(k,m)$-regular in the sense
of the following definition, see \cite[Claim 5]{Kh08}. 

\begin{definition} \rm Let $1\leq k\leq n<m\leq 2n.$
A set $S$ is said to be {\it white $(k,m)$-regular} if for every
$i,$ $k-1\leq i<m,$ such that $k\leq \psi (i)\leq m+1$ either $i$ or $\psi (i)-1$ does not belong 
to $S \cup \{ k-1, m\} .$

A set $S$ is said to be {\it black $(k,m)$-regular} if for every
$i,$ $k\leq i\leq m,$ such that $k\leq \psi (i)\leq m+1$ either $i$ or $\psi (i)-1$ belongs 
to $S \setminus \{ k-1,m\} .$

A set $S$ is said to be $(k,m)$-{\it regular} if it is either black or white $(k,m)$-regular.

 If $m\leq n,$ or $k>n$ (or, equivalently, if $u[k,m]$ is of degree $\leq 1$ in $x_n$),
 then by definition each set $S$ is both white and black $(k,m)$-regular.
\label{reg1}
\end{definition}

To illustrate the notion of a regular set, we shall need a {\it shifted representation} 
that appears from (\ref{grb1}) by shifting the upper line to the left by one step and putting
the colored point  labeled by $n,$ if any, to the vacant position (so that this point appears twice 
in the shifted scheme):
\begin{equation} 
\begin{matrix}
 \ \ \ \ \ \ \ & \stackrel{m}{\bullet } \ \cdots & \circ  
& \stackrel{n+i}{\circ }   \ \cdots & \stackrel{n+1}{\bullet } & \stackrel{n}{\circ } \Leftarrow \cr 
 \stackrel{k-1}\circ \ \circ  \ \cdots \ & \stackrel{\psi (m)-1}{\bullet } \ \cdots &
 \bullet  & \stackrel{n-i}{\bullet } \ \cdots & \stackrel{n-1}{\circ }& \stackrel{n}{\circ } \hfill   
\end{matrix}
\label{grb2}
\end{equation}

If $k\leq n<m$ and $S$ 
is white $(k,m)$-regular, then $n\notin S$,  for $\psi (n)-1=n.$ 
If additionally $m<\psi (k),$ then taking $i=\psi (m)-1$ we get $\psi (i)-1=m,$
hence the definition implies $\psi (m)-1\notin S$. We see that if $m<\psi (k),$
$k\leq n<m,$ then $S$ is white $(k,m)$-regular if and only if the shifted scheme
of $\Phi ^{S}(k,m)$ given in (\ref{grb2}) has no black columns:
\begin{equation}  
\begin{matrix}
 \ \ \ \ \ \ \ & \stackrel{m}{\bullet }&\cdots & \bullet  
& \stackrel{n+i}{\circ }  & \circ  & \cdots & \stackrel{n}{\circ }   \Leftarrow \cr 
 \stackrel{k-1}\circ \ \cdots & \stackrel{\psi (m)-1}{\circ }& \cdots & \circ  
& \stackrel{n-i}{\bullet } & \circ & \cdots & \stackrel{n}{\circ }   \hfill
\end{matrix}
\label{grab}
\end{equation}
In the same way, if $m>\psi (k),$
then for $i=\psi (k)$ we get $\psi (i)-1=k-1,$ hence $\psi (k)\notin S$.
That is, if $m>\psi (k),$
$k\leq n<m,$ then $S$ is white $(k,m)$-regular if and only if the shifted scheme
(\ref{grb2}) has no black columns and the first from the left complete column
is a white one.
\begin{equation} 
\begin{matrix}
\stackrel{m}{\bullet } \ \cdots & \stackrel{\psi (k)}{\circ } & \cdots &  \bullet & \stackrel{n+i}{\circ } & \circ & \cdots & \stackrel{n}{\circ }  \Leftarrow
\cr
 \ \ \ \ \ \hfill & \stackrel{k-1}{\circ }& \cdots & \circ &
\stackrel{n-i}{\bullet } & \circ &\cdots &  \stackrel{n}{\circ }  \hfill
\end{matrix}
\label{grab1}
\end{equation}

{\it All in all, a set $S$ is white $(k,m)$-regular, where $1\leq k\leq n<m\leq 2n,$ 
if the shifted scheme obtained by painting $k-1$ black does not contain columns 
with two black points.}

Similarly, if $k\leq n<m$ and $S$ 
is black $(k,m)$-regular, then $n\in S$. 
If additionally $m<\psi (k),$ then taking $i=\psi (m)-1$ we get $\psi (i)-1=m,$
hence  $\psi (m)-1\in S$. 
We see that if $m<\psi (k),$
$k\leq n<m,$ then $S$ is black $(k,m)$-regular if and only if the shifted scheme
(\ref{grb2}) has no white columns and the first from the left complete column 
is a black one.
\begin{equation}  
\begin{matrix}
 \ \ \ \ \ \ \ &\stackrel{m}{\bullet } & \cdots & \bullet  
& \stackrel{n+i}{\circ }  & \bullet  & \cdots & \stackrel{n}{\bullet }  \Leftarrow \cr 
 \stackrel{k-1}{\circ }\ \cdots & \stackrel{\psi (m)-1}{\bullet } & \cdots & \bullet  
& \stackrel{n-i}{\bullet } & \circ & \cdots & \stackrel{n}{\bullet }  \hfill  
\end{matrix}
\label{grab2}
\end{equation}
If $m>\psi (k),$
then for $i=\psi (k)$ we get $\psi (i)-1=k-1,$ hence $\psi (k)\in S$.
That is, if $m>\psi (k),$
$k\leq n<m,$ then $S$ is black $(k,m)$-regular if and only if the shifted scheme
(\ref{grb2}) has no white columns:
\begin{equation} 
\begin{matrix}
\stackrel{m}{\bullet } \ \cdots & \stackrel{\psi (k)}{\bullet } & \cdots &  
\bullet & \stackrel{n+i}{\circ } & \bullet & \cdots & \stackrel{n}{\bullet }  \Leftarrow 
\cr
 \ \ \ \ \ \hfill & \stackrel{k-1}{\circ } & \cdots & \circ &
\stackrel{n-i}{\bullet } & \bullet  & \cdots &  \stackrel{n}{\bullet }  \hfill 
\end{matrix}
\label{grab3}
\end{equation}

{\it All in all, a set $S$ is black $(k,m)$-regular, where $1\leq k\leq n<m\leq 2n,$ 
if the shifted scheme obtained by painting $m$ white does not contain columns 
with two white points.}

At the same time we should stress that {\it if $m=\psi (k),$
then no one set is $(k,m)$-regular}. Indeed, for $i=k-1$ we have
$\psi (i)-1=m.$ Hence both of the elements $i, \psi (i)-1$ belong to $S \cup \{ k-1, m\} ,$
and therefore $S$ is not white $(k,\psi (k))$-regular. If we take $i=m,$ then 
$\psi (i)-1=k-1,$ and no one of the elements $i, \psi (i)-1$
belongs to $S \setminus \{ k-1,m\} .$ Thus $S$ is neither black $(k,\psi (k))$-regular.
\begin{lemma}
A set $S$ is white $($black$)$ $(k,m)$-regular if and only if its complement $\overline{S}$ with respect to 
$[k,m)$ is black $($white$)$ $(k,m)$-regular.
\label{dop}
\end{lemma}
\begin{proof} The shifted scheme for $\Phi^{\overline{S}}(k,m)$ appears from
that for $\Phi^{S}(k,m)$ by changing the color of all points except the first one, $k-1,$ and the last one, $m.$ 
Under this re-coloring a scheme of type (\ref{grb2}) is
transformed to (\ref{grab2}), while  a scheme of type (\ref{grab}) is transformed to (\ref{grab3})
and vice versa.
\end{proof}
\begin{lemma}
A set $S$ is white $($black$)$ $(k,m)$-regular if and only if $\psi (S)-1$
is white $($black$)$ $(\psi (m), \psi (k))$-regular. Here $\psi (S)-1=\{ \psi (s)-1\, |\, s\in S\} .$
\label{dop1}
\end{lemma}
\begin{proof} The shifted scheme for $\Phi^{\psi (S)-1}(\psi (m), \psi (k))$ appears from
that for $\Phi^{S}(k,m)$ by switching rows  and changing the color of the first and the last points.
Under that  transformation a scheme of type (\ref{grab}) is
transformed to (\ref{grab1}), while  a scheme of type (\ref{grab2}) is transformed to (\ref{grab3}) and vice versa. 
\end{proof}
\begin{theorem}{\rm (\cite[Corollary 10.4]{Kh08})}. If $q$ is not a root of 1 then
every  right coideal subalgebra of $U_q^+(\mathfrak{so}_{2n+1})$ that contains $G$ is generated as an algebra 
by $G$ and a set of elements $\Phi^{S}(k,m)$ with $(k,m)$-regular sets $S.$
If $q^t=1,$ $t>4,$ then this is the case for every homogeneous right coideal subalgebra of $u_q^+(\mathfrak{so}_{2n+1})$ 
that contains $G.$ 
\label{rig}
\end{theorem}
Of course this theorem is valid for negative quantum Borel subalgebra as well.
In this case the generators take up the form $\Phi^{S}_-(k,m)$ with $(k,m)$-regular sets $S$,
where $\Phi^{S}_-(k,m),$ is the element (\ref{dhs})
under the replacement $x_i\leftarrow x_i^-,$ $1\leq i\leq n.$
\begin{proposition}
If $S$ is a $(k,m)$-regular set, then 
$$
\Phi ^{S}(k,m)\sim \Phi ^{T}(\psi (m),\psi (k)),
$$
where ${T}=\overline{\psi (S)-1}$ is a $(\psi (m),\psi (k))$-regular set
and $\psi (S)-1$ denotes the set $\{ \psi (s)-1\, |\, s\in S\} ,$
while  the complement is related to the interval  $[\psi (m),\psi (k)).$
\label{xn0}
\end{proposition}
\begin{proof} The proof follows from \cite[Proposition 7.10]{Kh08}
since due to Lemmas \ref{dop} and \ref{dop1} the set $S$ is white 
(black) $(k,m)$-regular if and only if $T$ is black (white) $(\psi (m),\psi (k))$-regular. 
 \end{proof}
\begin{lemma}
Let  $S$ be a white $(k,m)$-regular set. Assume $s$ is a black point on the scheme $(\ref{grb}),$
and $k-1\leq t<s\leq m.$ Then $S$ is white $(1+t,s)$-regular if and only if either $\psi (t)-1$ is a white point
or  $\psi (t)-1\notin [t,s].$ In particular if either $t$ is black or $t=k-1,$ then $S$ is white $(1+t,s)$-regular. 
\label{si}
\end{lemma}
\begin{proof} The general statement follows from interpretation of regular sets given on diagrams 
(\ref{grab}), (\ref{grab1}).  The points $t,$ $\psi (t)-1$ form a column 
on the shifted scheme. Hence if either $t$ is black or $t=k-1,$ then 
 $\psi (t)-1$ is white or it does not appear on the scheme at all, that is  $\psi (t)-1\notin [k-1,m]\supseteq [t,s].$
 \end{proof}
Similarly we have the following statement.
\begin{lemma}
Let  $S$ be a black $(k,m)$-regular set. Assume $t$ is a white point on the scheme $(\ref{grb}),$
and $k-1\leq t<s\leq m.$ Then $S$ is black $(1+t,s)$-regular if and only if either $\psi (s)-1$ is a black point
or  $\psi (s)-1\notin [t,s].$ In particular if either $s$ is white or $s=m,$ then $S$ is black $(1+t,s)$-regular. 
\label{si1}
\end{lemma}
\begin{lemma} {\rm (\cite[Corollaries 7.7, 7.13]{Kh08})} Let $k\leq t<m.$ The decomposition 
\begin{equation}  
\Phi^{S}(k,m)\sim \left[ \Phi^{S}(k,t),\Phi^{S}(1+t,m)\right] 
\label{desc1}
\end{equation}
is valid if either $S \cup \{ t\} $ is white $(k,m)$-regular and $t\notin S$, or
$S$ is black $(k,m)$-regular and $t\notin S\setminus \{ n\} .$
\label{xn1}
\end{lemma}
\begin{lemma} {\rm (\cite[Corollaries 7.5, 7.14]{Kh08})} Let $k\leq s<m.$ The decomposition 
\begin{equation}  
\Phi^{S}(k,m)\sim [\Phi^{S}(1+s,m),\Phi^{S}(k,s)]
\label{desc2}
\end{equation}
is valid if either $S$ is white $(k,m)$-regular and $s\in S\cup \{ n\} ,$ or
$S \setminus \{ s\}$ is black $(k,m)$-regular and $s\in S$.
\label{xn2}
\end{lemma}
We stress that due to Lemmas \ref{si}, \ref{si1} in these lemmas the set $S$ 
appears to be  both $(k,t)$-regular and $(1+t,m)$-regular;
that is, the multiple use of the lemmas is admissible. 
\begin{lemma} If $S$ is $(k,m)$-regular set, then we have
\begin{equation}  
g_{k\rightarrow m}\, \sigma (\Phi^{S}(k,m))\sim \Phi^{\psi (S)-1}(\psi(m), \psi (k))\sim \Phi^{\overline{S}}(k, m),
\label{desc3}
\end{equation}
where $\overline{S}$ is the complement of $S$ with respect to $[k,m),$
and $\sigma $ is the antipode.
\label{xn21}
\end{lemma}
\begin{proof} Assume $S$ is white $(k,m)$-regular.
We use induction on the number $r$ of elements in the intersection $S \cap [k,m).$
If $r=0,$ then the left hand side equals $g_{k\rightarrow m}\, \sigma (u[k,m])\sim u[\psi (m),\psi (k)]$
due to (\ref{ant2}). Proposition \ref{xn0} with $S \leftarrow [k,m)$ implies 
$u[\psi (m),\psi (k)]\sim \Phi^{[k,m)}(k, m),$ which is required. If $r>0$ then we choose $s\in S,$
$k\leq s<m.$ By Lemma \ref{xn2} we have decomposition (\ref{desc2}).  Using 
(\ref{ant1}) and  the inductive supposition, we have
\begin{equation}
g_{k\rightarrow m}\, \sigma (\Phi^{S}(k,m))\sim [\Phi^{\psi (S)-1}(\psi(s), \psi (k)),\Phi^{\psi (S)-1}(\psi(m), \psi (1+s))].
\label{desc4}
\end{equation}
At the same time Lemma \ref{dop1} implies that 
$\psi (S)-1$ is a white $(\psi(m), \psi (k))$-regular set, and $ \psi (1+s)=\psi (s)-1\in \psi (S)-1.$
Hence we may apply Lemma \ref{xn2}, that shows that the right hand side of (\ref{desc4})
is proportional to $\Phi^{\psi (S)-1}(\psi(m), \psi (k)).$ This proves the first proportion in (\ref{desc3}).
The second one follows from Proposition \ref{xn0}.

If $S$ is black $(k,m)$-regular, then Lemma  \ref{xn0} reduces 
the consideration to white regular case. 
 \end{proof}

\begin{lemma} Let $U^S(k,m)$ be the right coideal subalgebra generated by $G$ and by an element $\Phi ^S(k,m)$
with a $(k,m)$-regular set $S.$ In this case the monoid $\Sigma (U^S(k,m))$ defined in the above section
coincides with the monoid $\Sigma $ generated by all $[1+t:s]$ with $t$ being 
a white point and $s$ being a black point on the scheme $(\ref{grb}).$
\label{sig}
\end{lemma}
\begin{proof}
Proposition  9.3 \cite{Kh08} implies that degrees of all homogeneous elements from $U^S(k,m)$
belong to $\Sigma .$ Hence $\Sigma (U^S(k,m))\subseteq \Sigma .$ At the same time Lemma 9.7 \cite{Kh08}
says that every indecomposable in $\Sigma $ element $[1+t:s]$ 
 is a simple $U^S(k,m)$-root. Since certainly $\Sigma $ is generated by its indecomposable elements,
we have $\Sigma \subseteq \Sigma (U^S(k,m)).$
\end{proof}
\begin{lemma} Let $S$ be a white $(k,m)$-regular set, $t<s$ be respectively white and black
 points on the scheme $(\ref{grb}).$ If $\psi (1+t)$ is not a black point 
$($it is white or does not appear on the scheme at all$)$ then $[1+t:s]$ is a simple
$U^S(k,m)$-root, and $\Phi ^S(1+t,s)\in U^S(k,m).$
\label{sig1}
\end{lemma}
\begin{proof}
By \cite[Lemma 9.5]{Kh08} the element $[1+t:s]$ is indecomposable in $\Sigma .$ Hence
by Lemma \ref{sig} it is a simple $U^S(k,m)$-root. At the same time \cite[Theorem 9.8]{Kh08}
implies $\Phi ^S(1+t,s)\in U^S(k,m).$ 
\end{proof}
\begin{lemma} Let $S$ be a black $(k,m)$-regular set, $t<s$ be respectively white and black
 points on the scheme $(\ref{grb}).$ If $\psi (1+s)$ is not a white point 
then $[1+t:s]$ is a simple $U^S(k,m)$-root, and $\Phi ^S(1+t,s)\in U^S(k,m).$
\label{sig2}
\end{lemma}
\begin{proof}
Similarly by \cite[Lemma 9.6]{Kh08} the element $[1+t:s]$ is indecomposable in $\Sigma .$ Hence
by Lemma \ref{sig} it is a simple $U^S(k,m)$-root, while \cite[Theorem 9.8]{Kh08}
implies $\Phi ^S(1+t,s)\in U^S(k,m).$ 
\end{proof}
\begin{lemma} Let $S$ be a $(k,m)$-regular set. If $t<s$ are respectively white and black
 points on the scheme $(\ref{grb}),$ then  $\Phi ^S(1+t,s)\in U^S(k,m)$ unless $t<n<s.$
\label{sig3}
\end{lemma}
\begin{proof}
Let $S$ be white $(k,m)$-regular. Assume $s\leq n.$ 
The point $\psi (k)$ is not black on the schemes (\ref{grab}), (\ref{grab1}). Hence 
Lemma \ref{sig1} with $t\leftarrow k-1,$ $s\leftarrow s$ implies  $\Phi ^S(k,s)\in U^S(k,m).$
Again by Lemma \ref{sig1} applied to $U^S(k,s)$ we get $\Phi ^S(1+t,s)\in U^S(k,s)\subseteq U^S(k,m).$

Assume $t\geq n.$ The point $n=\psi (n+1)$ is white on the schemes (\ref{grab}), (\ref{grab1}). Therefore 
Lemma \ref{sig1} with $t\leftarrow n,$ $s\leftarrow m$ implies  $\Phi ^S(1+n,m)\in U^S(k,m).$
Again by Lemma \ref{sig1} applied to $U^S(1+n,m)$ we get $\Phi ^S(1+t,s)\in U^S(1+n,m)\subseteq U^S(k,m).$

If $S$ is black $(k,m)$-regular, then we may apply Lemma \ref{sig2} in a similar way or  just use
the duality given in Proposition \ref{xn0}.
\end{proof}

\section{Necessary condition} 

Let $U^-\supseteq F$ and $U^+\supseteq G$ be right coideal subalgebras of respectively 
negative and positive quantum Borel subalgebras. As we mentioned in the above section 
$U^+$ is generated  as algebra by $G$ and elements of the form $\Phi ^{S}(k,m)$
with $(k,m)$-regular sets $S.$ Respectively $U^-$ is generated 
as algebra by $F$ and  elements of the form $\Phi ^{T}_-(i,j)$ with $(i,j)$-regular sets $T.$
Here $\Phi ^{T}_-(i,j)$  appears from $\Phi ^{T}(i,j)$
given in (\ref{dhs}) under the substitutions $x_t\leftarrow x_t^-,$ $1\leq t\leq 2n.$

To state a necessary condition for tensor product (\ref{tru}) to be a subalgebra  
we display the regular generators  $\Phi ^{S}(k,m)$ and $\Phi ^{T}_-(i,j)$ 
graphically as defined in (\ref{grb}):
\begin{equation}
\begin{matrix} S \ \ 
\stackrel{k-1}{\circ } \ & \cdots \ & \stackrel{i-1}{\bullet } 
\ & \stackrel{i}{\bullet }\ \ & \stackrel{i+1}{\circ }\ & \cdots &
\ & \stackrel{m}{\bullet } \ & \ & \  \cr
T \ \ \ \ \ \ 
\ \ & \  \ & \circ  
\ & \circ \ \ & \bullet \ & \cdots &
\ & \bullet  \ & \cdots  \ & \stackrel{j}{\bullet } 
\end{matrix}\ \ \ .
\label{grr1}
\end{equation} 
We shall call this scheme  a $S_k^mT_i^j$-{\it scheme}. Sometimes in this notation we omit those of the indices 
that are fixed in the context. For example if $k,m,i,j$ are fixed, this is a $ST$-scheme.
Lemma \ref{xn0} shows that the element $\Phi ^{S}(k,m)$ up to a scalar factor equals the element   
$\Phi ^{\overline{\psi(S)-1}}(\psi (m),\psi (k))$ that has essentially  different representation (\ref{grb}). 
By this reason to the pair 
$\Phi ^{S}(k,m),$ $\Phi ^{T}_-(i,j)$ we may associate three more schemes:
\begin{equation}
\begin{matrix} S \ \ 
\stackrel{k-1}{\circ } \ & \cdots \ & \stackrel{\psi(j)-1}{\bullet } 
\ & \stackrel{\psi (j)}{\bullet }\ \ & \stackrel{\psi (j)+1}{\circ }\ & \cdots &
\ & \stackrel{m}{\bullet } \ & \ & \ \cr
{T^*} \ \ \ \ \ 
\ \ & \  \ & \circ  
\ & \bullet \ \ & \bullet \ & \cdots &
\ & \circ  \ & \cdots  \ & \stackrel{\psi (i)}{\bullet }
\end{matrix}\ \ \ .
\label{grr2}
\end{equation}
Here $T^*$ is the set $\overline{\psi ({T})-1},$ the complement 
of $ \{ \psi (t)-1\, |\, t\in {T}\} $ with respect to $[\psi (j),\psi (i)).$ 
By definition this is the $S_k^mT^{*\psi (i)}_{\psi (j)}$-scheme,
or shortly the $ST^*$-scheme.
\begin{equation}
\begin{matrix} {S^*} \ \ \ \ \ \ & \ &
\stackrel{\psi (m)-1}{\circ } & \cdots \ &\stackrel{j-2}{\bullet } & \stackrel{j-1}{\circ  } 
& \stackrel{j}{\circ  }&  \cdots
& \stackrel{\psi (k)}{\bullet } \cr
{T} \ \  \stackrel{i-1}{\circ } & \cdots 
& \bullet &  \cdots &\circ  & \circ & \bullet 
\end{matrix}\ \ \ .
\label{grr3}
\end{equation}
Here  $S^*$ is the set $\overline{\psi (S)-1},$ the complement 
of $ \{ \psi (s)-1\, |\, s\in {S}\} $ with respect to $[\psi (j),\psi (i)).$
By definition this is the $S_{\psi (m)}^{*\psi (k)}T^j_i$-scheme, or shortly the $S^*T$-scheme.
\begin{equation}
\begin{matrix} { S^*} \ \ \ \ \ \ \  \ \ \  \ & \ & \stackrel{\psi (m)-1}{\circ } & \cdots &
 \stackrel{\psi (i)-1}{\circ }\ 
\ & \stackrel{\psi (i)}{\circ } \ & \cdots & \stackrel{\psi (k)}{\bullet } \cr
{T^*} \ \ \stackrel{\psi (j)-1}{\circ }& \cdots & 
 \circ  & \cdots & \bullet \ \ & \bullet \
\end{matrix}\ \ \ .
\label{grr4}
\end{equation}
Again by definition this is the $S_{\psi (m)}^{*\psi (k)}T^{*\psi (i)}_{\psi (j)}$-scheme, or shortly the $S^*T^*$-scheme.

\begin{definition} \rm
A scheme  is said to be {\it balanced} if it has no fragments of the form  
\begin{equation}
\begin{matrix}
\stackrel{t}{\circ } \ & \cdots & \stackrel{s}{\bullet } \cr
\circ  
\ & \cdots  & \bullet 
\end{matrix}\ \ \ .
\label{gra2}
\end{equation}
\label{bal}
\end{definition}
\begin{theorem} 
Consider the triangular decomposition of a right coideal subalgebra given in Theorem $\ref{raz2}$
\begin{equation}
 U=U^-\otimes _{{\bf k}[F]} {\bf k}[H] \otimes _{{\bf k}[G]}U^+.
\label{trus}
\end{equation}
If  $\Phi ^{S}(k,m),$  $\Phi ^{T}_-(i,j)$ are the regular generators respectively of $U^+$ and $U^-$ defined by simple 
roots $[k:m]$ and $[i:j]^-,$ then either  
all four schemes $(\ref{grr1}-\ref{grr4})$ defined by this pair are balanced, 
or  one of them has the form
\begin{equation}
\begin{matrix}
\stackrel{t}{\circ } \ & \cdots & \circ & \cdots & \bullet & \cdots & \stackrel{s}{\bullet } \cr
\circ  
\ & \cdots  &  \bullet & \cdots & \circ & \cdots &  \bullet 
\end{matrix}\ \ \ ,
\label{gra3}
\end{equation}
where no one intermediate column has points of the same color.
\label{bale}
\end{theorem}
 The next lemma shows that to see
that a given pair satisfies the conclusion of the theorem it is sufficient to check
 just two first schemes  (\ref{grr1}), (\ref{grr2}).
\begin{lemma} 
$ST$-Scheme $(\ref{grr1})$ is balanced if and only if  so is 
$S^*T^*$- scheme $(\ref{grr4}).$ Similarly $ST^*$-scheme
$(\ref{grr2})$ is balanced if and only if  so is $S^*T$-scheme $(\ref{grr3}).$
$ST$-Scheme $(\ref{grr1})$ has the form $(\ref{gra3})$
 if and only if  so does $S^*T^*$-scheme $(\ref{grr4}).$ Respectively  $ST^*$-scheme
$(\ref{grr2})$ has the form $(\ref{gra3})$ if and only if  so does $S^*T$-scheme $(\ref{grr3}).$
\label{bal2}
\end{lemma}
\begin{proof}
Consider a transformation $\rho $ of schemes  that moves a point $a$ to $\psi (a)-1$ and changes the color.
This transformation maps  $ST$-scheme to $S^*T^*$-scheme 
and $ST^*$-scheme to $S^*T$-scheme. At the same time it changes the order
of columns. 
In particular  the fragment of the form (\ref{gra2}) transforms to a fragment of the same form 
with $t\leftarrow \psi (s)-1,$ $s\leftarrow \psi (t)-1.$
\end{proof}

\section{Additional relations}
In this and the next  technical sections 
we are going to describe  two important cases when $\left[ \Phi^{S}(k,m),\Phi^{T}_-(i,j)\right ]$
belongs to ${\bf k}[H].$ The first one (Theorem \ref{des1}) is the case 
when $ST$-scheme has the form (\ref{gra3}),
while the second one (Theorem \ref{str}) provides conditions when this bracket equals zero.

We fix the following notations. Let $h_i$ denote  
$g_if_i\in H,$ while  $g_{k\rightarrow m}$ is the product $g_kg_{k+1}\ldots g_{m},$
respectively $f_{k\rightarrow m}$ $=f_kf_{k+1}\ldots f_{m},$ and
$h_{k\rightarrow m}$ $=g_{k\rightarrow m}f_{k\rightarrow m}.$
In the same way $\chi ^{k\rightarrow m}$ $=\chi ^{k}\chi ^{k+1}\ldots \chi ^{m}.$
Similarly $P_{k\rightarrow m,i\rightarrow j}$ is
$\chi ^{k\rightarrow m}(g_{i\rightarrow j})$ $=\chi ^{i\rightarrow j}(f_{k\rightarrow m}).$
Of course we have
$P_{k\rightarrow m,i\rightarrow j}$ $=P_{\psi(m)\rightarrow \psi (k),\psi (j)\rightarrow \psi (i)}.$
In these notations Definition 4.4 takes the form 
$\sigma _k^m$ $=P_{k\rightarrow m,k\rightarrow m};$ 
$\mu _k^{m,i}$ $=P_{k\rightarrow i,i+1\rightarrow m}$ $\cdot P_{i+1\rightarrow m,k\rightarrow i}.$

\begin{theorem} 
If $S$ is a $(k,m)$-regular set then
$$
\left[ \Phi^{S}(k,m), \Phi^{\overline{S}}_-(k,m)\right] \sim 1-h_{k\rightarrow m},
$$
where  $\overline{S}$ is a complement of $S$ with respect to the interval $[k,m).$
\label{des1}
\end{theorem}
\begin{proof}
We use induction on $m-k.$ If $m=k,$ the statement is clear. 

Suppose firstly that $n\notin [k,m).$ In this case each set is both black and white $(k,m)$-regular. 
Hence by Lemma \ref{xn1} and Lemma \ref{xn2} with $t=m-1$ we have
$$
\Phi^{S}(k,m)\sim \left\{ 
\begin{matrix}
[ \Phi^{S}(k,m-1),x_m]\hfill & \mbox{ if } m-1\notin {S};  \hfill \cr
[x_m,\Phi^{S}(k,m-1)] \hfill & \mbox{ if } m-1\in {S},\hfill 
\end{matrix}
\right. 
$$
and
$$
\Phi^{\overline{S}}_-(k,m)\sim \left\{ 
\begin{matrix}
[x_m^-,\Phi^{\overline{S}}_-(k,m-1)]
\hfill & \mbox{ if } m-1\notin {S};  \hfill \cr
[ \Phi^{\overline{S}}_-(k,m-1),x_m^-] \hfill & \mbox{ if } m-1\in {S}.\hfill 
\end{matrix}
\right. 
$$
Let us fix for short the following designations:  $u=\Phi^{S}(k,m-1),$ $v^-=\Phi^{\overline{S}}_-(k,m-1).$
By the inductive supposition we have $[u,v^-]=\alpha (1-h_{k\, m-1}),$ $\alpha \neq 0.$
Consider the algebra ${\mathfrak F}_2$ defined by the quantum variables $z_1, z_2$ with
$g_{z_1}={\rm gr}(u)=g_{k\rightarrow m-1},$ $\chi ^{z_1}=\chi ^u,$ $g_{z_2}=g_m,$ $\chi ^{z_2}=\chi ^m,$
and respectively
$g_{z_1^-}={\rm gr}(v^-)=f_{k\rightarrow m-1},$ $\chi ^{z_1^-}=(\chi ^u)^{-1},$ $g_{z_2^-}=f_m,$ 
$\chi ^{z_2^-}=(\chi ^m)^{-1}.$
Since due to Lemma \ref{suu} we have $[u,x_m^-]=[x_m,v^-]=0,$ the map
$z_1\rightarrow u,$ $z_2\rightarrow x_m,$ $z_1^-\rightarrow \alpha ^{-1}v^-,$
$z_2^{-}\rightarrow x_m^-$ has an extension up to a homomorphism 
of algebras.  Hence by Lemma \ref{suu1} we have 
$[[u,x_m],[x_m^-,v^-]]=\varepsilon (1-h_{k\rightarrow m}),$ where the coefficient 
$\varepsilon =(1-p(z_1,z_2)p(z_2,z_1))$  equals $1-q^{-2},$
for $p(z_1,z_2)$ $=p(u,x_m)$ $=p_{km}p_{k+1\, m}\ldots p_{m-1\, m}$
and $p(z_2,z_1)$ $=p(x_m, u)$ $=p_{mk}p_{m\, k+1}\ldots p_{m\, m-1}.$
Since conditions of Lemma \ref{suu1} are invariant under the substitution $i\leftrightarrow j,$
we have also $[[x_m,u],[v^-,x_m^-]]=\varepsilon (1-h_{k\rightarrow m}),$ which is required.

Now consider the case $n\in [k,m).$ Suppose that $S$ is white $(k,m)$-regular and 
$m<\psi (k).$ In this case $\overline{S}$ is black $(k,m)$-regular. Let $t$ denote
the first white point next in order to $\psi (m)-1.$ Since $n$ is a white point, we have $t\leq n.$
\begin{equation}  
\begin{matrix}
 \ \ \ \ \ \ \ & \stackrel{m}{\bullet }&\circ &\circ &\stackrel{\psi (t)}{\circ }& \stackrel{\psi (t)-1}{*}  & \cdots & \stackrel{n}{\circ }   \Leftarrow \cr 
 \stackrel{k-1}\circ \ldots & \stackrel{\psi (m)-1}{\circ }& \bullet & \bullet &\stackrel{t-1}{\bullet } & \stackrel{t}{\circ } & \cdots & \stackrel{n}{\circ }   \hfill
\end{matrix}
\label{gb1}
\end{equation}
The set $S\cup \{ \psi (t)-1\} $ is white $(k,m)$-regular, unless $\psi (t)-1=n.$
 Hence by Lemma \ref{xn1} and Lemma \ref{xn2} we have 
$$
\Phi^{S}(k,m)\sim \left\{ 
\begin{matrix}
[ \Phi^{S}(k,\psi (t)-1),\Phi^{S}(\psi (t),m)]\hfill & 
\mbox{ if } \psi (t)-1\notin {S}\cup \{ n\};  \hfill \cr
[\Phi^{S}(\psi (t),m),\Phi^{S}(k, \psi (t)-1)] \hfill & 
\mbox{ if } \psi (t)-1\in {S}\cup \{ n\} .\hfill 
\end{matrix}
\right. 
$$
Similarly  $\overline{S}\setminus \{ \psi (t)-1\}$ is black $(k,m)$-regular, unless
$\psi (t)-1=n.$
The condition  $\psi (t)-1\notin \overline{S}\setminus \{ n\}$ is equivalent 
to $\psi (t)-1\in {S}\cup \{ n\}.$ Hence these lemmas imply also
$$
\Phi^{\overline{S}}_-(k,m)\sim \left\{ 
\begin{matrix}
[\Phi^{\overline{S}}_-(\psi (t),m),\Phi^{\overline{S}}_-(k,\psi (t)-1)]
\hfill & 
\mbox{ if } \psi (t)-1\notin {S}\cup \{ n\};  \hfill \cr
[ \Phi^{\overline{S}}_-(k,\psi (t)-1),\Phi^{\overline{S}}_-(\psi (t),m)] \hfill & 
\mbox{ if } \psi (t)-1\in {S}\cup \{ n\} .\hfill 
\end{matrix}
\right. 
$$
Let us fix for short the following designations:  $u=\Phi^{S}(k,\psi (t)-1),$ $v=\Phi^{S}(\psi (t),m),$
$w^-=\Phi^{\overline{S}}_-(k,\psi (t)-1),$ $y^-=\Phi^{\overline{S}}_-(\psi (t),m).$
By the inductive supposition we have 
\begin{equation}  
[u,w^-]=\alpha (1-h_{k\rightarrow \psi (t)-1}), \ \ [v,y^-]=\beta (1-h_{\psi (t)\rightarrow m}),
\label{no1}
\end{equation}
where $\alpha\neq 0,$ $\beta \neq 0.$

{\bf Assume $t\neq n$} (equivalently, $\psi (t)-1\neq n$). In this case $u$ and $w^-$ have 
further decompositions according to Lemmas \ref{xn1}, \ref{xn2}:
\begin{equation}  
u=[\Phi^{S}(n+1,\psi (t)-1),\Phi^{S}(k, n)],
w^-=[ \Phi^{\overline{S}}_-(k,n),\Phi^{\overline{S}}_-(n+1,\psi (t)-1)].
\label{no2}
\end{equation}
Moreover, $S$ and $\overline{S}$ are both black and white $(k,n)$-regular.
Since $\psi (m)-1,$ $t$ are white points for $S$ and black points for $\overline{S}$, we have
$$
\Phi^{S}(k, n)=[[a_1,a_2],a_3], \ \ 
\Phi^{\overline{S}}_-(k,n)=[b_3^-,[b_2^-,b_1^-]],
$$
where $a_1=\Phi^{S}(k, \psi (m)-1),$ $a_2=\Phi^{S}(\psi (m), t),$
$a_3=\Phi^{S}(t+1, n),$ and similarly 
$b_1^-=\Phi^{\overline{S}}_-(k, \psi (m)-1),$  
$b_2^-=\Phi^{\overline{S}}_-(\psi (m), t),$ $b_3^-=\Phi^{\overline{S}}_-(t+1, n).$
All points of the interval $[\psi (m),t)$ are black for $S$
(of course if $t=\psi (m),$ then this interval is empty). Hence all points of the interval
$[\psi (t),m)$ are white (otherwise $S$ is not white $(k,m)$-regular). In particular 
\begin{equation} 
v=\Phi^{S}(\psi (t),m)=\Phi^{\emptyset }(\psi (t),m)=u[\psi (t),m].
\label{gb2}
\end{equation}
At the same time, using Lemma \ref{xn0}, we have
\begin{equation}  
a_2=\Phi^{S}(\psi (m),t)=
\Phi^{[\psi (m),t) }(\psi (m),t)\sim \Phi^{\emptyset }(\psi (t),m)=u[\psi (t),m].
\label{gb3}
\end{equation}
Hence by (\ref{no1}) we have $[a_2,y^-]\sim [v,y^-]\sim 1-h_{\psi (t)\rightarrow m}.$
Lemma \ref{suu} implies
$$
0=[a_1,y^-]=[a_3,y^-]=[\Phi^{S}(n+1,\psi (t)-1),y^-].
$$
Therefore 
$$
[\Phi^{S}(k, n),y^-]=[[[a_1,a_2],a_3],y^-]\stackrel{(\ref{uno})}{\sim }
[[[a_1,a_2],y^-],a_3]
$$
$$
\stackrel{(\ref{jak3})}{=}[[a_1,[a_2,y^-]],a_3]
\stackrel{(\ref{cuq3})}{\sim } [a_1,a_3]=0,
$$
for $a_1,a_3$ are separated in $U_q^+(\mathfrak{so}_{2n+1}).$ Thus,
(\ref{no2}) implies $[u,y^-]=0.$

In perfect analogy we have $[v,w^-]=0.$ 
Consider the algebra ${\mathfrak F}_2$ defined by quantum variables $z_1, z_2$ with
$g_{z_1}={\rm gr}(u)=g_{k\rightarrow \psi (t)-1},$ $\chi ^{z_1}=\chi ^u,$ 
$g_{z_2}={\rm gr}(v)=g_{\psi (t)\rightarrow m},$ $\chi ^{z_2}=\chi ^v,$
and respectively
$g_{z_1^-}={\rm gr}(w^-)=f_{k\rightarrow \psi (t)-1},$ $\chi ^{z_1^-}=(\chi ^u)^{-1},$ 
$g_{z_2^-}={\rm gr}{(y^-)}=f_{\psi (t)\rightarrow m},$ 
$\chi ^{z_2^-}=(\chi ^v)^{-1}.$
Due to (\ref{no1}) and  $[u,y^-]=[v,w^-]=0,$ the map
$z_1\rightarrow u,$ $z_2\rightarrow v,$ $z_1^-\rightarrow \alpha ^{-1}w^-,$
$z_2^{-}\rightarrow \beta ^{-1}y^-$ has an extension up to a homomorphism 
of algebras.  Hence by Lemma \ref{suu1} we have 
$[[u,v],[y^-,w^-]]=\varepsilon (1-h_{k\rightarrow m}),$ where the coefficient 
$\varepsilon $  equals $1-q^{-2},$
for $p(z_1,z_2)p(z_2,z_1)$ $=p(u,v)p(v,u)$ $=\mu^{m,\psi (t)-1}_k=q^{-2}$ due to (\ref{mu2}).
Conditions of Lemma \ref{suu1} are invariant under the substitution $i\leftrightarrow j.$ Hence
we have also $[[v,u],[w^-,y^-]]=\varepsilon (1-h_{k\rightarrow m}),$ which proves the required relation 
for $t\neq n.$

{\bf Assume $t=n.$} In this case $\Phi^{S}(k,m)=[v,u],$ 
$\Phi^{\overline{S}}_-(k, m)=[w^-,y^-]$ and we have
$$
u= \Phi^{S}(k,n)=[a_1,b_1], \ \ w^-=\Phi^{\overline{S}}_-(k,n)=[b_2^-,b_1^-],
$$
where $a_1=\Phi^{S}(k,\psi (m)-1),$ $a_2=\Phi^{S}(\psi (m),n),$
and $b_1^-=\Phi^{\overline{S}}_-(k, \psi (m)-1),$ 
$b_2^-=\Phi^{\overline{S}}_-(\psi (m),n).$ 
Equalities (\ref{gb2}) and (\ref{gb3}) with $t\leftarrow n$ show that $a_2\sim v.$
Hence $\Phi^{S}(k,m)=[v,u]\sim [v,[a_1,v]],$ while $[v,[a_1,v]]\sim [[a_1,v],v]$
due to conditional identity (\ref{bri}), for $p(a_1v,v)p(v,a_1v)=\mu ^{m,n}_k=1,$ see (\ref{mu2}).
Similarly 
$$
\Phi^{\overline{S}}_-(k, m)=[w^-,y^-]\sim [[b_2^-,b_1^-],y^-]\sim [[y^-,b_1],y^-]
\sim [y^-,[y^-,b_1^-]].
$$

Consider the algebra ${\mathfrak F}_2$ defined by quantum variables $z_1, z_2$ with
$g_{z_1}={\rm gr}(a_1)=g_{k\rightarrow \psi (m)-1},$ $\chi ^{z_1}=\chi ^{a_1},$ 
$g_{z_2}={\rm gr}(v)=g_{n+1\rightarrow m},$ $\chi ^{z_2}=\chi ^v,$
and respectively
$g_{z_1^-}={\rm gr}(b_1^-)=f_{k\rightarrow \psi (m)-1},$ $\chi ^{z_1^-}=\chi ^{b_1^-}$ 
$=(\chi ^{a_1})^{-1},$ $g_{z_2^-}={\rm gr}{(y^-)}=f_{n+1\rightarrow m},$ 
$\chi ^{z_2^-}=\chi ^{y^-}$ $=(\chi ^v)^{-1}.$
By the considered above case $``n\notin [k,m)"$ we have $[a_1,b_1^-]=\gamma (1-h_{k,\psi (m)-1}).$
Since Lemma \ref{suu} implies $[a_1,y^-]$ $=[v,b_1^-]=0,$ the map
$z_1\rightarrow a_1,$ $z_2\rightarrow v,$ $z_1^-\rightarrow \gamma ^{-1}b_1^-,$
$z_2^{-}\rightarrow \beta ^{-1}y^-$ has an extension up to a homomorphism 
of algebras.  Hence by Lemma \ref{suu2} we have 
$[[[a_1,v],v], [y^-,[y^-,b_1^-]]]=\varepsilon (1-h_{k\rightarrow m}).$

It remains to note that $\varepsilon \neq 0.$ Definition (\ref{mu11})  implies
$p(z_2,z_2)=p(v,v)=\sigma _{n+1}^m,$ while (\ref{mu21}) shows that 
$\sigma _{n+1}^m=q.$
Further, $p(z_1,z_2)p(z_2,z_1)=p(a_1,v)p(v,a_1)=\mu _k^{n,\psi (m)-1}=q^{-2},$
see (\ref{mu1}), (\ref{mu2}).
Hence $\varepsilon =(1+q)(1-q^{-2})(1-q^{-1})\neq 0.$ This completes the proof of
the case ``$m<\psi (k),$ $S$ is white $(k,m)$-regular".

If $S$ is black $(k,m)$-regular and still $m<\psi (k),$ 
then by Lemma \ref{dop} the set $\overline{S}$ is white $(k,m)$-regular.
Hence $\left[ \Phi^{\overline{S}}(k,m), \Phi^{S}_-(k,m)\right] \sim 1-h_{k\rightarrow m}.$
Let us apply $h_{k\rightarrow m}\sigma ,$ where $\sigma $ is the antipode, to this equality.
By (\ref{ant1}) and (\ref{desc3}) we have
$\left[ \Phi^{\overline{S}}_-(k,m), \Phi^{S}(k,m)\right] \sim 1-h_{k\rightarrow m}.$
It remains to apply antisymmetry (\ref{dos}).

If $m>\psi (k)$ then Lemma \ref{xn0} reduces consideration to the case ``$m<\psi (k).$"
\end{proof}

\begin{corollary} If $k\leq m\neq \psi (k),$ then 
in the algebra $U_q(\mathfrak{so}_{2n+1})$ we have 
\begin{equation}  
[u[k,m],u[\psi(m),\psi (k)]^-]\sim \, 1-h_{k\rightarrow m}.
\label{kgb1}
\end{equation}
\label{dus1}
\end{corollary}
\begin{proof}
Proposition \ref{xn0} with $S=\emptyset $ applied to the mirror generators 
implies $u[\psi(m),\psi (k)]^-\sim \Phi ^{[k,m)}_-(k,m).$
Hence Theorem \ref{des1} works.
\end{proof}

\section{Pairs with strong schemes}
In this section we determine when $\left[ \Phi^{S}(k,m),\Phi^{T}_-(i,j)\right ]$ equals zero.
Let us consider firstly the case $S=T=\emptyset .$
\begin{proposition} Let $i\neq k,$ $j\neq m,$ $k\leq m,$ $i\leq j.$
If $\psi (m), \psi (k) \notin [i,j]$ or, equivalently, $k,m\notin [\psi (j),\psi (i)],$ 
then in $U_q(\mathfrak{so}_{2n+1})$ we have
$$
[u[k,m],u[i,j]^-]=0.
$$
\label{ruk4}
\end{proposition}
\begin{proof}  
If $m=\psi (k),$ then conditions $\psi (m), \psi (k)\notin [i,j]$ certainly imply 
$k,$ $m,$ $\psi (m),$ $\psi (k)\notin [i,j],$ and one may use Corollary \ref{ruk3}.
If $j=\psi (i),$ then $\psi (t)\notin [i,j]$ if and only if $t\notin [i,j].$ Hence again
$\psi (m), \psi (k)\notin [i,j]$ implies $k,m,\psi (m), \psi (k)\notin [i,j],$ and Corollary \ref{ruk3} applies.
Thus, further we may assume $m\neq \psi (k),$ $j\neq \psi (i).$

We shall use induction on the parameter $m-k+j-i.$ If either $m=k$ or $j=i,$ then 
the statement follows from (\ref{kom1}) and (\ref{kom2}). Assume $k<m,$ $i<j.$
Condition $\psi (m),\psi (k)\notin [i,j]$ holds if and only if one of the following two
options is fulfilled:

{\bf A}. $\psi (m)<i<j<\psi (k);$

{\bf B}. $\psi (m)<\psi (k)<i< j,$ or $i<j<\psi (m)<\psi (k);$

Let us consider these options separately.

\smallskip
{\bf A}. Since $\psi (k),$ $\psi (m)\notin [i,j],$  by
Corollary (\ref{ruk3}) we may suppose that either $k\in [i,j]$ or $m\in [i,j].$
The option {\bf A} is equivalent to $k<\psi (j)<\psi (i)<m,$ for $\psi $ changes the order. 

By Proposition \ref{ins2}  with  $i\leftarrow \psi (j)-1$
 we have $u[k,m]=[u[k,\psi (j)-1], u[\psi (j),m]].$ Indeed,
 the exceptional equality $\psi (j)-1=\psi (m)-1$ implies a contradiction $j=m.$   
The exceptional equality $\psi (j)-1=\psi (k)$ implies $j=k-1,$  hence
$j<k<m,$ in particular $k,m\notin [i,j].$

Similarly, Proposition \ref{ins2} with $k\leftarrow \psi (j),$ $i\leftarrow \psi (i)$
shows that $u[\psi (j),m]=[u[\psi (j),\psi (i)],u[\psi (i)+1,m]].$
Indeed, we have $m\neq \psi (\psi (j))=j,$ and $\psi (i)\neq \psi (\psi (j))=j.$ 
The remaining condition, $\psi (i)\neq \psi (m)-1,$ is also valid  since otherwise
$i=m+1,$ and again $k,m\notin [i,j],$ and again Corollary \ref{ruk3} applies.

Let us fix the for short following designations:  $u=u[k,\psi (j)-1],$ $v=u[\psi (j),\psi (i)],$ $w=u[\psi (i)+1,m],$
$z^-=u[i,j]^-.$ Corollary \ref{dus1} implies $[v,z^-]\sim 1-h_v.$
Proposition \ref{NU} with $i\leftarrow \psi (j)-1,$ $j\leftarrow \psi (i)$ shows that
$[u,w]=0,$ for $m=\psi (\psi (j)-1)-1$ is equivqlent to $m=j,$ while $\psi (i)=\psi (k)$
is equivalent to $i=k.$ Using (\ref{uno}) we have 
\begin{equation}
[u[k,m],u[i,j]^-]=[[u,[v,w]],z^-]=[u,[[v,w],z^-]]+p_{z,vw}[[u,z^-],[v,w]].
\label{fz}
\end{equation}
If $j\neq n,$ then $\psi (j)-1\neq j,$ and still $\psi (k)\neq \psi (i)-1.$
Hence by the inductive supposition with $m\leftarrow \psi (j)-1$
we have $[u,z^-]=0.$  If $i\neq n+1,$ then $\psi (i)+1\neq i,$ 
and still $m\neq \psi (\psi (i)+1)=i-1.$
Hence by the inductive supposition with $k\leftarrow \psi (i)+1$ we have $[w,z^-]=0.$
Thus for $i\neq n+1,$ $j\neq n$ we may continue (\ref{fz}):
\begin{equation}
\sim [u,[[v,z^-],w]] \stackrel{(\ref{cuq4})}{\sim } [u,h_v\cdot w]\stackrel{(\ref{cuq1})}{\sim } h_v[u,w]=0.
\label{fz1}
\end{equation}

Suppose that $j=n.$ In this case we have $k\in [i,j]=[i,n],$ for $m>\psi (j)=n+1>j.$
Moreover $i\neq n+1,$ for $i<j=n.$ Hence still $[w,z^-]=0,$
and the first addend in (\ref{fz}) is zero (see arguments in (\ref{fz1})).

By additional induction on $t-k$ we shall prove the following equation:
\begin{equation}
[u[k,t],u[i,t]^-]=\sum _{b=k-1}^{t-1} \alpha _bu[i,b]^-\cdot u[k,b],
\label{zz2}
\end{equation}
where $i<k\leq t\leq n,$ $0\neq \alpha _b\in {\bf k},$ and by definition $u[k, k-1]=1.$

If $t=k,$ the formula follows from (\ref{kom2}). In the general case by the main
 inductive supposition  we have 
$[u[k,t-1],u[i,t]^-]=0,$ for $k\neq i,$ $t-1\neq t;$ 
$\psi (k),$ $\psi (t-1)\notin [i,t];$ and $t-1\neq \psi (k),$ $t\neq \psi (i)$ due to 
$i<k\leq t\leq n.$ Therefore 
$$
[u[k,t],u[i,t]^-]=[[u[k,t-1],x_t], u[i,t]^-]\stackrel{(\ref{jak3})}{=}[u[k,t-1],[x_t,u[i,t]^-]].
$$
Here we would like to apply inhomogeneous substitution (\ref{kom2}) to the right 
factor of brackets. To do this we must fix the coefficient:
$$
\sim u[k,t-1]\cdot u[i,t-1]^--\chi ^{k\rightarrow t-1}(g_tf_{i\rightarrow t})u[i,t-1]^-\cdot u[k,t-1]
$$
$$
=[u[k,t-1],u[i,t-1]^-]+\alpha _{t-1}u[i,t-1]^-\cdot u[k,t-1],
$$
where $\alpha _{t-1} =\chi ^{k\rightarrow t-1}(f_{i\rightarrow t-1})(1-\chi ^{k\rightarrow t-1}(h_t))\neq 0.$
Thus by induction on $t$ we get (\ref{zz2}).

Relation (\ref{zz2}) with $t=n$
takes the form  $[u,z^-]=\sum _{b=k-1}^{n-1} \alpha _bu[i,b]^-\cdot u[k,b].$
Sinse the first addend in (\ref{fz}) is zero, we may continue (\ref{fz}):
$$
\sim \left[ \sum _{b=k-1}^{n-1} \alpha _bu[i,b]^-\cdot u[k,b], [v,w]\right] .
$$
We have seen that $[v,w]=u[n+1,m].$ At the same time
$[u[k,b], u[n+1,m]]=0$ by Proposition \ref{NU} with $i\leftarrow b,$
$j\leftarrow n.$ Indeed, $m\neq \psi (b)-1$ since $m>\psi (i)$ and 
$i<k\leq b$ implies $\psi (i)>\psi (b),$ while $n\neq \psi (k)$ since $k\in [i,n].$
It remains to note that  $[u[i,b]^-,u[n+1,m]]\stackrel{(\ref{dos})}{\sim }[u[n+1,m], u[i,b]^-]=0$ 
by the inductive supposition with $k\leftarrow n+1,$ $j\leftarrow b,$
for now $n+1\neq i,$ $m\neq b,$ $\psi (n+1)=n\notin [i,b],$ $\psi (m)\notin [i,b],$
and of course $m\neq \psi (n+1)$ $=n=j,$ $b\neq \psi (i)>n.$

Similarly we consider the case $i=n+1.$ In this case $m\in [i,j]=[n+1,j],$
for $k<\psi (i)=n<i.$ Moreover, $j\neq n,$ for $n+1=i<j.$ Hence still $[u,z^-]=0;$
that is, the second addend in (\ref{fz}) equals zero, and by means of (\ref{uno})
we contimue (\ref{fz}):
\begin{equation}
=[u,[[v,w],z^-]]=[u,[v,[w,z^-]]]+p_{z,w}[u,[[v,z^-],w]].
\label{fz2}
\end{equation}
Arguments in (\ref{fz1}) show that here the second addend is zero. Since $[u,w]=[u,z^-]=0,$
we have $[u,[w,z^-]]=0.$ Hence conditional identity (\ref{jak3}) implies that the first
addend in (\ref{fz2}) equals 
\begin{equation}
[[u,v],[w,z^-]]=[u[k,n],[w,z^-]].
\label{fz3}
\end{equation}

By downward induction on $t$ we shall prove the following equation:
\begin{equation}
[u[t,m],u[t,j]^-]=\sum _{a=t+1}^{\mu +1} \alpha _a h_{t\rightarrow a-1}\, u[a,j]^-\cdot u[a,m],
\label{zz3}
\end{equation}
where $n<t,$ $\mu=\min \{ m,j\} ,$ $0\neq \alpha _a\in {\bf k},$ and by definition $u[\mu +1,\mu]=1.$

If $t=j$ or $t=m$ the equation follows from (\ref{kom1}) and (\ref{kom2}). 
In the general case by the main inductive supposition we have 
$[u[t,m],u[t+1,j]^-]=0,$ for $t\neq t+1,$ $m\neq j,$ $\psi (t), \psi (m)\notin [t+1,j].$
Therefore
$$
[u[t,m],u[t,j]^-]=[u[t,m],[x_t^-,u[t+1,j]^-]]
$$
$$
\stackrel{(\ref{jak3})}{=}[[u[t,m],x_t^-]],u[t+1, j]^-]
\stackrel{(\ref{kom1})}{\sim } [h_t\cdot u[t+1,m],u[t+1,j]^-].
$$
To prove (\ref{zz3}) it remains to apply (\ref{cuq2}) and the inductive supposition
for downward induction. Here the new coefficient $\alpha _{t+1}$ is nonzero
 since $\chi ^{t+1\rightarrow j}(h_t)=q^{-2}\neq 1.$

Equation (\ref{zz3}) with $t=n+1$ implies
\begin{equation}
[u[k,n],[w,z^-]]=\left[ u[k,n], \sum _{a=n+2}^{m+1} \alpha _a h_{n+1\rightarrow a-1}u[a,j]^-\cdot u[a,m] \right] .
\label{fz4}
\end{equation}
By Proposition \ref{NU} with $i\leftarrow n,$
$j\leftarrow a-1$ we have
$[u[k,n], u[a,m]]=0,$  for $m\neq \psi (n)-1=n$ since $m>i=n+1,$ and 
$a-1\neq \psi (k)$ since $a-1\leq m\leq j<\psi (k).$

Due to (\ref{cuq1}) it remains to note that  
$[u[k,n],u[a,j]^-]=0$ 
by the inductive supposition with $m\leftarrow n,$ $i\leftarrow a,$
for now $n\neq j,$ $k\neq a,$ $\psi (k)\notin [a,j],$ $\psi (n)=n+1\notin [a,j],$
and $n\neq \psi (k)>n,$ $j\neq \psi (a)<n.$

\smallskip
{\bf B}. In this case $\psi (k),$ $\psi (m)\notin [i,j],$ hence  by
Corollary (\ref{ruk3}) we may suppose that either $k\in [i,j]$ or $m\in [i,j].$
Application of $\psi $ shows that the option {\bf B} is equivalent to 
\begin{equation}
\psi (j)<\psi (i)<k<m,\, \hbox{ or } \, k<m<\psi (j)<\psi (i).
\label{fz5}
\end{equation}
In particular again due to Corollary (\ref{ruk3})  we may suppose that either 
$i\in [k,m]$ or $j\in [k,m],$ for (\ref{fz5}) implies
$\psi (i), \psi (j)\notin [k,m].$ Since $i\neq k,$ $j\neq m,$
 it remains  to consider two configurations:
$k<i\leq m<j$ and $i<k\leq j<m.$ Moreover, the substitution $i\leftrightarrow k$ $j\leftrightarrow m$
transforms the original conditions {\bf B} to equivalent form (\ref{fz5}). Therefore 
it suffices to consider just one of the above configurations.

Suppose that $k<i\leq m<j.$ In this case Proposition \ref{ins2} with $k\leftarrow i,$
$m\leftarrow j,$ $i\leftarrow i$ shows that $u[i,j]^-=[x_i^-,u[i+1,j]^-],$ unless $i=\psi (j)-1.$
If  $i\neq \psi (j)-1,$ then by the inductive supposition we have $[u[k,m], u[i+1,j]^-]=0,$ for now
$k<i,$ $m\neq j$ and $\psi (k), \psi (m)\notin [i,j]\supset [i+1,j].$ 
Hence by (\ref{jak3}) and (\ref{kom1}) we have
$$
[u[k,m], [x_i^-,u[i+1,j]^-]]]=[u[k,m],x_i^-], u[i+1,j]^-]=\delta _i^m \cdot [u[k,m-1],u[i+1,j]],
$$
for $i\neq k,$ $i\neq \psi (k),$ $i\neq \psi (m),$ see original conditions {\bf B}. At the same time 
if $\delta _i^m\neq 0$ (that is $m=i$), then
$[u[k,m-1],u[i+1,j]]=0$ by Proposition \ref{NU} with $m\leftarrow j,$ $i\leftarrow m-1,$
$j\leftarrow i=m,$ for $j\neq \psi (k),$ $j\neq \psi (m-1)-1=\psi (m),$ $i=m\neq \psi (k)$ due to 
the original conditions {\bf B}. Thus, it remains to check the case $i=\psi (j)-1.$

Equality $i=\psi (j)-1$ with $k<i<m$ imply $k<\psi (j)\leq m,$ this contradicts to (\ref{fz5}).
Hence in this case we have  $i=m.$ Moreover, $k<i$ implies 
$$
\psi (k)>\psi (i)=\psi (\psi (j)-1)=j+1>i=m.
$$
In particular $\psi (k)\neq m-1.$ Hence by Proposition \ref{ins2} with $i\leftarrow m$ we have
$[u[k,m]=[u[k,m-1],x_m].$ 
Corollary \ref{ruk3} implies both $[u[k,m-1],u[i,j]^-]=0$ and $[u[k,m-1],u[i+1,j]^-]=0,$
for $m-1=i-1\notin [i,j]\supset [i+1,j],$ 
$\psi (m-1)$ $=\psi (i-1)$ $=\psi (\psi (j)-2)$ $=j+2\notin [i,j]\supset [i+1,j].$ Thus by (\ref{jak3})
and (\ref{cuq1}) we have
$$
[u[k,m],u[i,j]^-]=[[u[k,m-1],x_m],u[i,j]^-]=[u[k,m-1], [x_i,u[i,j]^-]]
$$
$$
=[u[k,m-1], h_i\cdot u[i+1,j]^-]]\sim h_i \cdot [u[k,m-1], u[i+1,j]^-]]=0.
$$
This completes the proof. 
\end{proof}

\begin{proposition} Let $i\neq k,$ $j\neq m,$ $k\leq m,$ $i\leq j.$
If $\psi (j), \psi (i)\notin [k,m]$ or, equivalently, $i,j\notin [\psi (m),\psi (k)],$ 
then in $U_q(\mathfrak{so}_{2n+1})$ we have
$$
[u[k,m],u[i,j]^-]=0.
$$
\label{ruk5}
\end{proposition}
\begin{proof} 
Substitution  $i\leftrightarrow k,$ $j\leftrightarrow m$  transforms the conditions of 
Proposition \ref{ruk5} to the conditions of Proposition \ref{ruk4}. 
Let us apply Proposition \ref{ruk4} with $i\leftrightarrow k,$ $j\leftrightarrow m$ to the mirror 
generators $y_i=p_{ii}^{-1}x_i^-,$ $y_i^-=-x_i.$ We get $[u[i,j]_y,u[k,m]_y^-]=0.$ 
 However $u[i,j]_y\sim u[i,j]^-,$ $u[k,m]_y^-\sim u[k,m].$ It remains to apply (\ref{dos}).
\end{proof}

\begin{proposition} Let $i\neq k, $ $j\neq m.$  If
 \begin{equation}
\psi (j)\leq k\leq \psi (i)\leq m
\label{rr1} 
\end{equation}
or, equivalently, 
\begin{equation}
\psi (m)\leq i\leq \psi (k)\leq j
\label{rr2} 
\end{equation}
then in $U_q(\mathfrak{so}_{2n+1})$ we have
\begin{equation}
[u[k,m],u[i,j]^-]\, \sim \, h_{k\rightarrow \psi (i)}\, u[\psi (k)+1,j]^-\cdot u[\psi (i)+1,m]
\label{wff1} 
\end{equation}
provided that $\psi (m)\neq i$ or $\psi (k)\neq j.$
Here by definition we set  $u[j+1,j]^-=u[m+1,m]=1.$
\label{fkk1}
\end{proposition}
\begin{proof} 
We note that condition (\ref{rr1}) is equivalent to the condition (\ref{rr2})
 since $\psi $ changes the order. 
Let $u=u[k,\psi (i)],$ $v=u[\psi (i)+1,m],$ $w^-=u[i,\psi (k)]^-,$
$t^-=u[\psi (k)+1,j]^-.$ Of course $v=1$ if $m=\psi (i),$ while $t^-=1$ if $j=\psi (k).$
By Lemma \ref{dus1} we have
\begin{equation}
[u,w^-]=[u[k,\psi (i)], u[i,\psi (k)]^-]\sim 1-h_u,
\label{ze1} 
\end{equation}
while Proposition \ref{ruk4} with $k\leftarrow \psi (i)+1,$ $i\leftarrow \psi (k)+1$
shows that 
\begin{equation}
[v,t^-]=[u[\psi (i)+1,m], u[\psi (k)+1,j]^-]=0,
\label{ze2} 
\end{equation}
for $\psi (m),$ $\psi (\psi (i)+1)=i-1\notin [\psi (k)+1, j]$ due to  (\ref{rr2}).
At the same time Proposition \ref{ruk4} with $m\leftarrow \psi (i),$ $i\leftarrow \psi (k)+1$
implies 
\begin{equation}
[u,t^-]=[u[k,\psi (i)], u[\psi (k)+1,j]^-]=0,
\label{ze3} 
\end{equation}
where $k\neq n+1,$ $j\neq \psi (i).$ 
Indeed, $\psi (\psi (i))=i,$ $\psi (k)$ $\notin [\psi (k)+1,j]$ due to   (\ref{rr2}),
while $k\neq \psi (k)+1$ due to $k\neq n+1.$ Similarly Proposition \ref{ruk5} with 
$k\leftarrow \psi (i)+1,$ $j\leftarrow \psi (k)$ shows that
\begin{equation}
[v,w^-]=[u[\psi (i)+1,m], u[i,\psi (k)]^-]=0, \ \hbox{ if } i\neq n+1,\ m\neq \psi (k),
\label{ze4} 
\end{equation}
for $\psi (\psi (k))=k,$ $\psi (i)$ $\notin [\psi (i)+1, m]$ due to (\ref{rr1}),
while $\psi (i)+1\neq i$ due to $i\neq n+1.$

We shall prove firstly the proposition when the parameters are in the {\bf general position};
that is, when $i,k\neq n+1,$ $i\neq m+1,$ $k\neq j+1,$ $m\neq \psi (k),$ $j\neq \psi (i).$

By Proposition \ref{ins2}  with $i \leftarrow \psi (i)$ we have $u[k,m]=[u,v]$
provided that $m>\psi (i)\neq \psi (m)-1,$ for $\psi (i)\neq \psi (k).$ The same proposition
with $k\leftarrow i,$ $m\leftarrow j,$ $i\leftarrow \psi (k)$ shows that $u[i,j]^-=[w^-,t^-]$
provided that $j>\psi (k)\neq \psi (j)-1.$ In particular 
if in the general position we have additionally $\psi (i)\neq m,$ $\psi (k)\neq j,$ then $u[k,m]=[u,v],$
$u[i,j]^-=[w^-,t^-],$ and all relations (\ref{ze1} --- \ref{ze4}) hold.  Hence we have the required proportions
$$ 
[[u,v],[w^-,t^-]]\stackrel{(\ref{fo1})}{\sim }[[[u,w^-],t^-],v]\stackrel{(\ref{cuq4})}{\sim }
[h_u\cdot t^-,v]\stackrel{(\ref{cuq21})}{\sim } h_ut^-\cdot v.
$$
The omitted coefficient after the application of (\ref{cuq4}) is $\chi ^{t^-}(h_u)-1,$ while 
$$
\chi ^{t^-}(h_u)=\chi^{\psi (k)+1\rightarrow j}_-(h_{k\rightarrow \psi (j)})=\chi _-^{\psi (k)+1\rightarrow j}(h_{i\rightarrow \psi (k)})
=(\mu _i^{j,\psi (k)})^{-1}=q^2\neq 1
$$
due to definition (\ref{mu1}) and relations (\ref{mu2}) and (\ref{mu4}).
Similarly the omitted coefficient after the application of (\ref{cuq21}) is $1-\chi ^v(h_u),$ while 
$$
\chi ^v(h_u)=\chi ^{\psi (i)+1\rightarrow m}(h_{k\rightarrow \psi (i)})=\mu _k^{m,\psi (i)}=q^{-2}\neq 1.
$$
If in the general position we have $\psi (i)=m,$ $\psi (k)\neq j,$ then $u[k,m]=u,$
$u[i,j]^-=[w^-,t^-],$ $[u,t^-]=0.$ Hence we again have the required relation
$$ 
[u,[w^-,t^-]]\stackrel{(\ref{cua})}{\sim }[[u,w^-],t^-]\stackrel{(\ref{cuq4})}{\sim } h_u\cdot t^-
$$
with the omitted coefficient $\chi ^{t^-}(h_u)-1=q^2-1.$ Similarly, if in the general position 
we have $\psi (i)\neq m,$ $\psi (k)=j,$ then $u[k,m]=[u,v],$
$u[i,j]^-=w^-,$ $[v,w^-]=0,$ and 
\begin{equation}
[[u,v],w^-]\stackrel{(\ref{uno})}{\sim }[[u,w^-],v]\stackrel{(\ref{cuq4})}{=} (1-q^{-2})h_u\cdot v.
\label{ze6} 
\end{equation} 
This completes the proof if  $k,m,i,j $ are in the general position.
Consider the exceptional cases.

{\bf 1.  $\bf k=n+1.$} In this case $i\neq n+1,$ for $i\neq k.$ In particular by (\ref{ze4})
we have $[v,w^-]=0.$ Moreover $i\neq m+1,$ for $\psi (j)\leq n+1\leq \psi (i)\leq m$ and 
$\psi (m)\leq i\leq n\leq j$ imply $i\leq n<m.$ Hence $u[k,m]=[u,v]$ if $m\neq \psi (i),$
and $u[k,m]=u$ otherwise. 

{\bf 1.1.} If $j=n,$ then $u[i,j]^-=w^-,$ for $j=\psi (k)=n.$ Moreover we may assume 
$m\neq  \psi (i)$ (otherwise one may apply Lemma \ref{dus1}); that is, $u[k,m]=[u,v].$
Now algebraic manipulations (\ref{ze6}) prove the required relation
\begin{equation}
[u[n+1,m],u[i,n]^-]\sim h_{n+1\rightarrow \psi (i)}\cdot u[\psi (i)+1,m], \ \ \ \psi (i)<m.
\label{ze7} 
\end{equation} 

{\bf 1.2.} Let $j=n+1.$ By definition (\ref{ww}) we have $u[i,j]^-=[u[i,n]^-,x_{n+1}^-],$
and of course $x_{n+1}^-=x_n^-.$ Hence Jacobi identity (\ref{cua}) and (\ref{kom1})
show that $[u[k,m],u[i,j]^-]$ is a linear combination of the following two terms
$$
[[u[n+1,m],u[i,n]^-],x_{n+1}^-], \ \ \ [h_{n+1}\cdot u[n+2,m],u[i,n]^-].
$$
 We claim that the former term equals zero. Indeed, if $\psi (i)=m,$ then by Lemma \ref{dus1} 
we have $[u[n+1,m],u[i,n]^-]\sim1-h_{i\rightarrow n}.$ However 
$$
\chi ^{n+1}(h_{i\rightarrow n})=\chi ^{n}(g_{n-1}f_{n-1}g_nf_n)=p_{n\, n-1}p_{n-1\, n}p_{nn}p_{nn}=1.
$$ 
Hence (\ref{cuq4}) shows that the former term equals zero.   If $\psi (i)<m,$ then by (\ref{ze7}) we have
 $[u[n+1,m],u[i,n]^-]\sim h_{n+1\rightarrow \psi (i)}\cdot u[\psi (i)+1,m].$
Since $\psi (i)\geq k=n+1,$ Lemma \ref{suu} implies $[u[\psi (i)+1,m],x_{n+1}^-]=0.$
At the same time $\chi ^{n+1}(h_{n+1\rightarrow \psi (i)})=\chi ^{n+1}(h_{i\rightarrow n})=1.$
Thus  (\ref{cuq21}) reduces the former term to zero.

To find the latter term we note that
$$
\chi ^{i\rightarrow n}(h_{n+1})=\chi ^{n-1}(g_nf_n)\chi ^{n}(g_nf_n)=p_{n\, n-1}p_{n-1\, n}p_{nn}p_{nn}=1.
$$
Hence by (\ref{cuq21}) the latter term is proportional to $h_{n+1}\cdot [u[n+2,m],u[i,n]^-].$
Since the points $k^{\prime }=n+2,$ $m^{\prime }=m,$ $i^{\prime }=i,$ $j^{\prime }=n$ are
in the general position, we may apply (\ref{wff1}):
$$
[u[n+2,m],u[i,n]^-]\sim h_{n+2\rightarrow \psi (i)}\, u[n,n]^-\cdot u[\psi (i)+1,m],
$$ 
which is required, for $u[n,n]^-=u[n+1,n+1]^-=x_n^-,$ and $h_{n+1}\cdot h_{n+2\rightarrow \psi (i)}=h_{n+1\rightarrow \psi (i)}.$

{\bf 1.3.} Let $j>n+1,$ $i<n.$ By definition (\ref{ww}) we have $u[k,m]=[x_{n+1},u[n+2,m]].$
Relation (\ref{kom2}) shows that  $[x_{n+1}, u[i,j]^-]=0.$ Hence conditional identity (\ref{jak3}) implies
\begin{equation}
[u[k,m],u[i,j]^-]=[x_{n+1}, [u[n+2,m], u[i,j]^-] ].
\label{ze8} 
\end{equation} 
At the same time the points  $k^{\prime }=n+2,$ $m^{\prime }=m,$ $i^{\prime }=i,$ $j^{\prime }=j$ are
in the general position. 
Moreover $i<n$ implies $\psi (i)+1>n+2,$ hence $[ x_{n+1}, u[\psi (i)+1,m] ]=0$
by Lemma \ref{sepp}. This allows us to continue (\ref{ze8}) applying (\ref{cuq1}), (\ref{br1}):
$$
\sim [x_{n+1},h_{n+2\rightarrow \psi (i)}\, u[n,j]^-\cdot u[\psi (i)+1,m]]\sim h_{n+2\rightarrow \psi (i)}\,[x_{n+1}, u[n,j]^-]\cdot u[\psi (i)+1,m],
$$
which is required due to (\ref{kom2}).

{\bf 1.4.} Let $j>n+1,$ $i=n.$ In this case by definition (\ref{ww}) we have $u[i,j]^-=[x_n^-,u[n+1,j]^-].$
Jacobi identity (\ref{cua}) and (\ref{kom1}) show that $[u[k,m],u[i,j]^-]$ is a linear combination of the following two terms
$$
[h_{n+1}u[n+2,m],u[n+1,j]^-], \ \ \ [x_n^-,[u[n+1,m],u[n+1,j]^-]].
$$
Proposition \ref{ruk4} implies $[u[n+2,m],u[n+1,j]^-]=0,$ for both $\psi (n+2)=n-1,$ and $\psi (m)$ are
less than $n+1.$ At the same time 
$$
\chi ^{n+1\rightarrow j}(h_{n+1})=\chi ^{\psi (j)\rightarrow n}(h_{n})=\chi ^{n-1}(g_nf_n)\chi ^{n}(g_nf_n)=p_{n-1\, n}p_{n\, n-1}p_{nn}^2=1.
$$
Hence by (\ref{cuq2}) the first term equals zero. Due to (\ref{zz3}) the second term takes the form
\begin{equation}
\left[ x_n^-, \sum _{a=n+2}^{\mu }\alpha _a\, h_{n+1\rightarrow a-1}\, u[a,j]^-\cdot u[a,m]\right] ,
\label{ze9} 
\end{equation} 
where $\mu =\min \{ j,m\} .$ By Lemma \ref{suu} we have $[x_n^-,u[a,m]]=0$ for all $a.$
At the same time $[x_n^-,u[a,j]^-=0$ for all $a>n+2,$ see Lemma \ref{sepp},
 while $[x_n^-,u[n+2,j]^-=u[n+1,j]^-$ since $x_n=x_{n+1}.$ Hence in (\ref{ze9})
remains just one term that corresponds to $a=n+2.$
By (\ref{cuq1}) and ( \ref{br1}) this term is proportional to
$$
h_{n+1}u[n+1,j]^-\cdot u[n+2,m], 
$$
which coincides with the right hand side of (\ref{wff1}) with $k=n+1,$ $i=n.$

{\bf 2. $\bf k=j+1.$} In this case inequality $\psi (j)\leq k$ reads $\psi (j)\leq j+1,$
or, equivalently,  $2n-j+1\leq j+1;$ that is, $j\geq n.$ If  $j=n,$ then we turn to the considered above case $k=n+1.$
Thus we have to consider just the case $j>n.$ In this case $k=j+1>n+1,$ and $j=k-1<\psi (i)$ since 
by the conditions of the proposition we have $k\leq \psi (i).$

We shall prove firstly by downward induction on $i$ with fixed $j,k$  the following  proportion
\begin{equation}
[u,u[i,j]^-]\sim h_{k\rightarrow \psi (i)}\, u[\psi (k)+1,j]^-.
\label{ze10} 
\end{equation} 
If $i=\psi (k)$ then (\ref{ze10}) follows from (\ref{kom2}). Let $i<\psi (k).$ In this case 
by Proposition \ref{ins2} we have $u=[u[k,\psi (i)-1],x_i],$ for $k>n.$
At the same time Proposition \ref{ruk5} implies 
$[u[k,\psi (i)-1],u[i,j]^-]=0$ since $\psi (j)\leq n<j=k-1<k,$ and $\psi (j), \psi (i)\notin [k,\psi (i)-1].$
Hence conditional identity (\ref{jak3}) with (\ref{kom2}) show that
$$
[u,u[i,j]^-]=[u[k,\psi (i)-1], [x_i, u[i,j]^-]]\sim [u[k,\psi (i)-1], h_i\cdot u[i+1,j]^-].
$$
This relation, after application of (\ref{cuq1}), and the inductive supposition 
imply (\ref{ze10}), for $h_i=h_{\psi (i)},$ $\psi (i)-1=\psi (i+1).$

If $m=\psi (i)$ then $u[k,m]=u,$ while (\ref{ze10}) coincides with the required (\ref{wff1}).
Let $m>\psi (i).$ In this case $u[k,m]=[u,v],$ for $k>n,$ see Proposition \ref{ins2}.
Lemma \ref{suu} shows that $[v,u[i,j]^-]=0;$ indeed, $v=u[\psi (i)+1,m]$ depends only
in $x_s$ with $s<i,$ while $u[i,j]^-$ depends only in $x_s^-$ with $i\leq s\leq n,$ for $j<\psi (i).$
We have 
\begin{equation}
[[u,v],u[i,j]^-]\stackrel{(\ref{uno})}{\sim }[[u,u[i,j]^-],v]\stackrel{(\ref{ze10})}{\sim }[h_{k\rightarrow \psi (i)}\, u[\psi (k)+1,j]^-,v].
\label{ze11} 
\end{equation}
Again by Lemma \ref{suu} we get $[u[\psi (k)+1,j]^-,v]=0.$ Therefore  we may continue 
(\ref{ze11}) applying (\ref{cuq2}):
$$
\sim (1-\chi ^v(h_{k\rightarrow \psi (i)}))\, h_{k\rightarrow \psi (i)}\, u[\psi (k)+1,j]^-\cdot v
$$
which is required since by definition $v=u[\psi (i)+1,m],$ and 
$$
\chi ^v(h_{k\rightarrow \psi (i)})=\chi ^{\psi (i)+1\rightarrow m}(h_{k\rightarrow \psi (i)})=\mu _k^{m, \psi (i)}=q^{-2}\neq 1.
$$ 

{\bf 3. $\bf i=n+1$ or $\bf i=m+1.$} Conditions of the proposition are invariant 
under the transformation  $i\leftrightarrow k,$ $j\leftrightarrow m.$  At the same time this
transformation reduce the condition ``$i=n+1$ or $i=m+1$" to the considered above
cases {\bf 1} or {\bf 2}. Hence for the mirror generators 
$y_i=p_{ii}^{-1}x_i^-,$ $y_i^-=-x_i$ we have 
$$
[u[i,j]_y,u[k,m]_y^-]\, \sim \, h_{k\rightarrow \psi (i)}\, u[\psi (i)+1,m]^-_y\cdot u[\psi (k)+1,j]_y.
$$
However $u[a,b]_y\sim u[a,b]^-,$ $u[a,b]_y^-\sim u[a,b].$ It remains to apply (\ref{dos})
and to note that by Proposition \ref{ruk4} the factors in the right hand side of (\ref{wff1})
skew commute each other, for $\psi (m)\leq i\leq \psi (k)\leq j$
implies $\psi (m),\psi (\psi (i)+1)=i-1\notin [\psi (k)+1,j].$ 

{\bf 4. $\bf j=\psi (i).$}
 If also $m=\psi (k)$ then (\ref{rr1}) reads $i\leq k\leq \psi (i)\leq \psi (k),$
where the first and the last inequalities are not consistent provided that $i\neq k.$
Hence we assume $m\neq \psi (k).$ Denote for short
$$
u=u[k,m], \ \ v^-=u[n+1,j]^-,\ \ w^-=u[i,n]^-.
$$
By definition (\ref{ww}) we have $u[i,j]^-\sim [v^-,w^-].$

If $k\leq n$ then $\psi (k)\notin [i,n].$ We have also $\psi (m)\notin [i,n],$
for Eq. (\ref{rr1}) with $m\neq j=\psi (i)$ imply $\psi (m)<i.$ Hence by Proposition \ref{ruk4} with $j\leftarrow n$
we have $[u,w^-]=0.$ At the same time 
$\psi (j)\leq k\leq\psi (n+1)\leq m,$ $\psi (n+1)\neq j.$ Therefore already proved  case of the proposition
 with $i\leftarrow n+1$ implies
$$
[u,v^-]\sim h_{k\rightarrow n}u[\psi (k)+1,j]^-\cdot u[n+1,m].
$$ 
Taking into account Jacobi identity (\ref{cua}) we have 
\begin{equation}
[u,[v^-,w^-]]=[[u,v^-],w^-]\sim [h_{k\rightarrow n}u[\psi (k)+1,j]^-\cdot u[n+1,m],w^-].
\label{ewz2}
\end{equation}
The second statement of Proposition \ref{NU} with $k\leftarrow i,$ $i\leftarrow n,$ $j\leftarrow \psi (k),$ $m\leftarrow j$
implies $[u[\psi (k)+1,j]^-,w^-]=0.$ Indeed, the conditions of Proposition \ref{NU} under that replacement are:
$j\neq \psi (n)-1,$ $\psi (k)\neq \psi (i),$ and $n\neq \psi (\psi (k))-1.$ They  are valid 
since $j=\psi (i)>n,$ $k\neq i,$ and $k\leq n$ respectively. Further, using Definition \ref{slo} and representations 
(\ref{mu21}), (\ref{mu2}), we have also
$$
\chi ^{i\rightarrow n}(h_{k\rightarrow n})=P_{i\rightarrow n,k\rightarrow n}P_{k\rightarrow n,i\rightarrow n}
=(\sigma _k^n)^2\mu _i^{n,k-1}=q^2\cdot q^{-2}=1.
$$
Hence ad-identity (\ref{br1f}) and identity (\ref{cuq2}) imply that the right hand side of (\ref{ewz2}) equals
$$
h_{k\rightarrow n}u[\psi (k)+1,j]^-\cdot [u[n+1,m],w^-].
$$
Here $\psi (n)=n+1\leq \psi (i)\leq m.$ Hence we may again use already proved case of the proposition  with $k\leftarrow n+1,$
$j\leftarrow n.$ This yields $[u[n+1,m],w^-]\sim h_{n+1\rightarrow j}u[1+j,m],$ which proves (\ref{wff1}),
for $h_{k\rightarrow n}\cdot h_{n+1\rightarrow j}=h_{k\rightarrow \psi (i)}$ in the case $j=\psi (i).$

If $k>n$ then in perfect analogy we have $[v^-,u]\sim [u,v^-]=0,$ while $[w^-,u]$ $\sim [u,w^-]$ 
$\sim  h_{k\rightarrow j} u[1+\psi (k),n]^- \cdot u[1+j,m].$ Therefore
$$
[u,[v^-,w^-]]\sim [[v^-,w^-],u]=[v^-,[w^-,u]]
$$ 
$$
\sim h_{k\rightarrow j} [v^-,u[1+\psi (k),n]^-]\cdot u[1+j,m],
$$
since $[v^-,u[1+j,m]]\sim [u[1+j,m],v^-]=0$ according to Lemma \ref{suu}.
We have $[v^-,u[1+\psi (k),n]^-]\sim u[1+\psi (k),j]^-,$ see Lemma \ref{rww}.
This completes the case $j=\psi (i).$

{\bf 5. $\bf m=\psi (k).$} By means of the mirror generators one may reduce the consideration to 
the case $j=\psi (i).$
The proposition is completely proved.
\end{proof}

\begin{definition} \rm  A scheme (\ref{grr1}) is said to be {\it strongly white} provided that
the following three conditions are met:   
first, it has no black-black columns; then, the first from the left column 
is incomplete; and next, if there are at least two complete columns, then 
the first from the left complete column is  a white-white one.

A scheme (\ref{grr1})  is said to be {\it strongly black} provided that the following three
conditions are met: first, it has no white-white columns; 
then, the last column is incomplete; and next,
 if there are at least two complete columns, then  the last complete column is a black-black one.

A scheme is said to be {\it strong} if it is either strongly white or strongly black.
\label{mb1}
\end{definition}

Alternatively we may define a strong scheme as follows. Let  $S^{\prime }$-scheme be a scheme that 
appears from the $S$-scheme (\ref{grb}) by changing colors of the first and the last points.
Then $ST$-scheme is strongly white (black) if  and only if  both $ST$-scheme and 
$S^{\prime } T^{\prime }$-scheme have no black-black (white-white) columns.

We stress that the map $\rho $ defined in Lemma \ref{bal2} transforms strongly white schemes 
to strongly black ones and vice versa. Therefore the $ST$-scheme is strong if  and only if
the $S^*T^*$-scheme is strong. Similarly, the $ST^*$-scheme is strong if and only if
the $S^*T$-scheme is strong.
\begin{theorem}
Suppose that $S$, $T$ are respectively $(k,m)$- and $(i,j)$-regular sets. 
If  $ST$-  and $ST^*$-schemes are strong,
then $[\Phi ^{S}(k,m),\Phi ^{T}_-(i,j)]=0.$
\label{str}
\end{theorem}  
\begin{proof} Without loss of generality we may suppose
 that both schemes are strongly white. Indeed, the mirror generators
allow us, if necessary,  to switch the roles of $S$ and $T$, 
while Lemma \ref{xn0} and Lemma  \ref{bal2} allow us to find a pair of strongly white  schemes.
Moreover, once $ST$-  and $ST^*$-schemes are strongly white,
Lemma \ref{xn0} allows one to switch the roles of $T$ and $T^*.$
Thus, without loss of generality, we may suppose also that $T$ is white $(i,j)$-regular.

{\bf 1}. Assume $S$ is white $(k,m)$-regular.
We shall use double induction on numbers of elements in $S\cap [k,m)$
and in $T\cap [i,j).$ If both intersections are empty then $i\neq k,$ $j\neq m,$
for $ST$-scheme is strongly white. 
$$
\begin{matrix}
\ & \circ & \stackrel{k}{\circ } & \cdots & \circ & \circ & \circ & \stackrel{m}{\bullet } \cr 
\circ & \stackrel{i}{\circ }& \circ & \cdots   & \circ & \stackrel{j}{\bullet } & \ & \
\end{matrix} \ \ \ \ \ 
\begin{matrix}
\ & \circ & \stackrel{k}{\circ } & \circ & \circ & \cdots & \circ & \circ &  \stackrel{m}{\bullet } \cr 
\ & \ & \circ & \stackrel{\psi (j)}{\bullet }& \bullet & \cdots   & \bullet & \stackrel{\psi (i)}{\bullet } & \
\end{matrix}
$$
Similarly $k,m\notin [\psi (j), \psi (i)],$  for  $ST^*$-scheme is strongly white. 
Hence Proposition \ref{ruk4} applies.

If $s\in \, {S}\, \cap [k,m),$ then by Lemma \ref{xn2} 
we have $\Phi ^{S}(k,m)$ $\sim [\Phi ^{S}(1+s,m),\Phi ^{S}(k,s)].$
It is easy to see that $S^sT$-  and $S^sT^*$-schemes (the schemes for the pair $\Phi ^{S}(k,s),\Phi ^{T}_-(i,j)$)
 are still strongly white, while 
$S$ is still white $(k,s)$- and $(1+s,m)$-regular.
 Hence  the inductive supposition implies
$[\Phi ^{S}(k,s),\Phi ^{T}_-(i,j)]=0.$ By the same reason
$[\Phi ^{S}(1+s,m),\Phi ^{T}_-(i,j)]=0.$ Now Jacobi identity (\ref{uno}) 
implies the required equality.

It remains to consider the case $S\cap [k,m)=\emptyset ;$ that is, 
$\Phi ^{S} (k,m)=u[k,m].$ If $t\in \, {T}\, \cap [i,j),$
then by Lemma \ref{xn2}  we have 
$\Phi ^{T}_-(i,j)$ $\sim [\Phi ^{T}_-(1+t,j),\Phi ^{T}_-(i,t)].$
In this case $T$ is still $(k,s)$- and $(1+s,m)$-regular, while $ST^t$-scheme is strongly white.
At the same time $ST^*_{\psi (t)}$-scheme is not strongly white only if
$\psi (t)-1=k-1$ (the first from the left column is complete).
%, while in the case $\psi (t)-1=m$ the first, and the only, complete column is white-black one). or $\psi (t)-1=m$

 Hence by the inductive supposition we have 
$[\Phi ^{S}(k,m),\Phi ^{T}_-(i,t)]=0$ with one exception being
$\psi (t)=k.$ Similarly $[\Phi ^{S}(k,m),\Phi ^{T}_-(1+t,j)]=0$
with one exception being $\psi (t)-1=m$. Hence, if in the set $T\cap [i,j)$ there exists a point
$t\neq \psi (k),$ $t\neq \psi (m)-1,$ then Jacobi  identity (\ref{cua}) implies the required 
equality. Certainly if $T\cap [i,j)$ has more than two elements then such a point does exist.

If $T\cap [i,j)$ has two points then there is just one exceptional configuration for the main
$ST^*$-scheme:
$$
\begin{matrix}
\ \ & \ & \circ & \stackrel{k}{\circ } & \circ \ \  \cdots & \circ  & \circ & \stackrel{m}{\bullet } & \ &\ \cr 
\circ & \stackrel{\psi (j)}{\bullet }\ \cdots \ \ \ \bullet 
& \circ & \bullet & \bullet \ \ \cdots  & \bullet  & \bullet  & \circ 
& \bullet  & \cdots \ \  \stackrel{\psi (i)}{\bullet }
\end{matrix}
$$
In this case $T\cap [i,j)=\{ t_1,t_2\} ,$ where $\psi (t_2)-1=k-1,$ $\psi (t_1)-1=m.$
Let $a=\Phi ^{S}(k,m)=u[k,m],$ $b^-=\Phi ^{T}_-(i,j),$
$u_0^-=u[i,t_1]^-$ $=u[i,\psi (m)-1]^-,$ $u_1^-=u[1+t_1,t_2]^-$ $=[\psi (m),\psi (k)]^-,$ 
$u_2^-=u[1+t_2,j]^-=u[1+\psi (k),j]^-.$ Using Lemma \ref{xn2} twice,
we have $b^-\sim [[u_2^-,u_1^-],u_0^-].$ Lemma \ref{dus1}  implies $[a,u_1^-]\sim 1-h_a.$ 
Inequality $\psi (m)\leq \psi (k)$ implies both $\psi (m), \psi (k)\notin [i,\psi (m)-1]$ $=[i,t_1]$
and $\psi (m), \psi (k)\notin [1+\psi (k),j]=[1+t_2,j].$ Therefore by Proposition \ref{ruk4} we have
$[a,u_0^-]=0$ unless $m=\psi (m)-1,$ and $[a,u_2^-]=0$ unless  $k=1+\psi (k).$
At the same time $k=1+\psi (k)$ implies $k=n+1,$ and hence $n=\psi (k)=t_2\in {T}\cap [i,j),$
which is impossible, for $T$ is white $(i,j)$-regular. Similarly, $m=\psi (m)-1$
implies $m=n,$ and hence $n=\psi (m)-1=t_1\in {T}\cap [i,j),$
 which is wrong by the same reason.

Taking into account  the proved relations, we may write 
$$
[a,b^-]\sim [a,[[u_2^-,u_1^-],u_0^-]]\stackrel{(\ref{jak3})}{=}[[a,[u_2^-,u_1^-]],u_0^-]
$$
$$
\stackrel{(\ref{cua})}{\sim }[[u_2^-,[a,u_1^-]],u_0^-] \stackrel{(\ref{cuq3})}{\sim }[u_2^-,u_0^-].
$$
Here we have applied inhomogeneous substitution (\ref{cuq3}) to the left factor in the brackets.
 Proposition \ref{NU} with $k\leftarrow i,$ $m\leftarrow j,$ $i\leftarrow t_1,$
$j\leftarrow t_2$ implies  $[u_2^-,u_0^-]=0$ provided that $j\neq \psi (i),$ $j\neq \psi (t_1)-1,$ $t_2\neq \psi (i),$
and $t_1\neq \psi (t_2)-1.$ The first inequality is valid since $T$ is $(i,j)$-regular. The second and third 
inequalities are equivalent to $j\neq m$ and $\psi (k)\neq \psi (i)$ respectively. 
However $j\neq m$ and $k\neq i$ are valid, for the main $ST$-scheme is strongly white.
The equality  $t_1=\psi (t_2)-1$ is equivalent to $m=t_2,$ while in this case 
on the $ST$-scheme we have a black-black column. 

 If ${T}\cap [i,j)=\{ t\} $  then there are just two exceptional configuration for the main
$ST^*$-scheme, where $\psi (t)=k$ in case A, and $\psi (t)=m+1$ in case B:
$$
{\rm A:}\ 
\begin{matrix}
&  & \circ & \stackrel{k}{\circ } & \circ & \circ  & \stackrel{m}{\bullet }  \cr 
\circ & \stackrel{\psi (j)}{\bullet } \bullet 
 & \circ & \bullet & \bullet   
 &\stackrel{\psi (i)}{\bullet } & \
\end{matrix};\ \ \ 
{\rm B:}\ 
\begin{matrix}
  \circ & \stackrel{k}{\circ } & \circ & \circ  & \circ & \stackrel{m}{\bullet } & \ &\ \cr 
\ & \ & \circ & \stackrel{\psi (j)}{\bullet } 
   & \bullet  & \circ 
& \bullet  & \stackrel{\psi (i)}{\bullet }
\end{matrix}
$$
In case A we keep the above notations $a=u[k,m],$ $b^-=\Phi ^{T}_-(i,j),$
$u_0^-=u[i,t]^-,$  $u_1^-=u[1+t,j]^-.$ Lemma \ref{xn2} implies $b^-=[u_1^-,u_0^-].$
We have $\psi (m),\psi (k)\notin [1+t,j]=[1+\psi (k),j].$
Moreover  $k\neq 1+t,$ for otherwise the first from the left complete
column on the main $ST$-scheme is white-black which contradicts the definition 
of a strongly white scheme (here $t\neq j$ and therefore the scheme has at least two 
complete columns).
Hence Proposition \ref{ruk4} implies $[a,u_1^-]=0.$
Since $\psi (t)=k\leq \psi (i)\leq m,$
Proposition \ref{fkk1} shows that $[a,u_0^-]\sim h_{k\rightarrow \psi (i)}u[\psi (i)+1,m].$
Thus we get
$$
[a,b^-]=[a,[u_1^-,u_0^-]]\stackrel{(\ref{cua})}{\sim }[u_1^-,[a,u_0^-]]
\stackrel{(\ref{cuq1})}{\sim }h_{k\rightarrow \psi (i)}[u_1^-,u[\psi (i)+1,m]]=0.
$$
The latter equality follows from antisymmetry identity (\ref{dos}) and Proposition \ref{ruk4}.
Indeed, $\psi (i)+1\neq 1+t=1+\psi (k),$ for $i\neq k,$ while in  configuration A we have
$\psi (m),$ $\psi (\psi (i)+1)$ $\notin [1+t,j]$ since 
$\psi (i)+1, m\notin [\psi (j), \psi (1+t )]=[\psi (j),k-1].$ This allows one to apply Proposition \ref{ruk4}.

In case B we consider the points $k^{\prime }=\psi (m),$ $m^{\prime }=\psi (k),$ 
$i^{\prime }=\psi (j),$ $j^{\prime }=\psi (i),$ and $t^{\prime }=\psi (t)-1=m.$ 
These points are in configuration A. Therefore we have 
$[u[k^{\prime },m^{\prime }], \Phi ^{\{ t^{\prime }\} }_-(i^{\prime },j^{\prime })]=0.$
Let us apply $g_{k\rightarrow m}f_{i\rightarrow j}\sigma ,$ where $\sigma $
is the antipode.
Using properties of the antipode given in (\ref{ant1}), (\ref{ant2}), (\ref{desc3})
we get the required equality.

{\bf 2}. If $S$ is black $(k,m)$-regular, but not white $(k,m)$-regular, then $n\in [k,m),$
and $n$ is a black point on the scheme $S.$
Lemma \ref{xn1} implies  $\Phi ^{S}(k,m)$ $=[\Phi ^{S}(k,n),\Phi ^{S}(n+1,m)].$
By definition $S$, as well as any other set, is white $(k,n)$- and $(n+1,m)$-regular.
Since $ST$- and $ST^*$-schemes are strongly white, the point $n$ is not black
on the schemes $T,T^*.$ At the same time 
$n$ is a white point on $T$ if and only if it is a black point on $T^*.$ Hence $n$
does not appear on $T,T^*$ at all, $n\notin [i-1,j].$ 
In particular  $S^nT$-, and  $S^nT^*$-schemes  (the schemes for the pair 
$\Phi ^{S}(k,n),$  $\Phi ^{T}_-(i,j)$) are still strongly white.
The above considered case  implies $[\Phi ^{S}(k,n),\Phi ^{T}_-(i,j)]=0.$
By the same reason $[\Phi ^{S}(n+1,m),\Phi ^{T}_-(i,j)]=0.$
It remains to apply Jacobi identity (\ref{uno}).
\end{proof}

\section{Proof of the main theorem}
\begin{lemma}
Let $k\leq s<n.$ If $s\in S,$ then
\begin{equation}
\left[ \Phi^{S}(k,n), \Phi^{\overline{S}}_-(k,s)\right] \sim  \Phi^{S}(1+s,n),
\label{but1}
\end{equation}
where $\overline{S}$ is the complement of $S$ with respect to $[k,s).$
\label{ed}
\end{lemma}
\begin{proof}
By Lemma \ref{xn2} we have $\Phi^{S}(k,m)=\left[ \Phi^{S}(1+s,n), \Phi^{S}(k,s)\right].$
At the same time $\left[ \Phi^{S}(1+s,n), \Phi^{\overline{S}}_-(s,n)\right] =0$ due to Lemma \ref{suu}.
Taking into account Theorem \ref{des1} we have
$$
\left[\left[ \Phi^{S}(1+s,n), \Phi^{S}(k,s)\right], \Phi^{\overline{S}}_-(k,s)\right] 
$$
$$
\stackrel{(\ref{jak3})}{=}
\left[ \Phi^{S}(1+s,n), \left[ \Phi^{S}(k,s), \Phi^{\overline{S}}_-(k,s)\right] \right]\stackrel{(\ref{cuq3})}{\sim } \Phi^{S}(1+s,n),
$$
where the coefficient of the proportion equals 
$1-\chi ^{1+s\rightarrow n}(h_{k\rightarrow s})$ $=1-\mu _k^{n,s}$ $=1-q^{-2},$  see (\ref{mu2}).
\end{proof}
\begin{lemma}
Let $i\neq k,$ $m\leq n.$ If the $S_k^mT_i^m$-scheme has only one black-black column $($the last one$),$
and the first complete column is white-white then 
\begin{equation}
\left[ \Phi^{S}(k,m), \Phi^{T}_-(i,m)\right] 
=\sum _{b=\nu -1}^{m-1}\alpha _b \, \Phi^{T}_-(i,b)\cdot \Phi^{S}(k,b),
\label{but}
\end{equation}
where $\nu =\max \{ i,k\} ,$ while $\alpha _b\neq 0$ if and only if the column $b$ is white-white.
Here by definition $\Phi^{S}(k,k-1)=\Phi^{T}_-(i,i-1)=1.$ 
\label{ed1}
\end{lemma}
\begin{proof} For the sake of definiteness, assume that $k<i$ (if $i<k$ then the proof is quite similar). 
We use induction on the number of white-white columns on the $S_k^mT_i^m$-scheme.
If there is just one white-white column then this is the first from the left column labeled by $i-1.$
Moreover all intermediate complete columns are white-black or black-white.
Hence Theorem \ref{des1} implies $\left[ \Phi^{S}(i,m), \Phi^{T}_-(i,m)\right] \sim 1-h_{i\rightarrow m}.$
By Lemma \ref{xn1} we have $\Phi^{S}(k,m)=\left[ \Phi^{S}(k,i-1), \Phi^{S}(i,m)\right].$
At the same time $\left[ \Phi^{S}(k,i-1), \Phi^{T}_-(i,m)\right] =0$ due to Lemma \ref{suu}.
Hence 
$$
\left[ \left[ \Phi^{S}(k,i-1), \Phi^{S}(i,m)\right] , \Phi^{T}_-(i,m)\right] 
$$
$$
\stackrel{(\ref{jak3})}{=}\left[ \Phi^{S}(k,i-1), \left[ \Phi^{S}(i,m), \Phi^{T}_-(i,m)\right] \right]
\stackrel{(\ref{cuq3})}{\sim } \Phi^{S}(k,i-1),
$$
which is required, for the coefficient of the proportion equals 
$1-\chi ^{k\rightarrow i-1}(h_{i\rightarrow m})=1-\mu _k^{m,i-1}=1-q^{-2}\neq 0$ (recall that $m\leq n$).

To make the inductive step, let $a$ be the maximal white-white column.
Then all columns between $a$ and $m$ are black-white or white-black.
Hence Theorem \ref{des1} implies 
$\left[ \Phi^{S}(1+a,m), \Phi^{T}_-(1+a,m)\right] \sim 1-h_{1+a\rightarrow m}.$
Let us fix for short the following designations:
$$u=\Phi^{S}(k,a),\  v=\Phi^{S}(1+a,m),\  w^-=\Phi^{T}_-(i,a), \ t^-=\Phi^{T}_-(1+a,m).$$
Then by Lemma \ref{xn1} we have $\Phi^{S}(k,m)=[u,v],$ $\Phi^{T}_-(i,m)=[w^-,t^-].$
Lemma \ref{suu} implies $[u,t^-]$ $=[v,w^-]$ $=0.$
At the same time  $[v,t^-]\sim 1-h_{1+a\rightarrow m},$
while $[u,w^-]$ equals the left hand side of (\ref{but}) with $m\leftarrow a.$
Applying inductive supposition to $[u,w^-]$ we see that $[[u,w^-], t^-]=0.$
Indeed,  we may apply inhomogeneous substitution (\ref{but}) to the left factor of the bracket.
Then  for each $b<a$ we have $\left[ \Phi^{S}(k,b),t^-\right ]=0$ by Lemma \ref{suu}, while 
$\left[ \Phi^{T}_-(i,b),t^-\right ]=0$ due to Lemma \ref{sepp}. Additionally, using (\ref{cuq3})
with $x_i\leftarrow v,$ $x_i^-\leftarrow t^-,$ we have $[w^-,[v,t^-]]\sim w^-,$ 
for $\chi ^{w^-}(g_vf_t)=(\mu _i^{m,a})^{-1}=q^{2}\neq 1$ according to (\ref{mu2}).

All that relations allow us to simplify  (\ref{fo1}):
$$
[[u,v],[w^-,t^-]]\sim [u,[w^-,[v,t^-]]]=u\cdot w^--p(u,w^-vt^-)w^-\cdot u
$$
\begin{equation}
=[u,w^-]+p(u,w^-)(1-p(u,vt^-))w^-\cdot u.
\label{vse}
\end{equation}
Here we apply inhomogeneous substitution to the right factor of the bracket.
By this reason we have to develop the bracket to its explicit form. We have
$p(u,vt^-)$ $=p(u,v)p(t,u)$ $=\mu _k^{m,a}$ $=q^{-2}$ $\neq 1.$
Thus inductive supposition applied to $[u,w^-]$ shows that (\ref{vse})
is the required sum. 
\end{proof}
\begin{lemma}
Let $S$ be a black $(k,m)$-regular set, $k\leq n<m.$ We have
\begin{equation}
\varepsilon ^-\otimes \varepsilon ^0\otimes {\rm id}\left(  \left[ \Phi^{S}(k,m), \Phi^{\overline{S}}_-(k,n)\right] \right)\neq 0.
\label{nac}
\end{equation}
This nonzero element has degree $[\psi (m):n]=[n+1:m].$
Here $\varepsilon ^-,$ $\varepsilon ^0$ are the counits of $ U_q^-$
and ${\bf k}[H]$ respectively,
 the tensor product of maps is related to the triangular decomposition $(\ref{tr}), (\ref{trs});$
while $\overline{S}$ is a complement of $S$ with respect to $[k,n).$
\label{xic} 
\end{lemma}
\begin{proof} Let us fix for short the following designations:
$$
u=\Phi^{S}(k,n), \ \ v=\Phi^{S}(1+n,m), \ \ w^-=\Phi^{\overline{S}}_-(k,n).
$$
By Lemma \ref{xn1} we have $\Phi^{S}(k,m)=[u,v],$ while Jacobi identity (\ref{uno})
implies that $\left[ \Phi^{S}(k,m), \Phi^{\overline{S}}_-(k,n)\right]$ is a linear combination of 
two elements, $[u,[v,w^-]]$ and $[[u,w^-],v].$ The latter one equals zero since 
due to Theorem \ref{des1}  we have $[u,w^-]\sim 1-h_v,$ and  coefficient of (\ref{cuq4})
with $x_i\leftarrow u,$ $x_i^-\leftarrow w^-,$ $u\leftarrow v$
is $\chi ^v(g_uf_w)-1$ $=\mu _k^{m,n}-1$ $=0,$ see (\ref{mu2}), (\ref{mu4}).  
Further, due to Proposition \ref{xn0} we have
\begin{equation}
[v,w^-]\sim \left[ \Phi^{\overline{\psi (S)-1}}(\psi (m),n), \Phi^{\overline{S}}_-(k,n)\right] . 
\label{nac2}
\end{equation}
Let us show that we may apply Lemma \ref{ed1} to this bracket. 
If $a\in [k,n)$ is a black point on $\overline{S}$ then $a$ is a white point
on $S.$ If additionally $a\in [\psi (m),n)$ then $\psi (a)-1\in [n,m).$ Moreover
since $S$ is black $(k,m)$-regular, the point $\psi (a)-1$ is black on $S.$
Hence $a=\psi (\psi (a)-1)-1$ is a white point on $\overline{\psi (S)-1}.$
Thus the $\overline{\psi (S)-1}_{\, \psi (m)}^{\, n}\, \overline{S}_{\, k}^{\, n}$-scheme has no
black-black columns except the last one. The first from the left complete column is labeled by
$\nu -1,$ where $\nu =\max \{ \psi (m), k \} .$ If $\psi (m)<k$ then $\psi (k)$ is black on $S,$
see (\ref{grab3}), hence $k-1=\psi (\psi (k))-1$ is white on $\overline{\psi (S)-1}.$
If $k<\psi (m)$ then $\psi (m)-1$ is black on $S,$ see (\ref{grab2}), hence 
$\psi (m)-1$ is white on $\overline{S}.$ Thus in both cases the first form the left 
complete column is white-white. By Lemma \ref{ed1} we may continue (\ref{nac2}):
\begin{equation}
=\sum _{b=\nu -1}^{n-1}\alpha _b \, \Phi^{\overline{S}}_-(k,b)\cdot \Phi^{\overline{\psi (S)-1}}(\psi (m),b)
\stackrel{\rm df}{=} \sum _{b=\nu -1}^{n-1}\alpha _b\, w_b^-\cdot v_b.
\label{nac3}
\end{equation}
In order to find $[u,[v,w^-]]$ we would like to substitute the found value of $[v,w^-].$
However this is inhomogeneous substitution to the right factor of the bracket.
Therefore we have to develop the brackets to their explicit form and analyze the coefficients. We have 
$$
p(u,vw^-)p(u,w^-_bv_b)^{-1}=p(u,v)p(u,v_b)^{-1}p(w,u)p(w_b,u)^{-1}
$$
$$
=P_{k\rightarrow n, 1+n\rightarrow \psi (b)-1}P_{1+b\rightarrow n,k\rightarrow n}
$$
$$
= P_{k\rightarrow n, 1+n\rightarrow \psi (b)-1}P_{1+n\rightarrow \psi (b)-1, k\rightarrow n}
=\mu _k^{\psi (b)-1,n},
$$ 
see definition (\ref{mu1}). Relations  (\ref{mu2}--\ref{mu4}) show that $\mu _k^{\psi (b)-1,n}=1$
unless $b=k-1.$ If $b=k-1$ then $\mu _k^{\psi (b)-1,n}$ $=\mu _k^{\psi (k),n}$ $=q^2,$ see (\ref{mu3}).
Thus all brackets $[u, w_b^-\cdot v_b]$ have the same coefficient as $[u,vw^-]$ does 
with only one exception being $b=k-1.$

If $k<\psi (m)$ then of course $b\neq k-1,$ for $b\geq \nu -1=\psi (m)-1\geq k.$ Hence 
in this case by ad-identity (\ref{br1}) the element $[u,[v,w^-]]$ splits in linear combination of two sums:
\begin{equation}
\sum _{b=\psi (m) -1}^{n-1}\alpha _b \, \left[ \Phi^{S}(k,n),\Phi^{\overline{S}}_-(k,b)\right] 
\cdot \Phi^{\overline{\psi (S)-1}}(\psi (m),b).
\label{nac4}
\end{equation}
and
\begin{equation}
\sum _{b=\psi (m) -1}^{n-1}\alpha _b \, \Phi^{\overline{S}}_-(k,b)\cdot 
\left[ \Phi^{S}(k,n),\Phi^{\overline{\psi (S)-1}}(\psi (m),b)\right] .
\label{nac5}
\end{equation}
By Lemma \ref{ed} we have $\left[ \Phi^{S}(k,n),\Phi^{\overline{S}}_-(k,b)\right] \sim \Phi^{S}(1+b,n),$
for $b$ is a white point on $\overline{S}$ (otherwise $\alpha _b=0$). In particular 
all terms in (\ref{nac4}) belong to the positive quantum Borel subalgebra, and hence application
of $\varepsilon ^-\otimes \varepsilon ^0\otimes {\rm id}$ does not change this sum.
Application of $\varepsilon ^-\otimes \varepsilon ^0\otimes {\rm id}$
to (\ref{nac5}) kill all terms, for $\varepsilon ^- (\Phi^{\overline{S}}_-(k,b))=0,$ $b\geq k.$
Thus the left hand side
of (\ref{nac}) takes up the form
\begin{equation}
 \alpha \, \Phi^{S}(\psi (m),n)+\sum _{b=\psi (m)}^{n-1}\alpha _b \, 
\Phi^{S}(1+b,n) \cdot \Phi^{\overline{\psi (S)-1}}(\psi (m),b),
\label{nac6}
\end{equation}
where $\alpha =\alpha _{\psi (m)-1}\neq 0.$ 
We may decompose all terms in this expression using definition (\ref{dhs}).
As a result we will get a polynomial, say $F,$ in  $u[i,j],$ $1\leq i\leq j\leq n.$
It is very important to note that all first from the left factors $u[i,j]$ in all monomials 
of $F$ satisfy $i>\psi (m)$ with only one exception, $\alpha \, u[\psi (m),n],$ coming from the first 
term of (\ref{nac6}). In particular $u[i,j]<u[\psi (m),n]$ (recall that $x_1>x_2>\ldots >x_n,$
while  words in $X$ are ordered lexicographically). Hence further diminishing process
of decomposition in PBW-basis (see \cite[Lemma 7]{Kh3}) produces words in 
$u[i,j],$ $j<\psi (i)$ that start with lesser than $u[\psi (m),n]$ elements. This means that 
$\alpha \, u[\psi (m),n]$ is still the leading term of (\ref{nac6}) after the  PBW-decomposition.
In particular (\ref{nac6}) is not zero.

If $\psi (m)<k$ then again  by ad-identity (\ref{br1}) the element $[u,[v,w^-]]$ splits in  sums (\ref{nac4}), (\ref{nac5})
with $\sum\limits_{b=\psi (m)-1}^{n-1}\leftarrow \sum\limits_{b=k}^{n-1}$ and an
additional term that corresponds to the value $b=k-1.$ Since $\alpha _{k-1}\neq 0,$ this term is proportional to  
\begin{equation}
u\cdot v_{k-1}-p(u,vw^-)\, v_{k-1}\cdot u,
\label{nac7}
\end{equation}
where $v_{k-1}=\Phi^{\overline{\psi (S)-1}}(\psi (m),k-1)$ was defined in (\ref{nac3}).
We have already seen that $p(u,vw^-)=p(u,v_{k-1})\, q^2.$ At the same time 
$p(v_{k-1},u)p(u,v_{k-1})$ $=p_{k-1\, k}\, p_{k\, k-1}$ $=q^{-2},$ for $p_{ij}p_{ji}=1,$ $j>i+1,$
see (\ref{b1rell}). Hence $p(u,v_{k-1})\, q^2$ $=p(v_{k-1},u)^{-1}.$ Therefore the term
(\ref{nac7}) is proportional to $[v_{k-1},u]$ with coefficient $-p(v_{k-1},u)^{-1}.$
Taking into account formula (\ref{desc1}), we have
\begin{equation}
[v_{k-1},u]=[\Phi^{\overline{\psi (S)-1}}(\psi (m),k-1), \Phi^{S}(k,n)]=\Phi^{R}(\psi (m),n),
\label{nac8}
\end{equation}
where $R=\left( \overline{\psi (S)-1}\cap [\psi (m),k-1)\right) \cup \left(S\cap [k,n)\right) .$

Certainly the map $\varepsilon ^-\otimes \varepsilon ^0 \otimes {\rm id}$
kills all terms of (\ref{nac5}) with $b\geq k,$ while 
Lemma \ref{ed} implies  $\left[ \Phi^{S}(k,n),\Phi^{\overline{S}}_-(k,b)\right] \sim \Phi^{S}(1+b,n),$
$b\geq k.$ Thus the left hand side of (\ref{nac}) is proportional to the sum 
\begin{equation}
 \Phi^{R}(\psi (m),n)+\sum _{b=k}^{n-1}\alpha _b^{\prime } \, 
\Phi^{S}(1+b,n) \cdot \Phi^{\overline{\psi (S)-1}}(\psi (m),b).
\label{nac9}
\end{equation}
This is a nonzero element precisely by the same reasons as (\ref{nac6}) is.
\end{proof}

{\it Proof of Theorem} \ref{bale}. Suppose that there exists a pair of simple roots 
such that one of  schemes (\ref{grr1}-\ref{grr4}) has fragment (\ref{bal})
and no one of these schemes has form (\ref{gra3}). Among all that pairs we choose 
a pair $[k:m],$ $[i:j]^-$ that has fragment (\ref{bal}) with minimal possible $s-t$ on one of the schemes.
Actually, due to Lemma \ref{bal2}, there are at least two of the schemes that have fragments 
with that minimal value of $s-t.$ Without loss of generality, changing if necessary 
notations $S\leftrightarrow S^*$ or $T\leftrightarrow T^*$ or both, we may assume
that the $ST$-scheme has that fragment. 
Since $s-t$ is minimal, there are no white-white or black-black columns between $t$ and $s.$
Hence due to Theorem \ref{des1} we have
\begin{equation}
\left[ \Phi^{S}(1+t,s), \Phi^{T}_-(1+t,s)\right] \sim 1-h_{1+t\rightarrow s},
\label{bit}
\end{equation}
provided that $S$ or, equivalently, $T$ is $(1+t,s)$-regular.

\smallskip

\underline{Let, first,  $s\leq n.$} In this case by definition $S$ is $(1+t,s)$-regular, while due to Lemma \ref{sig3} 
we have  $\Phi^{S}(k,s)\in U^{S}(k,m)\subseteq  U^+, $ $\Phi^{T}_-(1+t,s)\in U^{S}_-(i,j)\subseteq  U^-.$
Hence we get 
\begin{equation}
\left[ \Phi^{S}(k,s), \Phi^{T}_-(1+t,s)\right] \in [U^+,U^-]\subseteq U.
\label{bit1}
\end{equation}
At the same time by Lemma \ref{xn1} we have a decomposition
\begin{equation}
\Phi^{S}(k,s)\sim \left[ \Phi^{S}(k,t),\Phi^{S}(1+t,s) \right] .
\label{bit2}
\end{equation}
Lemma \ref{suu} implies 
\begin{equation}
\left[ \Phi^{S}(k,t),\Phi^{T}_-(1+t,s) \right] =0. 
\label{bit3}
\end{equation}
Applying first (\ref{jak3}), and then (\ref{cuq3}) with $x_i\leftarrow \Phi^{S}(1+t,s),$ 
$x_i^-\leftarrow \Phi^{T}_-(1+t,s)$ due to  (\ref{bit})  we see that the left hand side 
of (\ref{bit1}) is proportional to $\Phi^{S}(k,t),$ in which case the coefficient
equals $1-\chi ^{k\rightarrow t}(h_{1+t\rightarrow s})$ $=1-\mu _k^{s,t}=1-q^{-2},$
see (\ref{mu1}), (\ref{mu2}).
Thus $\Phi^{S}(k,t)\in U\cap U_q^+(\mathfrak{so}_{2n+1})=U^+;$
that is, $[k:t]$ is an $U^+$-root. According to Lemma \ref{sig} we have
  $[1+t:m]\in \Sigma (U^S(k,m))\subseteq \Sigma (U^+).$
This implies that $t=k-1,$ for otherwise we have a contradiction:  
$[k:m]=[k:t]+[1+t:m]$ is a decomposition of a simple $U^+$-root in $\Sigma (U^+).$
Similarly, due to the mirror symmetry,  we have $t=i-1;$ that is, $k=i=1+t.$

Now we are going to show that $m=s.$ Equality $t=k-1$ implies 
\begin{equation}
\left[ \Phi^{S}(k,m), \Phi^{T}_-(k,s)\right] \in [U^+,U^-]\subseteq U.
\label{bit4}
\end{equation}

\smallskip
{\it Let  $s=n.$} In this case $n$ is black on $S;$ that is, $S$ is black $(k,m)$-regular.
We have $\varepsilon ^-\otimes \varepsilon ^0\otimes {\rm id} (U)\subseteq U^+.$
Hence if $m\neq s=n,$ Lemma \ref{xic} allows us to find in $U^+$ a nonzero element of degree
$[n+1:m].$ In particular $[n+1:m]\in \Sigma (U^+).$ 
Hence $[k:m]=[k:n]+[1+n:m]$ is a decomposition of a simple $U^+$-root in $\Sigma (U^+).$
A contradiction that implies $m=n=s.$
 
\smallskip
{\it Let,  further, $s<n.$} If $S$ is white $(k,m)$-regular, or if $S$ is black $(k,m)$-regular and
$\psi (s)-1$ is not white, then $S$ is still $(1+s,m)$-regular, 
while Lemma \ref{xn2} provides a decomposition
\begin{equation}
\Phi^{S}(k,m)\sim \left[ \Phi^{S}(1+s,m), \Phi^{S}(k,s) \right] .
\label{bit5}
\end{equation}
Let us show that Theorem  \ref{str} implies 
\begin{equation}
\left[ \Phi^{S}(1+s,m),\Phi^{T}_-(k,s) \right] =0. 
\label{bit6}
\end{equation}
To see this we have to check that $S_{1+s}^mT_k^s$-  and $S_{1+s}^mT_{\psi (s)}^{*\psi (k)}$-schemes
are strong. The first one has just one complete column, hence it is both strongly  white and strongly black.
Suppose firstly that $S$ is white $(k,m)$-regular. 
 Let us show that if $s\neq n$ (even if $s>n$),
 then the $S_{1+s}^mT_{\psi (s)}^{*\psi (k)}$-scheme is strongly white.

If $a$ is a black point on $T_{\psi (s)}^{*\psi (k)},$ $\psi (s)\leq a<\psi (k),$
then by definition $\rho (a)=\psi (a)-1$ is white on $T_k^s.$ 
The inequalities  $k\leq \rho (a)<s$ imply that the point $\rho (a)$
 is intermediate for the minimal fragment (\ref{bal}),
recall that now $t=k-1.$ Therefore $\rho (a)$ is black on $S.$
Since $S$ is white $(k,m)$-regular, the point $a=\rho (\rho (a))$ is not black on $S.$
If $a=\psi (k),$ then still $a$ is not black on $S,$ see (\ref{grab1}).
Thus the $S_{1+s}^mT_{\psi (s)}^{*\psi (k)}$-scheme
has no  intermediate complete black-black columns.

Since $s\neq n,$ we have $\psi (s)\neq 1+s.$ Hence the first from the left column is incomplete. 

If there are at least two complete columns, then $m\geq \psi (s).$
In this case the first from the left complete column has the label  $a=\psi (s)-1=\rho (s).$
Since $s$ is black on $S,$ and $S$ is white $(k,m)$-regular, the
point $\rho (s)$ is white on $S.$ Thus the first from the left complete 
column is white-white one, and $S_{1+s}^mT_{\psi (s)}^{*\psi (k)}$-scheme is strongly white.

Similarly we shall show that  if $S$ is black $(k,m)$-regular and $\psi (s)-1$ is not white, $s<n,$
then the $S_{1+s}^mT_{\psi (s)}^{*\psi (k)}$-scheme is strongly black.
If $a$ is a white point on $T_{\psi (s)}^{*\psi (k)},$ $\psi (s)\leq a<\psi (k),$
then by definition $\rho (a)=\psi (a)-1$ is black on $T_k^s,$ and hence it is white on $S.$
Since $S$ is black $(k,m)$-regular, the point $a=\rho (\rho (a))$ is not white on $S.$
Thus the $S_{1+s}^mT_{\psi (s)}^{*\psi (k)}$-scheme
has no  complete white-white columns (recall that now $\psi (s)-1$ is not white on $S,$ 
hence the column  $\psi (s)-1$ is not white-white).

The last column is incomplete, for $\psi (k)\neq m.$  

If there are at least two complete columns; that is, $m\geq \psi (s),$ then
the last complete column is labeled by $m$ or by $\psi (k).$
In the former case $\psi (m)-1$ is black on $S,$ see (\ref{grab2}).
Hence, as an intermediate point for (\ref{bal}), it is white on $T.$
Therefore $m=\psi (\psi (m)-1)-1$ is black  on $T^*.$
It is still black on $T_{\psi (s)}^{*\psi (k)},$ for $m\neq \psi (s)-1.$
In the latter case $\psi (k)$ is black on $S,$ see (\ref{grab3}).
Hence it is black on $S_{1+s}^m$ too, for $\psi (k)\neq s$ (recall that now $k\leq s<n$).
Thus the  $S_{1+s}^mT_{\psi (s)}^{*\psi (k)}$-scheme is strongly black.
This completes the proof of (\ref{bit6}).

Now we show how  (\ref{bit}) with $t=k-1$ and (\ref{bit4}--\ref{bit6}) imply $s=m.$ 
Applying first (\ref{jak3}), and then (\ref{cuq3})  due to  (\ref{bit})  we see that the left hand side 
of (\ref{bit4}) is proportional to $\Phi^{S}(1+s,m),$ in which case 
$\chi ^{1+s\rightarrow m}(h_{k\rightarrow s})$ $=\mu _k^{m,s}\neq 1,$
with only one exception being $s=n,$ see (\ref{mu2}--\ref{mu4}).
Thus $\Phi^{S}(1+s,m)\in U\cap U_q^+(\mathfrak{so}_{2n+1})=U^+;$
that is, $[1+s:m]$ is an $U^+$-root. According to Lemma \ref{sig} we have
  $[k:s]\in \Sigma (U^S(k,m))\subseteq \Sigma (U^+).$
This implies $s=m,$ for otherwise we have a forbidden decomposition of a simple $U^+$-root,
$[k:m]=[k:s]+[1+s:m].$

If $S$ is black $(k,m)$-regular and
$\psi (s)-1$ is white on $S,$ then we may not use (\ref{bit5}). However in this case
$\Phi^{S}(k,n)\in U^S(k,m)\subseteq U^+,$ and certainly $S$ is white $(k,n)$-regular.
Hence instead of (\ref{bit5}) we may consider the decomposition 
$\Phi^{S}(k,n)\sim \left[ \Phi^{S}(1+s,n), \Phi^{S}(k,s)\right] ,$
while instead of (\ref{bit6}) use $\left[ \Phi^{S}(1+s,n),\Phi^{T}_-(k,s) \right] =0,$
which is valid due to Lemma \ref{suu}. Hence  we get $[1+s:n]\in \Sigma (U^+),$ for now $s<n.$
This also provides a forbidden decomposition, $[k:m]=[k:s]+[1+s:n]$ $+[n+1:\psi (s)-1]+[\psi (s):m],$
unless $s=m.$
Here $[n+1:\psi (s)-1]=[1+s:n],$ while $[\psi (s):m]\in \Sigma (U^+)$ due to Lemma \ref{sig}.

Thus in all cases $s=m.$
Due  to the mirror symmetry  we have also $s=j;$ that is, the $ST$-scheme has the form (\ref{gra3}).
This contradiction completes the case ``$s\leq n.$"

\smallskip

\underline{Let, then, $n\leq t.$} By Lemma \ref{bal2} the $S^*T^*$-scheme
also contains a fragment (\ref{bal}) with $t\leftarrow \psi (s)-1,$ $s\leftarrow \psi (t)-1.$
Since $n\leq t$ implies $\psi (t)-1\leq n,$ one may apply already considered case 
to the $S^*T^*$-scheme. 

\smallskip

\underline{Let, next, $t<n<s.$} In this case the $n$th column, as an intermediate one, is either white-black or
black-white. Since the color of the point $n$ defines the color of regularity, 
$S$ and $T$ have different color of regularity. For the sake of definiteness, we assume that 
$S$ is white $(k,m)$-regular, while $T$ is black $(i,j)$-regular (otherwise one may change the roles
of $S$ and $T$ considering the mirror generators).

If $\psi (t)-1$ is a black point on the scheme $S,$ then on the $ST^*$-scheme
we have a new fragment of the form (\ref{bal}) with $t\leftarrow n,$ $s\leftarrow \psi (t)-1,$
for the color of $\rho (t)=\psi (t)-1$ on the scheme $T^*$ is also black. Certainly $\psi (t)-1-n=n-t<s-t,$
for $n<s;$ that is, we have found a lesser fragment.
Hence  $\psi (t)-1$ is not black on the scheme $S.$ Lemma \ref{sig1}
implies  $\Phi^{S}(1+t,s)\in U^+,$ while Lemma \ref{si} shows that 
$S$ is white $(1+t,s)$-regular. In particular (\ref{bit}) is valid. 
Moreover 
 $S\cup \{ t \}$ is still white $(k,m)$-regular, hence we have decomposition 
(\ref{desc1}).
In perfect analogy  $\psi (s)-1$ is not white on the scheme $T.$
Hence Lemma \ref{sig2} implies $\Phi^{T}_-(1+t,s)\in U^-,$ 
and we have decomposition (\ref{desc2}) of $\Phi^{T}_-(i,j).$

By definition of a white regular set the point $\psi (k-1)-1=\psi (k)$ is not black
on the scheme $S,$ see (\ref{grab}), (\ref{grab1}). Hence Lemma \ref{sig1}
implies $\Phi^{S}(k,s)\in U^+.$ Therefore (\ref{bit1}) is still valid. Lemma \ref{xn1}
implies decomposition (\ref{bit2}), for $\psi (t)-1$ is not black on the scheme $S.$
Let us show that Theorem \ref{str} implies (\ref{bit3}). 

Indeed, the $S_k^{t}T_{1+t}^s$-scheme has just one complete column,
hence it is strongly white (and of course it is strongly black too). Let us check the
$S_k^{t}T_{\psi (s)}^{*\psi (t)-1}$-scheme. 
If $a$ is a black point on $S_k^t,$ $k\leq a<t,$ then $\psi (a)-1$ is not black on $S,$
for $(a, \psi (a)-1)$ is a column of the shifted scheme of white $(k,m)$-regular set $S.$
At the same time if $\psi (s)\leq a<\psi (t)-1,$
then $s>\psi (a)-1>t.$ In particular $\psi (a)-1$ appears on the scheme $S,$
and it is a white point on $S$. Further, $\psi (a)-1$ is an intermediate point 
on the minimal fragment (\ref{bal}), hence it is black on the scheme $T.$
Therefore $\psi (\psi (a)-1)-1=a$ is a white point on $T^*.$ Since $a\neq \psi (t)-1$ yet,
it is a white point on $T_{\psi (s)}^{*\psi (t)-1}$ as well. Thus the $S_k^{t}T_{\psi (s)}^{*\psi (t)-1}$-scheme
has no  intermediate complete black-black columns.

Consider the last column, $a=t.$ Since $T$ is black $(i,j)$-regular,
and $(t,\psi (t)-1)$ is a column of the shifted $T$-scheme, the point $\psi (t)-1$
is not white on $T.$ Therefore $t=\psi (\psi (t)-1)-1$ is not black on $T^*.$
It is neither black on $T_{\psi (s)}^{*\psi (t)-1},$ for $t=\psi (t)-1$ implies $t=n,$ while now $t<n.$

Let  $b$ be a label of the first from the left complete column 
of the $S_k^{t}T_{\psi (s)}^{*\psi (t)-1}$-scheme, $b=\max \{ k-1,\psi (s)-1\} .$
In this case $k-1\neq \psi (s)-1,$ for $\psi (k)$ is not black on $S,$ see (\ref{grab1}).
In particular the first from the left  column is incomplete.

If $k<\psi (s),$ $b=\psi (s)-1,$ then $(b,s)$ is a column of the shifted $S$-scheme.
Hence $b$ is white on $S.$ It is still white on $S_k^t,$ for 
$\psi (s)-1$ is not white on $T$ in particular  $b\neq t.$
Thus, the first from the left complete column on the 
$S_k^{t}T_{\psi (s)}^{*\psi (t)-1}$-scheme is white-white one.

If $k>\psi (s),$ $b=k-1,$ then due to (\ref{grab2}) the point $\psi (k)$
is white on $S.$ We have $t<n\leq \psi (k)<s;$ that is, $\psi (k)$
is an intermediate point   of the fragment (\ref{bal}). Hence $\psi (k)$
is black on $T,$ while $k-1=\psi (\psi (k))-1$ is white on $T^*.$
Thus, the first from the left complete column on the 
$S_k^{t}T_{\psi (s)}^{*\psi (t)-1}$-scheme is still white-white one.
This proves that $S_k^{t}T_{\psi (s)}^{*\psi (t)-1}$-scheme is strongly white,
and one may apply Theorem \ref{str} to see that (\ref{bit3}) is valid. 

While considering the case ``$s\leq n$", we have seen how relations 
(\ref{bit}--\ref{bit3}) with $\mu _k^{s,t}\neq 1$ imply $t=k-1.$ Here
$\mu _k^{s,t}\neq 1$ according to (\ref{mu2}--\ref{mu4}), for $t\neq n,$ 
and $s,$ being a black point on $S,$
is not equal to $\psi (k).$ Thus $t=k-1.$

Consider the  $T^*S^*$-scheme that corresponds to the mirror generators.
This scheme  contains a fragment (\ref{bal}) with $t\leftarrow \psi (s)-1,$
$s\leftarrow \psi (t)-1.$ In this case $T^*$ is white $(\psi (j),\psi (j))$-regular.
Therefore we may apply already proved ``$t=k-1$" to that situation. 
We get  $\psi (s)-1=\psi (j)-1;$ that is, $j=s.$

Further,  relations (\ref{bit4}) and (\ref{bit5}) are valid. 
While considering the case ``$s\leq n$", we have seen
 that if $t=k-1,$ then the $S_{1+s}^mT_{\psi (s)}^{*\psi (k)}$-scheme is strongly white even if $s>n.$
Hence Theorem \ref{str} implies (\ref{bit6}).  
At  the same time we know that relations  (\ref{bit4}--\ref{bit6})  imply $s=m.$ 

Applying this result to the $T^*S^*$-scheme that corresponds to the mirror generators
we have $\psi (k)=\psi (i);$ that is $k=i=t-1,$ $m=j=s.$
Thus, $ST$-scheme has the form (\ref{gra3}).
This contradiction completes the proof.

\end{document}